\documentclass[12pt,twoside]{article}
\usepackage{amsmath}
\usepackage{latexsym,amssymb,amsthm,lscape,epsfig,setspace,tikz,slashed,mathtools,enumerate}
\usepackage[all]{xy}
\usepackage[retainorgcmds]{IEEEtrantools}
\usetikzlibrary{decorations.markings}
 \usetikzlibrary{arrows}
\usepackage[pdftex]{hyperref}
\usepackage[margin=10pt,font=small,labelfont=bf]{caption}
\usepackage{mathrsfs}
\usepackage{tikz}   
\usepackage{tikz-cd}

\usepackage{soul}

\usepackage{combelow}

\theoremstyle{plain}
\newtheorem{theorem}{Theorem}[section]
\newtheorem{lemma}[theorem]{Lemma}
\newtheorem{corollary}[theorem]{Corollary}
\newtheorem{remark}[theorem]{Remark}

\newtheoremstyle{namedthm}{}{}{\itshape}{}{\bfseries}{\!\!.}{0.5em}{\thmnote{#3 }}
\theoremstyle{namedthm}

\theoremstyle{definition}
\newtheorem{definition}[theorem]{Definition}

\newtheorem{example}[theorem]{Example}

{%
\begin{enumerate}}%
{\end{enumerate}%
}%

\makeatletter
\def\Ddots{\mathinner{\mkern1mu\raise\p@
\vbox{\kern7\p@\hbox{.}}\mkern2mu
\raise4\p@\hbox{.}\mkern2mu\raise7\p@\hbox{.}\mkern1mu}}
\makeatother

\newcommand{\Ato}{\Rightarrow}
\renewcommand{\gg}{\mathfrak{g}}
\newcommand{\hh}{\mathfrak{h}}
\newcommand{\diffto}{\xrightarrow{\raisebox{-0.2 em}[0pt][0pt]{\smash{\ensuremath{\sim}}}}}

\date{} \usepackage{color} \definecolor{tocolor}{rgb}{.1,.1,.5}
\definecolor{urlcolor}{rgb}{.2,.2,.6}
\definecolor{linkcolor}{rgb}{.1,.1,.6}
\definecolor{citecolor}{rgb}{.6,.2,.1}
\hypersetup{colorlinks=true, urlcolor=urlcolor,
  linkcolor=linkcolor, citecolor=citecolor}
\definecolor{darkgreen}{rgb}{0.0, 0.5, 0.0}

\begingroup
\hypersetup{linkcolor=blue}
\endgroup

\renewcommand{\rm}{\normalshape}
\numberwithin{equation}{section}

\begin{document}

\title{The Serre spectral sequence of a Lie subalgebroid}

\author{Ioan M\u{a}rcu\cb{t}
  \and
 Andreas Sch\"{u}\ss ler
}

\newcommand{\Addresses}{{
  \bigskip
  \footnotesize

\noindent Ioan M\u{a}rcu\cb{t},\par\nopagebreak
\noindent \textsc{Mathematisches Institut, Universit\"{a}t zu K\"{o}ln,\\
Weyertal 86-90,
D-50931 K\"{o}ln, Germany}\par\nopagebreak
\noindent \textit{E-mail address}: \texttt{imarcut@uni-koel.de}

\medskip

 \noindent Andreas~Sch\"{u}\ss ler,\par\nopagebreak
\noindent  \textsc{KU Leuven, Department of Mathematics,\\
Celestijnenlaan 200B box 2400,  BE-3001 Leuven, Belgium}\\
 \textit{E-mail address:} \texttt{a.schuessler@math.ru.nl}
}}

\maketitle

\abstract{We study a spectral sequence approximating Lie algebroid cohomology associated to a Lie subalgebroid. 
This is a simultaneous generalisation of several classical constructions in differential geometry, including the Leray–Serre spectral sequence for de Rham cohomology associated to a fibration \cite{Ser51}, the Hochschild-Serre spectral sequence for Lie algebras \cite{HoSe53}, and the Mackenzie spectral sequence for Lie algebroid extensions \cite{Mack05}. We show that, for wide Lie subalgebroids, the spectral sequence converges to the Lie algebroid cohomology, and that, for Lie subalgebroids over proper submanifolds, the spectral sequence converges to the formal Lie algebroid cohomology. We discuss applications and recover several constructions in Poisson geometry in which this spectral sequence has appeared naturally in the literature.
}
\vskip12pt

\setcounter{tocdepth}{2}
\tableofcontents


\section{Introduction}

Lie algebroids generalise simultaneously various geometric structures, including Lie algebras, manifolds (via their tangent bundles), foliations, Lie algebra actions, Poisson structures, (generalised) complex structures (in the setting of complex Lie algebroids) etc. 

To a Lie algebroid $A\Ato M$ (with anchor denoted by $\sharp: A\to TM$) one associates the cochain complex of differential forms on $A$,
\[(\Omega^{\bullet}(A),\wedge, d_A).\]
This assignment is a fully faithful contravariant functor from the category of Lie algebroids to that of differential graded commutative algebras. The resulting cohomology groups, $H^{\bullet}(A)$, form the Lie algebroid cohomology of $A$. Given a representation $V\to M$ of $A$, forms on $A$ with values in $V$, form a differential graded $\Omega^{\bullet}(A)$-module, denoted by $(\Omega^{\bullet}(A,V),d_A)$, which yields the Lie algebroid cohomology of $A$ with values in $V$, denoted by $H^{\bullet}(A,V)$. 

Lie algebroid cohomology encodes important geometric information (e.g.\ invariant functions, infinitesimal automorphisms of Poisson and other geometric structures, deformations, etc.). 
However, calculating it can be quite difficult in general. 
One of the issues is that the defining complex is typically non-elliptic and so, even over a compact manifold, the cohomology groups might be infinite dimensional. Another difficulty is to understand and control the behaviour around points where the anchor drops rank. Moreover, there are few computational tools available in general. In this paper, we build such a tool, namely a spectral sequence approximating Lie algebroid cohomology associated to a Lie subalgebroid, which simultaneously generalises several classical constructions. 

The spectral sequence is constructed as follows. Given a Lie subalgebroid $L\Ato N$ of $A\Ato M$, the pullback along the inclusion $i\colon L\hookrightarrow A$ yields a differential graded ideal in $\Omega^\bullet(A)$,
\[ \mathcal{I}:=\ker \big( i^\ast:\Omega^\bullet(A)\to \Omega^\bullet(L)\big).  \]
This ideal gives rise to a descending, differential graded filtration on $\Omega^\bullet(A,V)$ by setting
\[ \mathcal{F}^p\Omega^\bullet(A,V):=\big(\wedge^p\mathcal{I}\wedge\Omega(A,V)\big)^{\bullet}\subset \Omega^{\bullet}(A,V). \]
This filtration induces a spectral sequence $\{ E_r^{p,q}\}_{r\geq 0} $ 
which we call the \textbf{Serre spectral sequence} of the Lie subalgebroid $i:L\hookrightarrow A$. This construction reproduces important spectral sequences known in the literature via suitable choices of the Lie subalgebroid. 

\begin{enumerate}
    \item For a Lie subalgebra $\hh$ of a Lie algebra $\gg$, we recover the classical Hochschild-Serre spectral sequence \cite{HoSe53}.
    \item The spectral sequence of a Lie algebroid extension (\cite{Mack05}, see also \cite{Brah10})
    \[ 0\to L\to A\to B\to 0\]
    is the Serre spectral sequence of $L\hookrightarrow A$ (Section \ref{sec:liealgebroid_extensions}). 
    \item The Leray-Serre spectral sequence \cite{Ser51} for de Rham cohomology of a locally trivial fibre bundle $\pi:M\to Q$ (developed in \cite{Hatt60}) is the Serre spectral sequence of the vertical distribution, i.e.\ of the subbundle $\ker d\pi\subset TM$ (Section \ref{subsection:Leray-Serre}). The construction works similarly for the pullback of a Lie algebroid along a fibration (Section \ref{sec:pullback_liealgebroids}).
    \item The spectral sequence of a regular Poisson manifold \cite{Vai90} is the Serre spectral sequence of the kernel of the anchor map (Section \ref{sec:vaisman}).
\end{enumerate}

Regarding convergence, we obtain the following result in Section \ref{sec:ss_convergence}. 

\begin{theorem}\label{introduction:theorem:convergence}
    Let $L\Ato N$ be a Lie subalgebroid and $V\to M$ a representation of $A\Ato M$.
    \begin{enumerate}
        \item If $N=M$, i.e.\ $L$ is a wide subalgebroid, then the Serre spectral sequence associated to $L$ converges to the cohomology of $A$ with values in $V$.
        \item If $N\subset M$ is a closed and embedded submanifold of positive codimension, then the Serre spectral sequence converges to the formal cohomology of $A$ around $N$ with values in $V$. 
    \end{enumerate}
\end{theorem}

Using the notion of Lie algebroid representations up to homotopy introduced in \cite{AbCr12}, we can describe the first page of the Serre spectral sequence in the settings of Theorem \ref{introduction:theorem:convergence}.

\begin{theorem}\label{introduction:theorem:submanifold} Let $L\Ato N$ be a Lie subalgebroid and $V\to M$ a representation of $A\Ato M$, where $N\subset M$ is a closed, embedded submanifold. The conormal bundle of $L$ in $A$ is canonically a VB-algebroid $\nu_L(A)^*\Ato L^{\circ}$, which corresponds to a representation up to homotopy of $L$ on $\nu_N(M)^*\oplus L^{\circ}$.    
      The first page of the Serre spectral sequence of $L$ is isomorphic to the cohomology of $L$ with coefficients in representations up to homotopy, as follows
\begin{equation}\label{introduction:eq:E_1}
            E_1^{p,q}\simeq\mathrm{H}(L,\wedge^p(\nu_N(M)^*\oplus L^{\circ})\otimes V|_N)^q.
    \end{equation}
These representations up to homotopy are classical representations in the following cases. 
    \begin{enumerate}    \item\label{introduction:theorem:E_1:item:wide} If $N=M$, i.e.\ $L$ is a wide Lie subalgebroid, then \eqref{introduction:eq:E_1} reduces to
    \[E_1^{p,q}\simeq \mathrm{H}^{q}(L,\wedge^pL^{\circ}\otimes V),\]
    where the representation of $L$ on $L^{\circ}\simeq (A/L)^*$ is the dual Bott connection.
\item\label{introduction:theorem:E_1:item:invariant} If $N$ is an invariant submanifold of $A$ and $L=A|_N$, then \eqref{introduction:eq:E_1} reduces to 
    \[ E_1^{p,q}\simeq\mathrm{H}^{p+q}(A|_N,S^p\nu_N(M)^\ast\otimes V|_N).\]
    \end{enumerate}
\end{theorem}

From Section \ref{sec:liealgebroid_extensions} on, we will consider subalgebroids which fit into a short exact sequence 
\begin{equation}\label{introduction:eq:LA_extensions}
\begin{tikzpicture}[baseline=(current bounding box.center)]
\usetikzlibrary{arrows}
\node (1) at (0,1) {$ 0 $};
\node (2) at (1.5,1) {$ L $};
\node (3) at (3,1) {$ A $};
\node (4) at (4.5,1) {$ B $};
\node (5) at (6,1) {$ 0 $};

\node (w) at (1.5,0) {$ M $};
\node (e) at (3,0) {$ M $};
\node (r) at (4.5,0) {$ Q $};
\draw[-Implies,double equal sign distance]
(2) -- (w);
\draw[-Implies,double equal sign distance]
(3) -- (e);
\draw[-Implies,double equal sign distance]
(4) -- (r);
\path[->]
(1) edge node[]{$  $} (2)
(2) edge node[above]{$ $} (3)
(3) edge node[above]{$ $} (4)
(4) edge node[]{$  $} (5)
(w) edge node[above]{$\mathrm{id}_M $} (e)
(e) edge node[above]{$ \pi $} (r);
\end{tikzpicture}
\end{equation}
where the base map $\pi: M\to Q $ is a surjective submersion. In this case we can describe the differential on $E_1$ more explicitly, and make rigorous the interpretation of the second page given in \cite{Brah10}, where the Serre spectral sequence of general extensions \eqref{introduction:eq:LA_extensions} was first considered.

\begin{theorem}
The $C^\infty(Q)$-module $\mathrm{H}^\bullet(L,V)$
is a generalised $B$-representation (defined in Section \ref{sec:basics}), and the Serre spectral sequence associated to $L$ satisfies
\[E_2^{p,q}\simeq \mathrm{H}^p(B,\mathrm{H}^q(L,V)).\]
\end{theorem}

In some cases, the generalised representation of $B$ is actually a classical one. For example, if $L$ is abelian and $\pi=\mathrm{id}_M$ the result was already obtained in \cite{Mack05}. If $A=\pi^! B$ is the pullback Lie algebroid along a locally trivial fibration $\pi:M\to Q$ and $L=\ker d\pi$, we obtain the following generalisation of the Leray-Serre spectral sequence for de Rham cohomology.
\begin{theorem}\label{introduction:theorem:fibre_bundles}
     Let $B\Ato Q$ be a Lie algebroid over the base of a locally trivial fibre bundle $\pi: M\to Q$ with typical fibre $F$. If $\mathrm{H}^\bullet(F)$ is finite dimensional then the associated spectral sequence computing $\mathrm{H}^\bullet(\pi^! B)$ satisfies
     \[ E_2^{p,q}\simeq\mathrm{H}^p(B,\mathcal{H}^q(\ker d\pi)), \]
     where the representation of $B$ on $\mathcal{H}^q(\ker d\pi)$ is the pullback via the anchor $\sharp:B\to TN$ of the Gauss-Manin connection, i.e.\ $\nabla_b=\nabla_{\sharp b}^{GM}$. Here, $\mathcal{H}^q(\ker d\pi)\to Q$ is the finite dimensional vector bundle with fibres $\mathcal{H}^q(\ker d\pi)_x=\mathrm{H}^q(\pi^{-1}(x))$.
\end{theorem}

However, interpreting the generalised representation as a classical one is not always possible. A class of such examples are given by submersions by Lie algebroids, discussed in Section \ref{sec:submersions_by_LAs}. Submersions by Lie algebroids, introduced and studied in \cite{Frej19}, consist of a Lie algebroid $A\Ato M$ and a surjective submersion $\pi:M\to Q$, such that $d\pi\circ \sharp:A\to TQ$ is pointwise surjective. A spectral sequence for the cohomology of a submersion by Lie algebroids was developed in \cite{Frej19}.
We compare this to the Serre spectral sequence of the Lie subalgebroid $L:=\ker (d\pi\circ \sharp)$. Moreover, we prove the following result.

\begin{theorem}\label{introduction:theorem:submersion_by_LA}
    Let $(A\Ato M,\pi:M\to Q)$ be a locally trivial submersion by Lie algebroids and $V\to M$ a representation of $A$. The spaces on the second page of the Serre spectral sequence of the Lie subalgebroid $L=\ker d\pi\circ \sharp$ are isomorphic to the sheaf cohomology of $Q$,
    \[ E_2^{p,q}\simeq\mathrm{H}^p(Q,\mathcal{S}^q_{L,V}), \]
    where $\mathcal{S}^q_{L,V}$ is a locally constant sheaf, and it sends an open subset $U\subset Q$ to 
    \[ \mathcal{S}^q_{L,V}(U)=\big\{ c\in \mathrm{H}^q(L|_{\pi^{-1}(U)},V|_{\pi^{-1}(U)})\, |\,  \nabla c=0  \big\}, \]
    where $\nabla$ is a generalised representation of $TQ$ on the $C^{\infty}(Q)$-module $\mathrm{H}^{q}(L,V)$.
\end{theorem}

In Section \ref{sec:coupling_poisson} and \ref{sec:normal_forms_presympl_leaves} we apply Theorem \ref{introduction:theorem:submersion_by_LA} to horizontally nondegenerate Dirac structures (see \cite{Wade08}, and \cite{Vor01,Vor05} in the Poisson case). 
We obtain descriptions of the cohomology of such Dirac structures in low degrees; in particular, we reproduce and generalise results for horizontally nondegenerate Poisson structures obtained in \cite{VBVo18}.

Finally, in Section \ref{sec:abelian_extensions} we consider extensions \eqref{introduction:eq:LA_extensions} for which the Lie subalgebroid $L$ is abelian. Then we can describe the differential on $E_2$ by means of the extension class \[[\gamma]\in \mathrm{H}^2(B,L).\] Note that, for $\pi\neq \mathrm{id}_M$, we again make use of the notion of a generalised representation. Namely, we obtain the following generalisation of \cite[Theorem 7.4.11]{Mack05} (for $B=TQ$), \cite[Theorem 8]{HoSe53} (for Lie algebras), and 
\cite[Corollary 4.3]{MaZe21} (for the anchor of a regular Lie algebroid).

 \begin{theorem}
 	Let $L\Ato M$ in \eqref{introduction:eq:LA_extensions} be abelian, and let $V$ be a representation of $A$ on which $L$ acts trivially. The second page of the Serre spectral sequence can be identified as
 	\[ 
 	\big(d_2: E_2^{p,q}\to E_2^{p+2,q-1}\big) \simeq \big( (-1)^p\mathrm{i}_{[\gamma]}: \mathrm{H}^{p}(B,\Omega^q(L,V))\to \mathrm{H}^{p+2}(B,\Omega^{q-1}(L,V))\big). \]
 \end{theorem}

\section{Basics on Lie algebroids}\label{sec:basics}
To fix notations and conventions, we recall in this section the basic notions from the theory of Lie algebroids. For more on the general theory, see e.g.\ \cite{Mack05,CrFe11,Mein17,CFM21}.

A \textbf{Lie algebroid}, denoted schematically by $A\Ato M$, consists of
\begin{itemize}
\item a smooth vector bundle over a smooth manifold, $A\to M$,
\item a vector bundle map covering the identity map
$\sharp: A\to TM$, called the anchor map,
\item a Lie algebra structure on $\Gamma(A)$ (:= the space of smooth sections of $A$),
\[[\cdot,\cdot]:\Gamma(A)\times\Gamma(A)\to \Gamma(A),\]
\end{itemize}
which satisfies the following compatibility condition, called the Leibniz identity,
\[[\alpha, f\beta]=f[\alpha,\beta]+(\mathscr{L}_{\sharp \alpha}f) \beta,\]
for all $\alpha,\beta\in \Gamma(A)$ and $f\in C^{\infty}(M)$. Here $\mathscr{L}_Xf$ denotes the Lie derivative of a function $f$ along the vector field $X$.

A \textbf{representation} of a Lie algebroid $A\Ato M$ on a vector bundle $V\to M$ is defined as a flat $A$-connection on $V$, i.e.\ a bilinear operator, 
\[\nabla:\Gamma(A)\times \Gamma(V)\to \Gamma(V),\quad (\alpha,\eta)\mapsto \nabla_\alpha\eta,\]
which is $C^{\infty}(M)$-linear in the first entry and satisfies the Leibniz identity in the second entry
\begin{equation}\label{eq:rep:1}
 \nabla_{f\alpha} \eta=f\, \nabla_\alpha \eta,\qquad \nabla_\alpha (f\eta)=f\, \nabla_\alpha \eta+(\mathscr{L}_{\sharp \alpha}f)\eta,
 \end{equation}
 and has trivial curvature
\begin{equation}\label{eq:rep:2}
 \nabla_\alpha \nabla_\beta \eta-
 \nabla_\beta \nabla_\alpha\eta=\nabla_{[\alpha,\beta]}\eta,
 \end{equation}
where $f\in C^{\infty}(M)$, $\alpha,\beta\in \Gamma(A)$ and $\eta\in \Gamma(V)$.

The notion of a Lie algebroid unifies several concepts in differential geometry. A Lie algebra $\gg$ is the same as a Lie algebroid over a point, $A\Ato \{*\}$, and the corresponding notions of representations coincide; the tangent bundle $TM$ of a manifold $M$ is a Lie algebroid for the usual bracket of vector fields, and a representation of $TM\Ato M$ is the same as a flat vector bundle $V\to M$; by the Frobenius Theorem, a foliation on a manifold $M$ is the same as a Lie subalgebroid $D\Ato M$ (in the sense below) of $TM\Ato M$; any Poisson bracket $\{\cdot,\cdot\}$ on $C^{\infty}(M)$ yields a Lie algebroid structure on $T^*M\Ato M$; an infinitesimal action of a Lie algebra $\gg$ on a manifold $M$ yields a so-called action Lie algebroid, $\gg\ltimes M\Ato M$.

\vspace{0.1cm}

We recall the construction of \textbf{Lie algebroid cohomology}. To any Lie algebroid $A\Ato M$, one associates the differential graded commutative algebra of de Rham forms on $A$,
\begin{equation}\label{eq:dgca}
\Big(\Omega^{\bullet}(A):=\bigoplus_{ k=0}^{r}\Omega^{k}(A),\wedge, d_A\Big), \quad \Omega^{k}(A):=\Gamma(\wedge^kA^*),
\end{equation}
where $r=\mathrm{rank}(A)$ and $d_A$ is defined similar to the usual de Rham exterior derivative 
\begin{align}\label{eq:Koszul}
(d_A\omega)(\alpha_0,\ldots,\alpha_k)=&\sum_{i}(-1)^i\mathscr{L}_{\sharp\alpha_i} \omega(\alpha_0,\ldots,\hat{\alpha}_i,\ldots, \alpha_k)\\
+&\sum_{i<j}(-1)^{i+j}\omega([\alpha_i,\alpha_j],\alpha_0,\ldots,\hat{\alpha}_i,\ldots,\hat{\alpha}_j,\ldots,\alpha_k),\nonumber
\end{align}
where we regard elements in $\Omega^{\bullet}(A)$ as $C^{\infty}(M)$-multilinear forms on sections $\alpha_i\in \Gamma(A)$ with values in $C^{\infty}(M)$. By passing to cohomology one obtains the Lie algebroid cohomology of $A$, denoted by $\mathrm{H}^{\bullet}(A)$. More generally, given a representation $V\to M$ of $A$ we have an induced differential on $V$-valued forms on $A$, which we continue denoting by $d_A$
\begin{equation}\label{eq:V-val-forms}
\big(\Omega^{\bullet}(A,V):=\Gamma(\wedge^{\bullet}A^*\otimes V),d_A\big),
\end{equation}
defined by formula \eqref{eq:Koszul} with the operator $\nabla_{\alpha_i}$ instead of $\mathscr{L}_{\sharp\alpha_i}$ (the latter corresponding to the canonical representation on $\mathbb{R}\times M$). Passing to cohomology, one obtains the Lie algebroid cohomology of $A\Ato M$ with coefficients in the representation $V\to M$, denoted by $\mathrm{H}^{\bullet}(A,V)$.

Lie algebroid cohomology recovers Chevalley-Eilenberg cohomology of Lie algebras, de Rham cohomology, foliated cohomology, Poisson cohomology, etc.

It is convenient to introduce the following notion. A \textbf{generalised representation} of a Lie algebroid $A\Ato M$ is a $C^{\infty}(M)$-module $\mathfrak{M}$, endowed with a bilinear map 
\[\nabla:\Gamma(A)\times \mathfrak{M}\to \mathfrak{M}\]
satisfying the axioms \eqref{eq:rep:1} and \eqref{eq:rep:2}. Usual representations correspond to the case when $\mathfrak{M}$ is a finitely generated projective $C^{\infty}(M)$-module.

To introduce $\mathfrak{M}$-valued cohomology, define the set of $p$-forms on $A$ with values in $\mathfrak{M}$ as
\begin{equation}
\label{eq:forms:in:module}
\Omega^p(A,\mathfrak{M}):=\big\{
\omega:
\Gamma\underbrace{(A)\times \ldots \times \Gamma(A)}_{p-\textrm{times}}\to \mathfrak{M}\, |\, \textrm{alternating and $C^{\infty}(M)$-multilinear}\big\}.
\end{equation}
If follows from Lemma \ref{lemma:tensor_vb} \eqref{App_2} that we have a canonical isomorphism of $C^{\infty}(M)$-modules
\begin{equation}
\label{eq:iso:forms:tensor}
\Omega^p(A,\mathfrak{M})\simeq \Omega^p(A)\otimes_{C^{\infty}(M)} \mathfrak{M}.
\end{equation}
The formula \eqref{eq:Koszul} generalises to this setting, and yields a differential 
\[d_{A}:\Omega^{\bullet}(A,\mathfrak{M})\to \Omega^{\bullet+1}(A,\mathfrak{M}).\]
The associated cohomology groups will be denoted by 
\[\mathrm{H}^{\bullet}(A,\mathfrak{M}).\]
For a classical representation $V$, the two equivalent notations $\mathrm{H}^{\bullet}(A,\Gamma(V))=\mathrm{H}^{\bullet}(A,V)$ will, hopefully, not lead to confusions.

General \textbf{Lie algebroid morphisms} are introduced efficiently using the Lie algebroid complex. Let $A\Ato M$ and $L\Ato N$ be two Lie algebroids. A vector bundle map $\Phi:L\to A$ covering $\phi:N\to M$ induces a pullback homomorphism between the graded commutative algebras of de Rham forms
\[\Phi^*:(\Omega^{\bullet}(A),\wedge) \to (\Omega^{\bullet}(L),\wedge),\quad 
(\Phi^*\omega)_{p}(\alpha_1,\ldots,\alpha_k):=\omega_{\phi(p)}(\Phi(\alpha_1),\ldots, \Phi(\alpha_k)),
\]
for all $p\in N$ and $\alpha_i\in L_{p}$. By definition, $\Phi$ is a Lie algebroid morphism if and only if $\Phi^*$ is a cochain map, i.e.\ for all $\omega\in \Omega^{\bullet}(A)$, we have 
\[\Phi^*(d_A\omega)=d_L(\Phi^*\omega).\]

For example, for Lie algebras $\gg$ and $\hh$, this notion recovers Lie algebra maps $\Phi:\gg\to \hh$; Lie algebroid morphisms $\Phi:TM\to TN$ are determined by their base map, via $\Phi=d\phi$; Lie algebroid morphisms $\Phi:TM\to \gg$ are the same as flat, principal connections on the trivial principal $G$-bundle $G\times M$, where $G$ is a connected Lie group integrating $\gg$.

Representations can be pulled back along a Lie algebroid morphism $\Phi:L\to A$. If $V\to M$ is a representation of $A\Ato M$ then 
 the pullback representation of $L\Ato N$ is on the pullback vector bundle $\phi^*V\to N$ and is uniquely determined by the condition that the pullback map is a cochain map
 \begin{equation*}
     \Phi^*:\big(\Omega^{\bullet}(A,V),d_{A}\big)\to  
 \big(\Omega^{\bullet}(L,\phi^*V),d_{L}\big).
 \end{equation*}

A \textbf{Lie subalgebroid} of $A\Ato M$ is a vector subbundle $L\hookrightarrow A$ over an injective immersion $N\hookrightarrow M$ endowed with a Lie algebroid structure $L\Ato N$ (necessarily unique) for which the inclusion $i: L\hookrightarrow A$ is a Lie algebroid morphism. 

The \textbf{pullback} of a Lie algebroid $A\Ato M$ along a smooth map $f:N\to M$ is given by
\begin{equation}\label{eq:pullback_def}
    f^!A:=\{(v,a)\in T_xN\oplus A_{f(x)}\, |\, df_x(v)=\sharp a\in T_{f(x)}M, \ x\in N\},
\end{equation}
which carries a canonical Lie algebroid structure over $N$, provided it has constant rank. If this is the case, its anchor is the first projection $\sharp=\mathrm{pr}_1: f^!A\to TN$ and the second projection is a Lie algebroid map $f^{!}:=\mathrm{pr}_2: f^!A\to A$ covering $f$.

\section{The Serre spectral sequence of a Lie subalgebroid}

In this section we introduce the Serre spectral sequence of a Lie subalgebroid. After discussing the general construction and convergence properties of the spectral sequence, we describe the zeroth page and the spaces on the first page first for Lie subalgebroids over the whole base manifold and then for Lie subalgebroids over closed and embedded submanifolds.

\subsection{The spectral sequence}\label{sec:ss_convergence}

We will recall the basic construction of a spectral sequence associated to a filtered complex (see for example \cite[Chapter 5.4]{Weib94}, \cite[Chapter 1.1]{McC00} or \cite[Chapter 7.4]{Mack05}). 

Consider a differential graded \textbf{filtration} on the cochain complex $(\Omega^\bullet(A,V), d_A)$, 
\[\ldots \subset \mathcal{F}^{p}\Omega^{\bullet}(A,V)\subset \mathcal{F}^{p-1}\Omega^{\bullet}(A,V)\subset \ldots \subset \mathcal{F}^{0}\Omega^{\bullet}(A,V)= \Omega^{\bullet}(A,V).\]
The \textbf{spectral sequence} associated to this filtration is defined as follows. For $r\geq 0$, let 
\[Z_{r}^{p,q}:=\{\omega\in\mathcal{F}^{p}\Omega^{p+q}(A,V)\, |\, d_A\omega\in \mathcal{F}^{p+r}\Omega^{p+q+1}(A,V)\},\]
where, for $p\leq  
 0$, $\mathcal{F}^p\Omega^{\bullet}(A,V):=\Omega^{\bullet}(A,V)$. Since $Z_{r-1}^{p+1,q-1}\subset Z_{r}^{p,q}$, we can define the quotients
\[E^{p,q}_r:=\frac{Z_{r}^{p,q}}{Z_{r-1}^{p+1,q-1}+d_A\, Z_{r-1}^{p-r+1,q+r-2}}
\]
 where $Z_{-1}^{p,q}:=\mathcal{F}^{p}\Omega^{p+q}(A,V)$. Since $d_A\, Z_{r}^{p,q}\subset Z_{r}^{p+r,q-r+1}$, we have induced differentials
\[d_r:E_r^{p,q}\to E_r^{p+r,q-r+1}.\]
The $r$-th page of the spectral sequence is the total complex
\[\Big(E_r^{\bullet}=\bigoplus_{p+q=\bullet}E_r^{p,q},d_r\Big).\]
The $r+1$-th page is canonically isomorphic to the cohomology of the $r$-th page: 
\[E_{r+1}^{p,q}\simeq \mathrm{H}^{p,q}(E_r^{\bullet},d_r).\]

A filtration as above can be constructed from a \textbf{differential graded ideal} 
\[\mathcal{I}^{\bullet}\subset  \Omega^{\bullet}(A),\]
by using the powers of $\mathcal{I}$, as follows
\[\mathcal{F}^{p}_{\mathcal{I}}\Omega^{\bullet}(A,V):=\big(\wedge^p\mathcal{I}\wedge\Omega(A,V)\big)^{\bullet}\subset \Omega^{\bullet}(A,V).\]

A Lie subalgebroid $i: L\hookrightarrow A$ yields a differential graded ideal
\[\mathcal{I}^{\bullet}_{L}:=\big(\ker i^*:\Omega^{\bullet}(A)\to \Omega^{\bullet}(L)\big)\subset \Omega^{\bullet}(A),\]
and so, a filtration
\[\mathcal{F}^{p}_{L}\Omega^{\bullet}(A,V):=\mathcal{F}^{p}_{\mathcal{I}_{L}}\Omega^{\bullet}(A,V)\subset \Omega^{\bullet}(A,V).\]
The corresponding spectral sequence will be called \textbf{Serre spectral sequence} associated to the Lie subalgebroid $i: L\hookrightarrow A$. If $N\subset M$ is the base of $L$, note that $\mathcal{I}_L^0=\mathcal{I}_N$ is given by the vanishing ideal of $N$.

\begin{remark}\rm
    Likewise, one can use an arbitrary Lie algebroid morphism $\Phi: L\to A$ to obtain a differential graded ideal and so a filtration
        \[\mathcal{I}^{\bullet}_{\Phi}:=(\ker \Phi^*)^{\bullet}\subset \Omega^{\bullet}(A), \qquad \mathcal{F}^{p}_{\mathcal{I}_{\Phi}}\Omega^{\bullet}(A,V)\subset \Omega^{\bullet}(A,V),\]
    giving rise to a spectral sequence. While Theorem \ref{theorem:convergence} can still be formulated and proven in this more general setting, in this paper we exclusively discuss the Serre spectral sequence arising from Lie subalgebroids.
\end{remark}

We have the following result regarding convergence.

\begin{theorem}\label{theorem:convergence}
Let $i:L\hookrightarrow A$ be a Lie subalgebroid of $A\Ato M$ with base $N\subset M$ and $V\neq 0_M$ a representation of $A$. The following are equivalent:
\begin{enumerate}[(a)]
\item $N$ is dense in $M$;
\item $\mathcal{I}_{L}^0=0$;
\item The filtration $\big\{\mathcal{F}^{p}_{L}\Omega^{\bullet}(A,V)\big\}_{p\geq 0}$ is finite;
\item The filtration $\big\{\mathcal{F}^{p}_{L}\Omega^{\bullet}(A,V)\big\}_{p\geq 0}$ is Hausdorff.
\end{enumerate}
If either condition holds, then 
\[\mathcal{F}^{p}_{L}\Omega^{\bullet}(A,V)=0,\quad \textrm{for} \quad p>\bullet,\]
and in particular the Serre spectral sequence stabilises at the page $r+1$, where $r=\mathrm{rank}(A)$, and therefore it converges
\[\mathrm{H}^{\bullet}(A,V)\simeq \bigoplus_{p+q=\bullet}E^{p,q}_{r+1}.\]
\end{theorem} 

\begin{proof}
Condition $(a)$ is equivalent to the restriction $C^{\infty}(M)\to C^{\infty}(N)$ being injective, which is equivalent to $(b)$. 

If $(b)$ holds, then 
\[\mathcal{F}_{L}^{p}\Omega(A,V)=\wedge^p\mathcal{I}_{L}\wedge \Omega(A,V)\subset \bigoplus_{k=p}^{r}\Omega^{k}(A,V),\qquad \textrm{hence} \qquad 
\mathcal{F}_{L}^{r+1}\Omega(A,V)=0.\]
So we obtain $(c)$ and that the spectral sequence stabilises at the page $r+1$. 

Clearly, $(c)$ implies $(d)$. 

Assume now that $N$ is not dense in $M$. By the standard construction of bump functions, we find a non-zero function $f\in C^{\infty}(M)$, with support in $M\setminus N$, such that $f\geq 0$, and $f^{\frac{1}{p}}\in C^{\infty}(M)$, for all $p\geq 1$. Hence $f\in (\mathcal{I}_{L}^0)^p$, for all $p\geq 0$. Choose $\eta\in \Gamma(V)$ such that $f\eta\neq 0$. It follows that
\[f\eta\in \bigcap_{p\geq 0}\mathcal{F}^p_{L}\Omega^0(A,V),\]
thus the filtration is not Hausdorff. This shows that $(d)$ implies $(a)$.
\end{proof}

Even if $N\subset M$ is not dense, one can make use of the spectral sequence. If $N\subset M$ is a closed and embedded submanifold of positive codimension, by Theorem \ref{theorem:convergence} the induced filtration on $\Omega^\bullet(A)$ is neither Hausdorff nor finite. 
Instead, the spectral sequence converges to formal cohomology along $N$ of forms on $A$ with values in $V$.

\begin{theorem}\label{theorem:convergence_submanifold}
    Let $L\Ato N$ be a Lie subalgebroid of $A\Ato M$ over a closed, embedded submanifold $N\subset M$ and fix a representation $V\to M$ of $A$.
    The Serre spectral sequence associated to $L\Ato N$ converges to the formal cohomology of $A$ around $N$ with values in $V$.
\end{theorem}

To define formal cohomology we first recall the notion of jets of sections of a vector bundle. Let a vector bundle $E\to M$ be given and $N\subset M$ a closed, embedded submanifold with vanishing ideal $\mathcal{I}_N$. We denote the space of $\infty$-jets of sections of $E$ along $N$ by
\[ \mathscr{J}_N^\infty\Gamma(E):= \Gamma(E) / \mathcal{I}_N^\infty \Gamma(E), \]
where we define $\mathcal{I}_N^{\infty}:=\bigcap_{\ell\geq 0} \mathcal{I}_N^{\ell}$. 


For any closed and embedded submanifold $N\subset M$, the $\infty$-jets of forms on $A$ along $N$ inherit a differential. 

\begin{lemma}\label{lemma:subcomplexes_for_submanifolds}
The set $\mathcal{I}_N^\infty\Omega^\bullet(A,V)\subset \Omega^\bullet(A,V)$ is a differential ideal.
\end{lemma}
\begin{proof}
    By \cite[Theorem 1]{Nag73} every function in $\mathcal{I}_N^\infty$ is the product of two functions in $\mathcal{I}_N^\infty$. Then the Leibniz rule of the differential on $\Omega^\bullet(A,V)$ implies the statement.
\end{proof}

Lemma \ref{lemma:subcomplexes_for_submanifolds} allows to define \textbf{formal cohomology} of a Lie algebroid around a submanifold as the cohomology of the quotient complex $\mathscr{J}_N^\infty \Omega^\bullet(A,V)$. Having this notion clarified, we move on to the proof of Theorem \ref{theorem:convergence_submanifold}.

\begin{proof}[Proof of Theorem \ref{theorem:convergence_submanifold}]
    For any Lie subalgebroid $L\Ato N$ of $A\Ato M$, the ideal $\mathcal{I}_L$ is given by
    \[\mathcal{I}_L\cap \Omega^q(A) =\{\omega\in \Omega^q(A)\, | \, i^\ast\omega=0 \} \supset \mathcal{I}_N\Omega^q(A),\]
    which is an equality in degree $q=0$. Moreover, by counting degrees we find that elements in $\mathcal{F}_L^p \Omega^q(A,V)$ can be written as sums of elements of the form
    \[ \omega_{i_1}\wedge\ldots\wedge\omega_{i_p}\wedge \eta, \]
    where $\omega_{i_j}\in \mathcal{I}_L\cap \Omega^{i_j}(A)$, $q \geq i_1+\ldots+i_p=k$, $i_1,\ldots,i_p\geq 0$, and $\eta\in \Omega^{q-k}(A,V)$. In particular, if $p>q$, at least $p-q$ of the indices $i_1,\ldots,i_k$ have to be zero, thus $\mathcal{F}_L^p\Omega^q(A,V)\subset \mathcal{I}_N^{p-q}\Omega^q(A,V)$. Together, for $p>q$ we obtain
    \begin{equation*}
        \mathcal{I}_N^p \Omega^q(A,V)\subset \mathcal{F}_L^p\Omega^q(A,V)\subset \mathcal{I}_N^{p-q}\Omega^q(A,V),
    \end{equation*}
    which implies 
    \[ \bigcap_{p=0}^\infty \mathcal{F}^p_L\Omega^\bullet(A,V)=  \mathcal{I}_N^\infty \Omega^\bullet(A,V). \]
The rest of the proof is a general argument which applies to spectral sequences corresponding to a filtered complex (see e.g.\ \cite[Exercise 5.4.2]{Weib94}).
First, the induced filtration on the quotient complex $\mathscr{J}_N^\infty \Omega^\bullet(A,V)$
    \[ \hat{\mathcal{F}}_L^p \mathscr{J}^\infty_N\Omega^\bullet(A,V):= \mathcal{F}^p_L\Omega^\bullet(A,V) /  \mathcal{I}_N^\infty \Omega^\bullet(A,V) \subset \mathscr{J}_N^\infty \Omega^\bullet(A,V) \]
    is Hausdorff, and the induced spectral sequence converges to formal cohomology around $N$ of forms of $A$ with values in $V$. The quotient map $\Omega^\bullet(A,V)\to \mathscr{J}_N^\infty \Omega^\bullet(A,V)$ preserves the respective filtrations and thus induces a map between spectral sequences. This map is an isomorphism on the zeroth page as
    \[  \hat{E}_0^{p,q}=\frac{\hat{\mathcal{F}}_L^p \mathscr{J}^\infty_N\Omega^{p+q}(A,V)}{\hat{\mathcal{F}}_L^{p+1} \mathscr{J}^\infty_N\Omega^{p+q}(A,V)}\simeq \frac{\mathcal{F}^{p}_L\Omega^{p+q}(A,V)}{\mathcal{F}^{p+1}_L\Omega^{p+q}(A,V)}=E_0^{p,q}. \]
    By the Mapping Lemma (\cite[Lemma 5.2.4]{Weib94}, see also \cite[Theorem 3.5]{McC00}) the two spectral sequences are isomorphic, showing that the Serre spectral sequence converges to the formal Lie algebroid cohomology around $N$ with values in $V$.
\end{proof}

In the rest of the section, we discuss in detail the structure on the zeroth page of the Serre spectral sequences in case of a wide Lie subalgebroid (Section \ref{sec:wide_subalgebroids}, applying Theorem \ref{theorem:convergence}) and a Lie subalgebroid over a closed embedded submanifold (Section \ref{sec:submanifolds_general}, applying Theorem \ref{theorem:convergence_submanifold}).

\subsection{Wide Lie subalgebroids}\label{sec:wide_subalgebroids} 
Throughout this section, we fix a Lie algebroid $A\Ato M$,
a wide Lie subalgebroid $L\subset A$ (i.e.\ a Lie subalgebroid over the same base),
 and a representation $V$ of $A$. By Theorem \ref{theorem:convergence}, the filtration corresponding to $L$ is finite. The filtration can be given a more classical description.

\begin{lemma}\label{theorem:classical_form} If $p>n$, then $\mathcal{F}_L^{p}\Omega^{n}(A,V)=0$, and if $p\leq n$ then
\[\mathcal{F}_L^{p}\Omega^{n}(A,V)=\{\omega\in \Omega^{n}(A,V)\, |\, \omega(\alpha_1,\ldots,\alpha_{n})=0, \textrm{ if } \alpha_1,\ldots, \alpha_{n-p+1}\in \Gamma(L)\}.\]
\end{lemma} 
\begin{proof}
For clarity, choose a vector subbundle $C\subset A$ which is a complement of $L$ in $A$
\[A=L\oplus C.\]
This gives a dual decomposition $A^*=C^{\circ}\oplus L^{\circ}$, where $L^{\circ}$ is the annihilator of $L$ and $C^{\circ}$ is the annihilator of $C$. This induces a decomposition on the level of forms:
\[\Omega^{n}(A,V)=\bigoplus_{k=0}^{n}\Gamma(\wedge^{n-k}C^{\circ}\otimes\wedge^kL^{\circ} \otimes V).\]
Using the similar decomposition for $\Omega^{n}(A)$, one obtains that
\[\mathcal{I}_{L}^{n}=\bigoplus_{k=1}^{n}\Gamma(\wedge^{n-k}C^{\circ}\otimes\wedge^kL^{\circ}).\]
We have that 
\[\wedge^p\mathcal{I}_{L}^{n}=\bigoplus_{k=p}^{n}\Gamma( \wedge^{n-k}C^{\circ}\otimes\wedge^kL^{\circ}).\]
That the left-hand side is included in the right-hand side is obvious, the other inclusion follows by applying repeatedly Lemma \ref{lemma:tensor_vb} \eqref{App_1}. A similar argument shows that
\begin{equation}\label{eq:decomposition}
    \mathcal{F}_L^{p}\Omega^{n}(A,V)=\bigoplus_{k=p}^{n}\Gamma(\wedge^{n-k}C^{\circ}\otimes\wedge^kL^{\circ}\otimes V).
\end{equation}
This is equivalent to the intrinsic description in the statement. 
\end{proof}

We go on to identify the first page of the spectral sequence. For this, recall that $L$ has a canonical representation on the ``normal bundle'' $A/L$, induced by the Lie bracket,
\begin{equation}\label{Bott_connection}
\nabla:\Gamma(L)\times \Gamma(A/L)\to \Gamma(A/L),\quad \nabla_{\beta}(\overline{\alpha}):=\overline{[\beta,\alpha]}, 
\end{equation}
for all $\alpha\in \Gamma(A)$ and $\beta\in \Gamma(L)$. Here $\overline{\delta}\in \Gamma(A/L)$ denotes the image of $\delta\in \Gamma(A)$ under the projection $A\to A/L$. This representation is also called the \textbf{Bott-connection}. This induces the dual representation on $(A/L)^*$, the exterior power representation on $\wedge^p(A/L)^*$, and finally, the tensor product representation of $L$ on $\wedge^p(A/L)^*\otimes V$,
\begin{equation}\label{eq:rep_on_powers}
\nabla:\Gamma(L)\times \Gamma(\wedge^p(A/L)^*\otimes V)\to \Gamma(\wedge^p(A/L)^*\otimes V).
\end{equation}

We have the following.

\begin{theorem}\label{theorem:first_page}
For the first page of the Serre spectral sequence associated to $L$, there is a canonical isomorphism
\[E_1^{p,q}\simeq \mathrm{H}^{q}(L,\wedge^p(A/L)^*\otimes V).\]
More precisely, there exists a canonical isomorphism
\[E_0^{p,q}\simeq \Omega^q(L,\wedge^p(A/L)^{*}\otimes V),\]
under which the differential $d_0$ 
corresponds to the differential of the  representation \eqref{eq:rep_on_powers}.
\end{theorem}

\begin{proof}
The identification follows by using the short exact sequence:
\[0\to \mathcal{F}_L^{p+1}\Omega^{p+q}(A,V)\to \mathcal{F}_L^{p}\Omega^{p+q}(A,V)\stackrel{\mathrm{pr}}{\to} \Omega^{q}(L,\wedge^{p}(A/L)^*\otimes V)\to 0,\]
where the map $\mathrm{pr}$ acts as
\[\mathrm{pr}(\omega)(\alpha_1,\ldots,\alpha_{q}):=\omega(\alpha_1,\ldots,\alpha_{q},\cdot,\ldots,\cdot)\in \Gamma(\wedge^{p}(A/L)^*\otimes V),\]
and we use the canonical isomorphism $(A/L)^*\simeq L^{\circ}$. That the above is indeed a short exact sequence follows immediately from the definitions. 
A direct calculation (see e.g.\ the proof of \cite[Proposition 7.4.3]{Mack05}) implies the statement about the differentials.
 \end{proof}

The description of the filtration from Lemma \ref{theorem:classical_form} shows that the Serre spectral sequence generalises classical constructions.

\begin{example}
Let $\gg$ be a Lie algebra and $V$ a representation of $\gg$. The filtration induced by a Lie subalgebra $\hh\subset \gg$ is given by 
\[\mathcal{F}^p_{\hh}\wedge^{p+q}\gg^*\otimes V=(\wedge^p\hh^{\circ})\wedge(\wedge^{q}\gg^*)\otimes V,\]
and the resulting spectral sequence coincides with the Hochschild-Serre spectral sequence for Lie algebras \cite{HoSe53}. By Theorem \ref{theorem:first_page}, the first page is given by
\[E_1^{p,q}\simeq \mathrm{H}^{q}(\mathfrak{h},\wedge^p(\mathfrak{g}/\mathfrak{h})^*\otimes V).\]
\end{example}

\begin{example}\label{example:Leray-Serre1}
A representation of $TM\Ato M$ is the same as a bundle $V\to M$ endowed with a flat connection. The Lie algebroid cohomology of $TM$ with coefficients in $V$ can be understood as de Rham cohomology of $M$ with \emph{local/twisted coefficients}. A surjective submersion $\pi:M\to Q$ yields a wide subalgebroid $\ker d\pi\subset TM$. The induced filtration on $\Omega^{\bullet}(M,V)$ is given by
\[\mathcal{F}^p_{\pi}\Omega^{p+q}(M,V)=\{\omega\in \Omega^{p+q}(M,V)\, |\, \omega(v_1,\ldots,v_{p+q})=0, \textrm{ if } v_1,\ldots, v_{q+1}\in \ker d\pi\}.\]
The resulting spectral sequence coincide with the Leray–Serre spectral sequence in de Rham cohomology with local coefficients in $V$, which was carefully developed in \cite{Hatt60}. We will come back to this example in Subsection \ref{subsection:Leray-Serre}.
\end{example}

\begin{example}\label{example:foliations}
By Frobenius' Theorem, wide Lie subalgebroids of $TM$ are the same as foliations $\mathcal{F}$ on $M$. Denote the tangent bundle of a foliation $\mathcal{F}$ by $T\mathcal{F}\Ato M$. Forms on the Lie algebroid $T\mathcal{F}$ are called \emph{foliated forms} and will be denoted by $\Omega^{\bullet}(\mathcal{F})$. The representation \eqref{Bott_connection} becomes the Bott connection of $T\mathcal{F}$ on the normal bundle $\nu_{\mathcal{F}}=TM/T\mathcal{F}$. 

The Serre spectral sequence of the inclusion $T\mathcal{F}\subset TM$ converges to the cohomology of $M$, and Theorem \ref{theorem:first_page} shows that its first page is given by 
\[E_1^{p,q}\simeq \mathrm{H}^q(\mathcal{F},\wedge^p\nu_{\mathcal{F}}^*).\]
For later use, we describe the differential on the first page in position on $E_1^{0,q}$, i.e.\
\[d_1: \mathrm{H}^q(\mathcal{F})\to
\mathrm{H}^q(\mathcal{F},\nu_{\mathcal{F}}^*).\]
Let $c\in \mathrm{H}^q(\mathcal{F})$, with representative a foliated $q$-form $\eta\in \Omega^q(\mathcal{F})$, i.e.\ $\eta$ is a smoothly varying family of closed $q$-forms on the leaves of $\mathcal{F}$. Let $\tilde{\eta}\in \Omega^q(M)$ be an extension of $\eta$ to a $q$-form on $M$. Since the pullback of 
$d\tilde{\eta}$ to the leaves of $\mathcal{F}$ vanishes, we obtain an element
\[(d\tilde{\eta})_{1,q}\in \Omega^{q}(\mathcal{F},\nu_{\mathcal{F}}^*),
\qquad (d\tilde{\eta})_{1,q}(X_1,\dots, X_q):=(d\tilde{\eta})(X_1,\ldots,X_q,\cdot)\in \nu_{\mathcal{F}}^*,
\]
where we regard $\nu_{\mathcal{F}}^*$ as the annihilator of $T\mathcal{F}$. Moreover, $(d\tilde{\eta})_{1,q}$ is closed for the complex computing foliated cohomology with values in $\nu_{\mathcal{F}}^*$, and its class is independent of the extension $\tilde{\eta}$ of $\eta$ or even on the chosen representative $\eta$ of $c$. With these, we have that
\[d_1[\eta]=[(d\tilde{\eta})_{1,q}]\in \mathrm{H}^q(\mathcal{F},\nu_{\mathcal{F}}^*).\]
\end{example}

In the generality of Theorem \ref{theorem:first_page}, not much more can be said about the Serre spectral sequence of wide Lie subalgebroids. Starting from Section \ref{sec:liealgebroid_extensions}, we will restrict to the class of Lie subalgebroids that are kernels of surjective morphisms, which will enable us to reveal more information about their associated Serre spectral sequence.

\subsection{Lie subalgebroids over closed submanifolds}\label{sec:submanifolds_general}

In this section, we fix a Lie algebroid $A\Ato M$,
a Lie subalgebroid $L\subset A$ over a closed, embedded submanifold $N\subset M$, and a representation $V$ of $A$. By Theorem \ref{theorem:convergence_submanifold}, we obtain a spectral sequence converging to the formal cohomology of $A$ around $N$ with values in $V$, for which we have the following.

\begin{theorem}\label{theorem:first_page_submanifold}
    There exists an isomorphism
    \[E_1^{p,q}\simeq \mathrm{H}(L,\wedge^p(\nu_N^*\oplus L^{\circ})\otimes V|_N)^q,\]
    where $\nu_N^*\oplus L^{\circ}$ is the representation up to homotopy of $L$ corresponding to the VB-algebroid
      \begin{equation*}
\begin{tikzpicture}[baseline=(current bounding box.center)]
\usetikzlibrary{arrows}
\node (2) at (1.2,1.2) {$ \nu_L(A)^* $};
\node (3) at (3,1.2) {$ L^\circ $};

\node (w) at (1.2,0) {$ L $};
\node (e) at (3,0) {$ N $};
\draw[-Implies,double equal sign distance]
(2) -- (3);
\draw[-Implies,double equal sign distance]
(w) -- (e);
\path[->]
(2) edge node[above]{$ $} (w)
(3) edge node[above]{$ $} (e);
\end{tikzpicture}
\end{equation*}
\end{theorem}

\begin{remark}\rm
We denote the normal bundle of a closed, embedded submanifold $Y\subset X$ by
\[ \nu_Y(X)=TX|_Y / TY, \]
or $\nu_Y$ if the ambient manifold $X$ is clear from context. 
\end{remark}

\begin{remark}\rm
As we will discuss below, $\nu_L(A)^\ast\to L$ is canonically a VB-algebroid, and therefore, with the aid of a splitting, it can be regarded as a two-term representation up to homotopy of $L$ \cite[Theorem 4.11]{GrMe10}.
\end{remark}

\begin{remark}\rm
    Theorem \ref{theorem:first_page_submanifold} can be seen as a generalisation of Theorem \ref{theorem:first_page}, since for $N=M$ the representation up to homotopy of $L$ with values in $\wedge^p L^{\circ}\otimes V$ coincides with the classical representation \eqref{Bott_connection}.
\end{remark}

After recalling some of the necessary tools and objects, we will prove Theorem \ref{theorem:first_page_submanifold}.

First, we recall the notion of a representation up to homotopy \cite[Definition 3.1]{AbCr12}.

\begin{definition}
    A \textbf{representation up to homotopy} of $L\Ato N$ on a $\mathbb{Z}$-graded vector bundle $E^\bullet$ over the same base $N$ is a differential $$D: \Omega(L,E)^\bullet\to \Omega(L,E)^{\bullet+1},$$ where the total degree is defined as
    \[ \Omega(L,E)^q=\bigoplus_{i+j=q}\Omega^i(L,E^j), \]
    satisfying, for any $\omega\in \Omega^k(L)$ and $\eta\in \Omega(L,E)^\bullet$, the Leibniz rule
    \begin{equation}\label{eq:ruth_leibniz}
        D(\omega\wedge \eta)= d_L\omega\wedge \eta +(-1)^k\omega\wedge D\eta.
    \end{equation}
\end{definition}

Note that the differential of a representation up to homotopy of $L$ on $E^\bullet$ is determined by its values on $\Omega^0(L,E^\bullet)=\Gamma(E^\bullet)$, because of the Leibniz rule \eqref{eq:ruth_leibniz}. 
Let us also note that one can construct duals and tensor powers of representations up to homotopy \cite[Section 4]{AbCr12}. 

Next we quickly recall the notion of a VB-algebroid and the construction of the corresponding representation up to homotopy following \cite{GrMe10}. Let a double vector bundle
\begin{equation}\label{eq:diagram_VB_algebroid}
\begin{tikzpicture}[baseline=(current bounding box.center)]
\usetikzlibrary{arrows}
\node (2) at (1.5,1) {$ B $};
\node (3) at (3,1) {$ E^0 $};

\node (w) at (1.5,0) {$ L $};
\node (e) at (3,0) {$ N $};
\path[->]
(2) edge node[above]{$ $} (w)
(3) edge node[above]{$ $} (e)

(2) edge node[above]{$ $} (3)
(w) edge node[above]{$ $} (e)
;
\end{tikzpicture}
\end{equation}
with core $E^{-1}:=\ker \pi_{B\to E^0}\cap \ker \pi_{B\to L}$ be given. The bundle $B\to E^0$ has two distinguished classes of sections, called  \textbf{linear} and \textbf{core sections},
\begin{equation*}
    \begin{aligned}
        \Gamma_{\mathrm{core}}(B)&:=\{c\circ \pi_{E^0\to N}+_L 0_{E^0\to B}\,|\, c\in \Gamma(E^{-1})\}\subset \Gamma(B\to E^0)\\
        \Gamma_{\mathrm{lin}}(B)&:= \{ b\,|\, b: E^0\to B \text{ is a vector bundle morphism} \}\subset  \Gamma(B\to E^0).
    \end{aligned}
\end{equation*}

Two Lie algebroid structures $B\Ato E^0$ and $L\Ato N$ yield a \textbf{VB-algebroid} structure on the double vector bundle $B$ if and only if the following compatibility conditions hold:
\[ \begin{aligned}
    [\Gamma_{\mathrm{lin}}(B),\Gamma_{\mathrm{lin}}(B)]&\subset \Gamma_{\mathrm{lin}}(B),\\
    [\Gamma_{\mathrm{lin}}(B),\Gamma_{\mathrm{core}}(B)]&\subset \Gamma_{\mathrm{core}}(B) \text{ and}\\
    [\Gamma_{\mathrm{core}}(B),\Gamma_{\mathrm{core}}(B)  ]&=0.
\end{aligned} \]

To describe how the VB-algebroid \eqref{eq:diagram_VB_algebroid} encodes a representation up to homotopy of $L$ on $E^\bullet=E^{-1}\oplus E^0$, we recall that linear sections are sections of a Lie algebroid denoted by $\hat{L}\to N$. This Lie algebroid fits into a short exact sequence
\begin{equation}\label{eq:ruth_splitting}
    0\to \mathrm{Hom}(E^0,E^{-1})\to \hat{L}\to L\to 0.
\end{equation}
In \eqref{eq:ruth_splitting}, the map $\hat{L}\to L$ is given on sections by projecting a linear section to its base map, which is necessarily a section of $L$. 
Choosing a splitting $\sigma: L\to \hat{L}$ of \eqref{eq:ruth_splitting} allows to define 
\begin{itemize}
    \item An $L$-connection $\nabla^{E^{-1}}$ on $ E^{-1} $ given by 
    \[ \nabla^{E^{-1}}_a(c)=[\sigma(a),c]_B\] 
    for $a\in \Gamma(L)$ and $c\in \Gamma(E^{-1})=\Gamma_{\mathrm{core}}(B)$.
    \item An $L$-connection $\nabla^{E^0}$ on $E^0$ with dual connection given by
    \[ \nabla^{(E^0)^\ast}_a \xi =\sharp_B(\sigma(a)) \xi, \]
    where $a\in \Gamma(L)$ and $\xi\in \Gamma((E^0)^\ast)$ is considered as a linear function on $E^0$. 
    \item A $\mathrm{Hom}(E^0,E^{-1})$-valued two-form $\gamma$ on $L$ given by the curvature of $\sigma$, i.e.\ 
    \[ \gamma(a_1,a_2)=[\sigma(a_1),\sigma(a_2)]-\sigma([a_1,a_2]) \]
    for $a_1,a_2\in \Gamma(L)$.
    \item Finally, independent of the splitting, there is the core map $\partial : E^{-1}\to E^0$, defined to be minus the anchor of $B$ from the core of $B$ to the core $E^0$ of $TE^0$.
\end{itemize}
These maps piece together to the restriction of the differential $D$ of $\Omega(L,E)^\bullet$ to $\Gamma(E^\bullet)$.

\begin{example}\label{example:double_LA_normalbundle}
    For a Lie subalgebroid $L\Ato N$ of $A\Ato M$ there is a VB algebroid
\begin{equation}\label{eq:normal_VB_algebroid}
\begin{tikzpicture}[baseline=(current bounding box.center)]
\usetikzlibrary{arrows}
\node (2) at (1,1.5) {$ \nu_L(A) $};
\node (3) at (3,1.5) {$ \nu_N(M) $};

\node (w) at (1,0) {$ L $};
\node (e) at (3,0) {$ N $};
\draw[-Implies,double equal sign distance]
(2) -- (3);
\draw[-Implies,double equal sign distance]
(w) -- (e);
\path[->]
(2) edge node[above]{$ $} (w)
(3) edge node[above]{$ $} (e);
\end{tikzpicture}
\end{equation}
with $E^0=\nu_N(M)$ and core $E^{-1}=A|_N/L$ \cite{MePi21}. The bracket of $\nu_L(A)\Ato \nu_N(M)$ is defined such that the map induced by the normal bundle functor  
\begin{equation}\label{eq:normal_functor}
\nu:\Gamma(A,L)\to \Gamma(\nu_L(A)),\quad \nu(a)(X\, \mathrm{mod}\, TN)=(d a)(X)\, \mathrm{mod}\, TL, \quad X\in TM|_N
\end{equation}
is bracket preserving, where $\Gamma(A,L)$ denotes the set of sections of $A$ that restrict to $L$ along $N$. The image of the map $\nu$ are precisely the linear sections. To obtain a splitting and identify the core sections, we choose a vector bundle isomorphism $A|_E\simeq \mathrm{pr}^\ast A|_N$, where $\mathrm{pr}: E\to N$ is a tubular neighbourhood, and we choose a complement $A|_N=L\oplus C$. Then core sections are sections of $C\simeq E^{-1}$. We define a splitting of \eqref{eq:ruth_splitting} by using the map $\nu$ from \eqref{eq:normal_functor}, via
\[ \sigma: \Gamma(L)\to \Gamma(\hat{L}),\qquad 
 a\mapsto \nu(\mathrm{pr}^\ast a).\] 
 The $L$-connections are given as follows. For $a\in \Gamma(L)$ and $c\in \Gamma(C)$,  
\begin{equation}\label{eq:submanifold_rep_on_core}
    \nabla^C_a c=\mathrm{pr}_C [\mathrm{pr}^\ast a,\mathrm{pr}^\ast c]|_N
\end{equation}
and for $a\in \Gamma(L)$ and $f\in \mathcal{I}_N$, with $df|_N\in \Gamma(\nu_N^*)$, 
\begin{equation}\label{eq:submanifold_rep_on_conormal}
    \nabla^{\nu_N^\ast}_a df|_N=d (\sharp(\mathrm{pr}^\ast a)f)|_N.
\end{equation}
\end{example}

\begin{remark}\label{remark:dual_VB}\rm
    Given a VB-algebroid as in \eqref{eq:diagram_VB_algebroid} one can dualise over $L$ to obtain a new VB-algebroid
    \begin{equation*}
\begin{tikzpicture}[baseline=(current bounding box.center)]
\usetikzlibrary{arrows}
\node (2) at (1.2,1.2) {$ B^* $};
\node (3) at (3,1.2) {$ (E^{-1})^* $};

\node (w) at (1.2,0) {$ L $};
\node (e) at (3,0) {$ N $};
\draw[-Implies,double equal sign distance]
(2) -- (3);
\draw[-Implies,double equal sign distance]
(w) -- (e);
\path[->]
(2) edge node[above]{$ $} (w)
(3) edge node[above]{$ $} (e);
\end{tikzpicture}
    \end{equation*}
    with core $(E^0)^*$. If $\sigma$ is a splitting of \eqref{eq:diagram_VB_algebroid} with curvature $\gamma$, then the structure maps corresponding to the representation up to homotopy induced by the dual VB-algebroid are given by $-\partial^\ast$, $\nabla^{(E^0)^\ast}$, $\nabla^{(E^{-1})^\ast}$ and $-\gamma^\ast$. Note that the different signs compared to \cite[Example 4.1]{AbCr12} come from a degree shift.
\end{remark}

\begin{example}\label{example:conormal_VB}
Taking the dual of the VB-algebroid $\nu_L(A)\to L$, we obtain the conormal VB-algebroid 
   \begin{equation*}
\begin{tikzpicture}[baseline=(current bounding box.center)]
\usetikzlibrary{arrows}
\node (2) at (1.2,1.2) {$ \nu_L(A)^* $};
\node (3) at (3,1.2) {$ L^\circ $};

\node (w) at (1.2,0) {$ L $};
\node (e) at (3,0) {$ N $};
\draw[-Implies,double equal sign distance]
(2) -- (3);
\draw[-Implies,double equal sign distance]
(w) -- (e);
\path[->]
(2) edge node[above]{$ $} (w)
(3) edge node[above]{$ $} (e);
\end{tikzpicture}
\end{equation*}
with core $\nu_N^\ast$. Fix a tubular neighbourhood $\mathrm{pr}: E\to N$ of $N$ in $M$, a vector bundle isomorphism $A|_E\simeq \mathrm{pr}^\ast (A|_N)$ and a complement $A|_N=L\oplus C$.
Then we obtain a representation up to homotopy of $L$ on the graded bundle
\[\nu_N^\ast \oplus  L^\circ,\] where $\mathrm{deg}(\nu_N^\ast)=-1$ and
$\mathrm{deg}(L^\circ)=0$, and the differential is defined as explained in Example \ref{example:double_LA_normalbundle} and Remark \ref{remark:dual_VB}. Taking the (graded) exterior power of this representation up to homotopy, in the sense of \cite{AbCr12}, we obtain a representation up to homotopy of $L$ on the graded vector bundle
\[\wedge^p(\nu_{N}^*\oplus L^{\circ})=\bigoplus_{k=0}^pS^{k}(\nu_{N}^*)\otimes \wedge^{p-k}(L^{\circ}),\]
where $\mathrm{deg}(S^{k}(\nu_{N}^*)\otimes \wedge^{p-k}(L^{\circ}))=-k$.
\end{example}


\begin{lemma}\label{lemma:submanifold_spaces_E0}
   Using the notation and choices from Example \ref{example:conormal_VB}, we can identify
    \begin{equation}\label{eq:submanifold_E0}
        E_0^{p,q}\simeq\bigoplus_{i=0}^{p} \Omega^{p+q-i}(L, S^{p-i}(\nu_N^\ast)\otimes \wedge^i (L^\circ) \otimes V|_N).
    \end{equation}
   The differential induced by $d_0$ satisfies a Leibniz rule, turning 
    \[ (E_0^{p,\bullet},d_0)\simeq(\Omega(L, \wedge^p( \nu_N^\ast\oplus L^\circ)\otimes V|_N)^\bullet, D)  \]
    into a representation of $L$ up to homotopy.
\end{lemma}
\begin{proof}
    Over $E$ we decompose $A|_E=\mathrm{pr}^\ast L\oplus \mathrm{\mathrm{pr}^\ast C}$, and then forms on $A|_E$ decompose as
        \[ \Omega^k(A|_E,V|_E)=\bigoplus_{i=0}^k \Gamma(\wedge^i\mathrm{pr}^\ast C^\ast\otimes \wedge^{k-i}\mathrm{pr}^\ast L^\ast\otimes \mathrm{pr}^\ast V|_N), \]
    as in the proof of Lemma \ref{theorem:classical_form}.
    Moreover, note that the ideal $ \mathcal{I}_L $, i.e.\ the kernel of the pullback map $i^*:\Omega^{\bullet}(A|_E)\to \Omega^\bullet(L)$, is generated by $ \mathcal{I}_N$ and $\Omega^1(\mathrm{pr}^\ast C) $. From this it follows that 
    \[ \mathcal{F}_L^p\Omega^k(A|_E,V|_E)= \bigoplus_{i=0}^{k} \mathcal{I}_N^{p-i}\Gamma(\wedge^i\mathrm{pr}^\ast C^\ast\otimes \wedge^{k-i}\mathrm{pr}^\ast L^\ast\otimes \mathrm{pr}^\ast V|_N). \]
    Equation \eqref{eq:submanifold_E0} follows by using the canonical isomorphism
    \begin{equation*}
        \mathcal{I}_N^p \Gamma(F) / \mathcal{I}_N^{p+1} \Gamma(F)\simeq
        \Gamma (S^p \nu_N^\ast\otimes F|_N),
    \end{equation*}
     which hold for any vector bundle $F\to M$, and the identification
     \[ \frac{ \mathcal{F}^p_L \Omega^{\bullet}(A,V)}{\mathcal{F}^{p+1}_L \Omega^{\bullet}(A,V)} = \frac{\mathcal{F}^p_L \Omega^\bullet(A|_E,V|_E)}{\mathcal{F}^{p+1}_L \Omega^\bullet(A|_E,V|_E)}.\]

    Finally, we check the Leibniz rule \eqref{eq:ruth_leibniz}. For $\eta\in \mathcal{F}_L^p \Omega^\bullet(A|_E,V|_E)$ and $\omega\in \Omega^k(L)$, we have 
        \begin{align*}
            d_0 (\omega\wedge[\eta])&= [ d_A(\mathrm{pr}^\ast\omega\wedge\eta) ]\\
            &=[ d_A(\mathrm{pr}^\ast\omega)\wedge\eta+(-1)^k \mathrm{pr}^\ast\omega\wedge d_A\eta ]\\
            &=[\mathrm{pr}^\ast d_L(\omega)\wedge\eta+(-1)^k \mathrm{pr}^\ast\omega\wedge d_A\eta + \underbrace{(d_A \mathrm{pr}^\ast \omega - \mathrm{pr}^\ast d_L\omega)}_{\in \mathcal{I}_L}\wedge \eta]\\
            &=d_L(\omega)\wedge [\eta]+(-1)^k\omega\wedge d_0[\eta].\qedhere
        \end{align*}
\end{proof}

To complete the proof of Theorem \ref{theorem:first_page_submanifold} we need to show that the differential we obtain via the identification \eqref{eq:submanifold_E0} is given by graded antisymmetric powers of the representation up to homotopy described in Example \ref{example:conormal_VB}.

\begin{proof}[Proof of Theorem \ref{theorem:first_page_submanifold}] 
First we determine the differential on 
\[E_0^{1,\bullet}=\Omega^{\bullet+1}(L,\nu_N^\ast\otimes V|_N)\oplus \Omega^\bullet(L,L^\circ \otimes V|_N)\]
under the identification \eqref{eq:submanifold_E0}.
For $f\in \mathcal{I}_N$ and $v\in \Gamma(V|_N)$ we obtain 
\[d_0 (df|_N\otimes v)\in \Omega^1 (L,\nu_N^\ast\otimes V|_N)\oplus \Gamma(L^\circ\otimes V|_N)\]
for degree reasons. 
To compute the first summand, let $a\in \Gamma(L)$ be given. Then
\begin{equation*}
    \begin{aligned}
        d_0(df|_N\otimes v)(a)&=[d_A(f)(\mathrm{pr}^\ast a)\otimes \mathrm{pr}^\ast v + f\nabla^V_{\mathrm{pr}^\ast a} \mathrm{pr}^\ast v]\\
        &=d(\sharp(\mathrm{pr}^\ast a)f)|_N\otimes v +df|_N\otimes \nabla^{V|_N}_a v,
    \end{aligned}
\end{equation*}
which is \eqref{eq:submanifold_rep_on_conormal}. To compute the term in $\Gamma(L^\circ\otimes V|_N)$, let $c\in \Gamma(C)$ be given. Then
\begin{equation*}
    \begin{aligned}
        d_0(df|_N\otimes v)(c)&= d_A(f\otimes \mathrm{pr}^\ast v)|_N (c)\\
        &= df|_N (\sharp c)\otimes v + f|_N \nabla^V_{\mathrm{pr}^\ast c}\mathrm{pr}^\ast v |_N\\
        &= \sharp^\ast (df|_N)(c) \otimes v.
    \end{aligned}
\end{equation*}
Now let $\gamma \in \Gamma(L^\circ) $ be given. Then 
\[ d_0( \gamma \otimes v) \in \Omega^1(L, L^\circ)\oplus \Omega^2(L, \nu_N^\ast). \]
Thus, let $a,b\in \Gamma(L)$ and $c\in \Gamma(C)$ be given. Then
\begin{equation*}
    \begin{aligned}
        d_0 (\gamma\otimes v)(a)(c)&= d_A (\mathrm{pr}^\ast \gamma\otimes \mathrm{pr}^\ast v)(\mathrm{pr}^\ast a, \mathrm{pr^\ast c})\\
        &=\sharp(\mathrm{pr}^\ast a)(\mathrm{pr}^\ast\gamma(c))\otimes \mathrm{pr}^\ast v |_N+\mathrm{pr}^\ast\gamma(c) \nabla^{V}_{\mathrm{pr}^\ast a}\mathrm{pr}^\ast v |_N-\mathrm{pr}^\ast\gamma([\mathrm{pr}^\ast a, \mathrm{pr}^\ast c])|_N\\
        &= \sharp a(\gamma(c))+\gamma(c)\nabla^{V|_N}_a v-\gamma([\mathrm{pr}^\ast a, \mathrm{pr}^\ast c]|_N) v,
    \end{aligned}
\end{equation*}
which is dual to \eqref{eq:submanifold_rep_on_core}.
Finally, for the contribution in $\Omega^2(L, \nu_N^\ast)$ we find
\begin{equation*}
    \begin{aligned}
        d_0 (\gamma\otimes v) (a,b)&=- d(\mathrm{pr}^\ast\gamma([\mathrm{pr}^\ast a,\mathrm{pr}^\ast b]))|_N   \otimes v\\
        &=-d \mathrm{pr}^\ast\gamma ([\mathrm{pr}^\ast a,\mathrm{pr}^\ast b]-\mathrm{pr}^\ast[a, b])|_N \otimes v,
    \end{aligned}
\end{equation*}
which shows that $(E_0^{1,\bullet},d_0)$ is indeed dual of the representation up to homotopy given by the dual of \eqref{eq:normal_VB_algebroid} by Remark \ref{remark:dual_VB}.

To complete the proof of Theorem~\ref{theorem:first_page_submanifold}, we check that the differential on $E_0^{p,\bullet}$ for $p> 1$ is given by tensor powers of the differential on $E_0^{1,\bullet}$. 
Again, by the Leibniz rule \eqref{eq:ruth_leibniz} it is enough to calculate $d_0$ on $\Gamma(\wedge^p( \nu_N^\ast\oplus L^\circ)\otimes V|_N)$. 
Let $f\in \mathcal{I}_N$ and $\eta\in \mathcal{F}_L^{p-1}\Omega(A,V)^\bullet$ be given. 
Then $f\eta\in \mathcal{F}^p_L\Omega(A,V)^\bullet$ and, using the Leibniz rule of $d_A$,
\begin{equation*}
    \begin{aligned}
        d_0[f \eta]&=[ d_A (f\eta)  ]\\
        &=[d_A f \wedge \eta + f d_A\eta]\\
        &=d_0[f]\wedge [\eta]+ [f]\wedge d_0[\eta].
    \end{aligned}
\end{equation*}
Similarly, a graded Leibniz rule holds for $\gamma\wedge $, where $\gamma\in \Gamma(L^\circ)$.   
\end{proof}

\subsubsection{Invariant submanifolds}\label{sec:linearisable_submanifold}
In this section let $N\subset M$ be a closed, embedded and invariant submanifold of $A\Ato M$, i.e.\ \[\sharp (A|_N)\subset TN.\] 
In this case, $L=A|_N$ is a subalgebroid, and $L^{\circ}=0_N$. Moreover, the filtration becomes
\begin{equation}\label{eq:filtration_for_A|_N}
\mathcal{F}_L^p\Omega^\bullet(A,V)=\mathcal{I}_N^p\Omega^\bullet(A,V).
\end{equation}
The VB-algebroid $\nu_{A|_N}(A)\Ato \nu_N$ (recall Example \ref{example:double_LA_normalbundle}) is the \textbf{action Lie algebroid}
\[ A|_N\ltimes \nu_N\Ato \nu_N,\]
corresponding to the canonical representation of $A|_N$ on the normal bundle $\nu_N$, given by
\[ \nabla_a (X|_N \text{ mod }TN)=[\sharp \tilde{a},X]|_N \text{ mod }TN \]
for $X\in \Gamma(TM) $ and $\tilde{a}\in \Gamma(A)$ some extension of $a\in \Gamma(A|_N)$. We briefly recall the construction of the action Lie algebroid for our case, following \cite[Theorem 2.4]{HiMa90}. Writing $ \mathrm{pr}: \nu_N\to N $ for the projection, the vector bundle structure is $ A|_N\ltimes \nu_N: =\mathrm{pr}^* A|_N\to \nu_N $. 
On pullback sections the anchor $ \sharp_\ltimes : A|_N\ltimes\nu_N\to T\nu_N $ is defined by
	\begin{equation*}
	[\sharp_\ltimes (\mathrm{pr}^* a),V^\mathrm{ver}  ]=\big( \nabla_a V \big)^\mathrm{ver}
	\end{equation*}
	for $ a\in \Gamma(A|_N) $ and $ V\in \Gamma(\nu_N) $. Here, we denote by 
    \[ \cdot\,^\mathrm{ver}: \Gamma(\nu_N)\to \Gamma(T\nu_N) \]
    the vertical lift. 
    In particular, $ \sharp_\ltimes $ maps pullback sections into linear vector fields, i.e.\ vector fields of homogeneous degree $0$. The bracket $ [\cdot,\cdot]_\ltimes $ on $ \Gamma(A|_N\ltimes \nu_N) $ is defined on pullback sections by
	\begin{equation*}
	[\mathrm{pr}^* a,\mathrm{pr}^* b]_\ltimes =\mathrm{pr}^* ([a,b]_{A|_N}),
	\end{equation*}
	where $ a,b\in \Gamma(A|_N) $, and is extended to all sections using the Leibniz rule. 
 
 The action Lie algebroid is linear in the following sense.

\begin{lemma}\label{lemma:linearisable_scalar_multiplication_is_morph}
    The scalar multiplication on $\nu_N$ induces a Lie algebroid automorphism
\[ m_\lambda: A|_N\ltimes \nu_N\diffto A|_N\ltimes \nu_N, \qquad \lambda\in \mathbb{R}\backslash \{0\}.\]
\end{lemma}
\begin{proof}
    First note that for all $ a\in \Gamma(A|_N) $ we have
		\begin{equation*}
		m_\lambda^\ast \sharp_\ltimes (\mathrm{pr}^* a)=\sharp_\ltimes (\mathrm{pr}^* a)
		\end{equation*}
	since $ \sharp_\ltimes(\mathrm{pr}^* a)  $ is a linear vector field. Then compatibility with the bracket follows immediately since it clearly holds on pullback sections.
\end{proof}

For later use, we introduce the following notion.
\begin{definition}
A Lie algebroid $A\Ato M$ is called \textbf{linearisable} around the invariant submanifold $N\subset M$ if $A$ is isomorphic  in a neighbourhood of $N$ to the action Lie algebroid $A|_N\ltimes \nu_N\Ato \nu_N$ restricted to a neighbourhood of $N$. 
\end{definition}

 The dual VB-algebroid $\nu_{A|_N}(A)^*\Ato N$ is the semi-direct product associated to the conormal representation of $A|_N$ on \[\nu_N^\ast\to N.\]

The corresponding representation up to homotopy is therefore the classical representation
on $\nu_N^*$, and thus, we obtain the following simplification of Theorem \ref{theorem:first_page_submanifold}.

\begin{theorem}\label{theorem:invariant}
    Let $N\subset M$ be closed, embedded submanifold which is invariant for $A\Ato M$. The Serre spectral sequence of the Lie subalgebroid $L=A|_N$ converges to the formal Lie algebroid cohomology $\mathrm{H}^\bullet(\mathscr{J}^\infty_N\Omega(A,V))$, and its first page is given by
        \[ E_1^{p,q}\simeq\mathrm{H}^{p+q}(A|_N, S^p\nu_N^\ast\otimes V|_N). \]
\end{theorem}

This theorem has a straightforward consequence.

\begin{corollary}\label{corollary:vanishing:cohomology}
If $\mathrm{H}^k(A|_N,S^p\nu_N^*\otimes V|_N)=0$ for all $p\geq 0$, then
\[\mathrm{H}^{k}(\mathscr{J}^\infty_N\Omega(A,V))=0.\]
\end{corollary}

The assumptions of the corollary hold, for example in the following cases:
\begin{itemize}
    \item $N$ is a point, the Lie algebra $\mathfrak{g}=A|_N$ is semisimple and $k=1$ or $k=2$. This is the Whitehead Lemma.
        \item $A$ is the Lie algebroid of a Hausdorff Lie groupoid $\mathcal{G}\rightrightarrows N$, such that: (1) $\mathcal{G}$ is proper, i.e.\ the map $(t,s):\mathcal{G}\to M\times M$ is proper; (2) the fibres of $t$ have zero de Rham cohomology in degree $i$, with $1\leq i \leq k$; and (3) $V|_N$ integrates to a representation of $\mathcal{G}$. See \cite[Proposition 1 \& Theorem 4]{Crai03}.
    \item $A$ is the Lie algebroid of a Hausdorff Lie groupoid $\mathcal{G}\rightrightarrows N$, such that: (1) the target map $t:\mathcal{G}\to N$ is proper; (2) the fibres of $t$ have zero de Rham cohomology in degree $k$; and (3) $V|_N$ integrates to a representation of $\mathcal{G}$. See \cite[Lemma C.1]{Mar14}. 
\end{itemize}

\begin{remark}\rm
Formal cohomology plays an important role in formal linearisation problems in Poisson geometry (see Section \ref{sec:coupling_poisson} for more on Poisson geometry, including definitions and references). A Poisson manifold $(M,w)$ has an associated cotangent Lie algebroid $A=T^*M$, whose cohomology is called the Poisson cohomology of $(M,w)$. We say that $w$ is \textbf{linearisable} around an invariant submanifold $N\subset M$ (also called a Poisson submanifold), if the Lie algebroid $T^*M$ is linearisable around $N$ (for leaves, this is equivalent to the more standard notion of linearisation discussed in Section \ref{sec:normal_forms_presympl_leaves}, see \cite[Theorem 8.2]{FeMa22}). In degree two, Poisson cohomology encodes infinitesimal deformations of the Poisson structure. Therefore, the formal infinitesimal rigidity of $w$ around the Poisson submanifold $N\subset M$ translates to
\begin{equation}\label{eq:inf_rig}
    \mathrm{H}^2(\mathscr{J}^\infty_N\Omega(T^*M))=0.
\end{equation}
Formal rigidity was studied first around points, i.e.\ when $N=\{x\}$, and it was shown that if $\mathfrak{g}:=A_x$ is semisimple, then the Poisson structure $w$ is formally linearisable around $x$ \cite[Theorem 6.1]{Wein83}. This result was extended to symplectic leaves in \cite[Theorem 7.1]{IKV98}, and to general Poisson submanifolds $N$ in \cite[Theorem 1.1]{Mar12}, where it was shown that Poisson structures satisfying $\mathrm{H}^2(T^*M|_N,S^p\nu_N^*)=0$ for all $p$ are formally rigid. Corollary \ref{corollary:vanishing:cohomology} can be seen as the infinitesimal version of this result, as it gives the infinitesimal version of formal rigidity, i.e.\ that \eqref{eq:inf_rig} holds. 
\end{remark}

Next we show that if the Lie algebroid is linearisable around the invariant submanifold, then the Serre spectral sequence stabilises at $E_1$.

\begin{theorem}\label{theorem:linearisable}
  Let $N\subset M$ be closed, embedded submanifold which is invariant for $A\Ato M$. Suppose that $A$ is linearisable around $N$. Then the Serre spectral sequence associated to the subalgebroid $L=A|_N$ stabilises at $E_1$. Hence, the formal algebroid cohomology is given by 
  \[ \mathrm{H}^\bullet (\mathscr{J}^\infty_N\Omega^\bullet(A))\simeq \prod_{j=0}^\infty \mathrm{H}^\bullet(A|_N,S^j\nu_N^\ast). \]
\end{theorem}

\begin{proof}
The formal cohomology of $A$ around $N$ is canonically isomorphic to the formal cohomology of $A|_U$, for any neighbourhood $U$ of $N$. Since $A\Ato M$ is linearisable around $N$, we may therefore assume that $A=A|_N\ltimes \nu_N$. To deduce the result, it suffices to show that the short exact sequence of cochain complexes
\begin{equation}\label{eq:ses_for_linearisable}
    0\to \mathcal{F}_{A|_N}^{p+1}\Omega^{p+q}(A)\to \mathcal{F}_{A|_N}^{p}\Omega^{p+q}(A)\to E_0^{p,q}\to 0
\end{equation}
admits a splitting $\sigma: E_0^{p,q}\to \mathcal{F}_{A|_N}^{p}\Omega^{p+q}(A)$ which is compatible with the differentials. Indeed, if $\eta\in E_0^{p,q}$ is $d_0$-closed then $\sigma(\eta)$ is $d_A$-closed. Thus all subsequent differentials $d_r$, $r>0$ are zero, $E_\infty^{p,q}=E_1^{p,q}$ and the statement follows.

To build the splitting, we use the canonical identification
    \[ \mathrm{Pol}: \Gamma(S^p\nu_N^\ast)\diffto \mathrm{Pol}^p(\nu_N),\]
    where $\mathrm{Pol}^p(\nu_N)\subset C^\infty(\nu_N)$ denotes the homogeneous polynomials of degree $p$ on $\nu_N$. This yields the splitting  $\sigma:=\mathrm{pr}^\ast\otimes \mathrm{Pol}: \Omega^{p+q}(A|_N,S^p\nu_N^\ast)\to \Omega^{p+q}(A)$ of \eqref{eq:ses_for_linearisable}. To show that $\sigma$ is compatible with the differentials, it suffices to show that its image is a subcomplex. For this, note that the image of $\sigma$ can be characterised as forms $\theta\in \Omega^{p+q}(A)$ such that,
    \[ m_\lambda^\ast \theta=\lambda^p \theta,\qquad \textrm{for all }\quad \lambda\in \mathbb{R}\backslash\{0\}. \] 
Since $m_\lambda$ is an automorphism of $A$ (Lemma \ref{lemma:linearisable_scalar_multiplication_is_morph}), it follows that $d_A\circ m_{\lambda}^*= m_{\lambda}^*\circ d_A$. Therefore, the image of $\sigma$ is a subcomplex. 
\end{proof}

\subsubsection{Finite jets along invariant submanifolds}

In this subsection we discuss a version of the spectral sequence for finite order jets. 

Let $N\subset M$ be a closed, embedded invariant submanifold of $A\Ato M$. For $k\in \mathbb{N}$, the space of $k$-th order jets along $N$ of forms on $A$ with values in $V$ is defined as
\[ \mathscr{J}^k_N\Omega^\bullet(A,V)=\Omega^\bullet(A,V) / \mathcal{I}_N^{k+1} \Omega^\bullet(A,V).\]
By \eqref{eq:filtration_for_A|_N}, $\mathcal{I}_N^{k+1} \Omega^\bullet(A,V)$ is a differential graded $\Omega^\bullet(A)$-submodule of $\Omega^\bullet(A,V)$. Consequently, there is an induced differential on $\mathscr{J}^{k}_N\Omega^\bullet(A,V)$, giving rise to cohomology of finite jets.

If $L\Ato N$ is a Lie subalgebroid, the inclusion induces a filtration on $\mathscr{J}^k_N\Omega^\bullet(A,V)$ by
\begin{equation}\label{eq:filtration_finite_jets}
\mathcal{F}^p_L\mathscr{J}^k_N\Omega^\bullet(A,V):=\mathcal{F}^p_L\Omega^\bullet(A,V) / (\mathcal{I}_N^{k+1} \Omega^\bullet(A,V)\cap \mathcal{F}^p_L\Omega^\bullet(A,V) ). 
\end{equation}
The corresponding spectral sequence always converges.

\begin{theorem}
    Let $L\Ato N$ be a Lie subalgebroid of $A\Ato M$ over the closed, embedded invariant submanifold $N\subset M$. Fix $k\in \mathbb{N}$. 
    The spectral sequence induced by the filtration     \eqref{eq:filtration_finite_jets} stabilises at page $r+1$, where $r:=k+\mathrm{rank}(A)-\mathrm{rank}(L)$, and converges to $\mathrm{H}^\bullet(\mathscr{J}^k_N\Omega(A,V))$.
     The zeroth page is given by 
        \[ E_0^{p,q}\simeq \bigoplus_{i=\mathrm{max}\{0,p-k\}}^p \Omega^{p+q-i}(L,S^{p-i}\nu_N^\ast\otimes \Lambda^iL^\circ \otimes V|_N)
     \]
    with $d_0$ corresponding to the representation up to homotopy of $L$ on $\wedge^p (\nu_N^*\oplus L^{\circ})\otimes V|_N$. 
\end{theorem}
\begin{proof}
    We restrict to a tubular neighbourhood $E\subset M$ of $N$. By the proof of Lemma \ref{lemma:submanifold_spaces_E0} and the definition \eqref{eq:filtration_finite_jets}, 
    \[  \mathcal{F}^p_L\mathscr{J}^k_N\Omega^\bullet(A|_E,V|_E)= \bigoplus_{i=\mathrm{max}\{0,p-k\}}^{p} \mathcal{I}_N^{p-i}\Gamma(\wedge^i\mathrm{pr}^\ast L^\circ\otimes \wedge^{{\bullet}-i}\mathrm{pr}^\ast L^\ast\otimes \mathrm{pr}^\ast V|_N).\]
    This implies that $\mathcal{F}^p_L\mathscr{J}^k_N\Omega^\bullet(A,V)=0$ if $p\geq r+1$. Hence the filtration is finite and the spectral sequence stabilises as claimed. The rest of the proof is completely analogous the proof of Theorem \ref{theorem:first_page_submanifold}.
\end{proof}

For $L=A|_N$, we obtain the following simplifications, analogue to Theorems \ref{theorem:invariant} and  \ref{theorem:linearisable}.

\begin{theorem}
    Let $N\subset M$ be a closed, embedded invariant submanifold of $A\Ato M$.     \begin{enumerate}
        \item The first page of the spectral sequence converging to $\mathrm{H}^\bullet(\mathscr{J}^k_N\Omega(A,V))$ is given by
        \[ E_1^{p,q}\simeq\begin{cases}
            \mathrm{H}^{p+q}(A|_N, S^p\nu_N^\ast\otimes V|_N)&\text{ for }p\leq k\\
            0 &\text{ otherwise.}
        \end{cases} \]
        \item If $A\Ato M$ is linearisable around $N$, then the spectral sequence stabilises at $E_1$ and so
    \[ \mathrm{H}^\bullet (\mathscr{J}^k_N\Omega^\bullet(A))\simeq \prod_{j=0}^{k} \mathrm{H}^\bullet(A|_N,S^j\nu_N^\ast). \]
    \end{enumerate}
\end{theorem}

\subsubsection{Transverse submanifolds and locality}

Another special class of submanifolds is given by transversals. A \textbf{transversal} $\iota: X\hookrightarrow M$ is an embedded submanifold, such that the inclusion is transverse to the anchor. Then $\iota^! A=\sharp^{-1}(TX)$ is a Lie subalgebroid. Note that any Lie subalgebroid of $A$ over $X$ is contained in $\iota^!A$. By \cite[Theorem 4.1]{BLM19} there exists a tubular neighbourhood $\mathrm{pr}: E\to X$ of $X$ over which there is an isomorphism of Lie algebroids 
\begin{equation}\label{eq:tranverse_trivial}
 A|_E\simeq \mathrm{pr}^! \iota^! A.
 \end{equation}

This simple local form also manifests itself at the level of cohomology. 

\begin{theorem}\label{theorem:formal_of_transversal}
    Let $\iota: X\hookrightarrow M$ be a closed transversal of $A\Ato M$. Then the first page of the Serre spectral sequence associated $\iota^! A$ is given by
    \[ E_1^{p,q}\simeq \mathrm{H}^{p+q}(\iota^! A, V|_X). \]
    Moreover, the spectral sequence stabilises at the first page, and so
     \begin{equation}\label{eq:iso:formal:restr}
\mathrm{H}^\bullet(\mathscr{J}_X^\infty\Omega^\bullet(A,V))\simeq \mathrm{H}^\bullet(\iota^! A,V|_X). 
\end{equation}
\end{theorem}

Theorem \ref{theorem:formal_of_transversal} follows immediately from the following lemma, which emphasises the local nature of the Serre spectral sequence for formal cohomology.

\begin{lemma}\label{lemma:coregular_homotopy}
    Let $\iota : X\to M$ be a closed transversal of $A\Ato M$, $L\Ato N$ be a Lie subalgebroid over a closed embedded submanifold $N\subset X$ and $V\to M$ be a representation of $A\Ato M$. Let $\{E_r^{\bullet,\bullet}\}$ and $\{\hat{E}_r^{\bullet,\bullet}\}$ denote the spectral sequences arising from the inclusions $L\hookrightarrow A$ and $L\hookrightarrow \iota^! A$, respectively.
    Then the inclusion $j\colon \iota^! A\hookrightarrow A$ induces an isomorphism
    \[ j^\ast\colon \{E_r^{\bullet,\bullet}\}_{r>0}\diffto \{\hat{E}_r^{\bullet,\bullet}\}_{r>0}.  \]
\end{lemma}

\begin{remark}\label{remark:cohomology_for_transversals} \rm
In fact, a local version of the isomorphism \eqref{eq:iso:formal:restr} in Theorem \ref{theorem:formal_of_transversal} holds as well. Namely, for a tubular neighbourhood $E$ of $N$ above which \eqref{eq:tranverse_trivial} holds, we have that
\begin{equation}\label{eq:cohomology_vb}
\mathrm{H}^\bullet(A|_E, V|_E)\simeq \mathrm{H}^\bullet(\iota^! A,V|_X).
\end{equation}
A way to obtain this isomorphism, which will be used in the proof of Lemma \ref{lemma:coregular_homotopy}, is 
by constructing a homotopy operator from an Euler vector field. Denote $B=\iota^!A$, and identify $\mathrm{pr}^!B\simeq A|_E$. Let $a\in \Gamma(\ker d\mathrm{pr})$ be the Euler vector field of $E$, and $\Phi^a_{t}: \mathrm{pr}^! B\to \mathrm{pr}^! B$ be the flow of $a$, viewed as a section $a\in \Gamma(\mathrm{pr}^!B)$. Then $\mu_t:=\Phi^a_{\log t}$ extends to $t=0$ as $\mu_0=j\circ \mathrm{pr}^!$, where $j:B\hookrightarrow \mathrm{pr}^!B$ is the inclusion of $B$ into $\mathrm{pr}^!B$ over the zero section. Moreover, $\mu_1=\mathrm{id}_{\mathrm{pr}^! B}$ and $\mu_t|_{j(B)}=\mathrm{id}_{j(B)}$. Using parallel transport along the $\mathrm{pr}^!B$-paths $\mu_t$, we obtain a compatible isomorphism of representations $V|_E\simeq {\mathrm{pr}}^!(V|_X)$. Analogous to the proof of the Relative Poincar\'{e} Lemma \cite{Wein71}, for $\omega\in \Omega^\bullet(\mathrm{pr}^! B,{\mathrm{pr}}^!(V|_X))$ one finds    
\begin{equation*}
			\begin{aligned}
			\omega-(\mathrm{pr}^!)^\ast j^\ast \omega&= \mu^\ast_1\omega-\mu^\ast_0\omega\\
			&=\int_0^1 \frac{d}{d t}\mu^\ast_t\omega d t\\
			&=\int_0^1 \frac{1}{t} \mu_t^\ast \mathscr{L}_a \omega d t\\
			&= \int_0^1 \frac{1}{t} \mu_t^\ast ( d_{\mathrm{pr}^!B} \mathrm{i}_a+\mathrm{i}_a d_{\mathrm{pr}^!B} )\omega d t\\
			&= ( d_{\mathrm{pr}^!B}\circ h+h\circ d_{\mathrm{pr}^!B} )\omega,
			\end{aligned}
			\end{equation*}
			where 
			\begin{equation*}
			h(\omega) = \int_0^1 \frac{1}{t} \mu_t^\ast\mathrm{i}_a\omega d t.
			\end{equation*}
Thus $j^\ast:\Omega^\bullet (\mathrm{pr}^!B,{\mathrm{pr}}^!(V|_X))\to  \Omega^\bullet(B,V|_X)$ is a quasi-isomorphism with quasi-inverse $(\mathrm{pr}^!)^*$. Alternatively, one can view $\mu_t$ as a retraction of $\mathrm{pr}^! B$ to $j(B)$ and obtain the result that way, see \cite[Theorem 5.1, Remark 6.7]{BrPa20}, \cite[Theorem 11]{Bal12}, or use spectral sequence arguments \cite[Theorem 2]{Crai03}, see Section \ref{sec:pullback_liealgebroids} for details.
    
\end{remark}

\begin{proof}[Proof of Lemma \ref{lemma:coregular_homotopy}]
    We can assume that $A=A|_E=\mathrm{pr}^! \iota^! A$. First note that $j^\ast$ and $(\mathrm{pr}^!)^\ast$ respect the filtrations. In fact, we even have 
    \[ j^\ast \mathcal{F}^p_L\Omega^\bullet(A,V)=\hat{\mathcal{F}}_L^p\Omega^\bullet(\iota^! A, V|_X). \]
    Thus, both maps descend to $E_0$ and we obtain a map between spectral sequences
    \[ j^\ast\colon \{E_r^{\bullet,\bullet}\}\to \{\hat{E}_r^{\bullet,\bullet}\}.   \]
    To show that $j^\ast\colon E_0\to \hat{E}_0$ is a quasi-isomorphism, we show that the homotopy $h$ from Remark \ref{remark:cohomology_for_transversals} is compatible with the filtration on $\Omega^\bullet(A,V)$. Since the Euler vector field vanishes on $X$ we immediately obtain $i_a\mathcal{F}^p_L\Omega^\bullet(A,V)\subset \mathcal{F}^p_L\Omega^{\bullet-1}(A,V)$. Moreover, the flow $\mu_t$ stabilises $\iota^! A$, showing that $h$ indeed respects the filtration and descends to $E_0$. 
    
    Thus, $j^\ast\colon \{E_1^{\bullet,\bullet}\}\diffto \{\hat{E}_1^{\bullet,\bullet}\}$ is an isomorphism, and by the Mapping Lemma \cite[Lemma 5.2.4]{Weib94} the statement follows.
\end{proof}

\subsubsection{Coregular submanifolds}

A submanifold $N\subset M$ is called \textbf{coregular submanifold} for a Lie algebroid $A\Ato M$ if 
\[W_N:=\sharp (A|_N) + TN\subset TM|_N\] is a subbundle. 
Then $L=\sharp^{-1}(TN)$ is a Lie subalgebroid, and as for transversals, any Lie subalgebroid of $A$ with base $N$ is also a Lie subalgebroid of $L$. 
In fact $N$ is a coregular submanifold if and only if $L=\sharp^{-1}(TN)$ is a subbundle of $A|_N$. Coregular submanifolds interpolate between invariant submanifolds 
($W_N=TN$) and transverse submanifolds ($W_N=TM|_N$). As for these extreme cases, we find that the first page of the Serre spectral sequence associated to $L$ is given by the cohomology of $L$ with values in a classical representation.

\begin{theorem}\label{theorem:coregular_E1}
    Let $N\subset M$ be a closed, embedded coregular submanifold of $A\Ato M$ and $V\to M$ a representation.
    Then $L:=\sharp^{-1}(TN)\Ato N$ has a canonical representation on 
    \[W_N^\circ \subset T^\ast M|_N,\] and the first page of the Serre spectral sequence associated to $L$ is given by
    \[ E_1^{p,q}\simeq \mathrm{H}^{p+q}(L,S^p (W_N^\circ)\otimes V|_N).\]
\end{theorem}

We first describe the representation.

\begin{lemma}\label{lemma:coregular_representation}
 The following defines a representation of $L:=\sharp^{-1}(TN)\Ato N$ on $W_N^{\circ}$.  \begin{equation}\label{eq:coregular_representation}
        \nabla_b (df|_N)=d(\sharp (\tilde{b})(f))|_N
    \end{equation}
    for $\tilde{b}\in \Gamma(A)$ with $\tilde{b}|_N=b$ and $f\in \mathcal{I}_N$ such that $df|_N\in (\sharp A|_N)^\circ$.
\end{lemma}
\begin{proof}
    Since $N$ is coregular, $W_N^{\circ}\to N$ is a vector bundle. Note that, if 
    $f\in \mathcal{I}_N$ satisfies $df|_N\in (\sharp A|_N)^\circ$ then $df|_N\in \Gamma(W_N^{\circ})$. Since $N$ is closed and embedded, all sections of $W_N^{\circ}$ can be written in this form.
    
    To see that \eqref{eq:coregular_representation} is independent of the choice of the extension $\tilde{b}$, let $b\in \Gamma(L)$ and $\tilde{b},\hat{b}\in \Gamma(A)$ be two extensions. 
    Then $(\tilde{b}-\hat{b})|_N=0$. Thus, we can write 
    \[ \tilde{b}-\hat{b}=\sum\nolimits_i g_i a_i \]
    for suitable $a_i\in \Gamma(A)$ and $g_i\in \mathcal{I}_N$. For $f\in C^\infty(M)$ such that $df|_N\in (\sharp A|_N)^\circ$
    we have that
    \[ \sharp (\tilde{b}-\hat{b})(f)=\sum\nolimits_i g_i df(\sharp (a_i)) \]
    vanishes to second order along $N$, hence \eqref{eq:coregular_representation} does not depend on the choice of $\tilde{b}$. 
    
    Next, we show that \eqref{eq:coregular_representation} defines a section of $W_N^{\circ}$. 
    For $f\in \mathcal{I}_N$ we have $\sharp(\tilde{b})(f)|_N=0$ as $\sharp(\tilde{b})|_N\in \Gamma(TN)$. Thus $ \nabla_b (df|_N)\in \Gamma(TN^\circ)$. 
    To see that $\nabla_b (df|_N)$ annihilates $\sharp(A|_N)$ let $a\in \Gamma(A)$ be given. Then
    \begin{equation*}
        \begin{aligned}
            \nabla_b (df|_N) (\sharp(a)|_N)&=\sharp(a)(\sharp(\tilde{b}) f)|_N\\
            &=\underbrace{\sharp([a,\tilde{b}])f|_N}_{=0}+\sharp(\tilde{b})(\underbrace{\sharp(a)f}_{\in \mathcal{I_N}})|_N=0.
        \end{aligned}
    \end{equation*}
    In conclusion, \eqref{eq:coregular_representation} is well-defined. 
    Clearly, \eqref{eq:coregular_representation} gives an $L$-connection on $W_N^{\circ}$, and flatness follows from the fact that, for $b_1,b_2\in \Gamma(L)$, $[\tilde{b_1},\tilde{b_2}]\in \Gamma(A)$ extends $[b_1,b_2]$.
\end{proof}

To prove Theorem \ref{theorem:coregular_E1} we use the fact that any coregular submanifold $N\subset M$ is contained in some \textbf{minimal transversal} $X\subset M$, i.e.\ one satisfying (see \cite[Section 9.2]{FeMa22})
\[TX|_N\cap W_N=TN,\qquad TX|_N + W_N=TM|_N.\]
Then $N\subset X$ is an invariant submanifold for $\iota_X^!A\Ato X$, and we have $L=(\iota_X^! A)|_N$. A direct consequence of Lemma \ref{lemma:coregular_homotopy} and Lemma \ref{theorem:invariant} is the following observation.

\begin{lemma}\label{lemma:coregular_E0}
    Let $A\Ato M$ be a Lie algebroid and $V\to M$ a representation of $A$. Let $ N\subset M $ be a closed, embedded coregular submanifold, $L=\sharp^{-1}(TN)$ and $\iota_X: X\to M$ a minimal transversal containing $N$. 
    Then the Serre spectral sequences arising from the inclusions $L\hookrightarrow A$ and $L\hookrightarrow \iota_X^! A$ are isomorphic on all pages $r>0$. In particular,
    \[ E_1^{p,q}\simeq \mathrm{H}^{p+q}(L,S^p\nu_N(X)^\ast \otimes V|_N). \]
\end{lemma}

\begin{proof}[Proof of Theorem \ref{theorem:coregular_E1}]
    By Lemma \ref{lemma:coregular_E0} we have $E_1^{p,q}\simeq \mathrm{H}^{p+q}(L,S^p\nu_N(X)^\ast \otimes V|_N)$ for a minimal transversal $X\hookrightarrow M$ of $N$. There is a canonical isomorphism $\nu_N(X)^\ast\simeq W_N^{\circ}$ induced by the inclusion $X\hookrightarrow M$, under which the respective representations of $L$ coincide. Thus Theorem \ref{theorem:coregular_E1} follows.
\end{proof}

The main tool in proving the splitting theorem for Lie algebroid transversals in \cite{BLM19} are Euler-like sections. In the following theorem we use a more general version of these sections. 
Such sections appear for example when blowing up a transversal of codimension one (called an \textbf{elementary modification} in \cite{GuLi14}), see \cite[Lemma 5.8]{Sch24}). We obtain the following generalisations (with trivial coefficients) to Remark \ref{remark:cohomology_for_transversals}, Lemma \ref{lemma:coregular_homotopy} and Theorem \ref{theorem:formal_of_transversal}.

\begin{theorem}
    Let $N\subset M$ be a closed, embedded coregular submanifold of $A\Ato M$ and $B=\sharp^{-1}(TN)$. Suppose there exists a section $a\in \Gamma(A)$ such that $\sharp a \in \Gamma(TM)$ is Euler-like along $N$ (i.e.\  $\sharp a$ is the Euler vector field of a tubular neighbourhood $E$ of $N$) and the inner derivation
    \begin{equation}\label{eq:eulerlike_vanishing_inner_derivation}
        [a|_N,\cdot]_B:\Gamma(B)\to \Gamma(B)
    \end{equation}
    vanishes identically. 
    \begin{enumerate}
        \item The inclusion $j:B\hookrightarrow A|_E$ induces an isomorphism in Lie algebroid cohomology
        \[ j^\ast:\mathrm{H}^\bullet(A|_E)\diffto \mathrm{H}^\bullet(B). \]
        \item\label{theorem:coregular_cohomology:item:lemma} Let $L\Ato Y$ be a Lie subalgebroid of $A$ with $Y\subset N$, such that $a|_Y\in \Gamma(L)$. Then, denoting the spectral sequences arising from the inclusions $L\hookrightarrow A$ and $L\hookrightarrow B$ by $\{E_r^{\bullet,\bullet}\}$ and $ \{\hat{E}_r^{\bullet,\bullet}\}$ respectively, the map $j^\ast: \{E_r^{\bullet,\bullet}\}\to \{\hat{E}_r^{\bullet,\bullet}\}$ induced by the inclusion of $B$ is a quasi-isomorphism on page zero and an isomorphism on all pages $r>0$. 
        \item  We have 
    \[ j^\ast:\mathrm{H}^\bullet(\mathscr{J}_N^\infty\Omega^\bullet(A))\diffto \mathrm{H}^\bullet(B). \]
    \end{enumerate}
\end{theorem}
\begin{proof}
    First note that $\sharp a|_N=0$ implies $a|_N\in \Gamma(B)$, thus \eqref{eq:eulerlike_vanishing_inner_derivation} is well-defined.
    The proof of the first part is completely analogous to Remark \ref{remark:cohomology_for_transversals} using the flow of $a$. 
    The assumption \eqref{eq:eulerlike_vanishing_inner_derivation} ensures that $\mu_t|_{j(B)}=\mathrm{id}_{j(B)}$, and $\mu_0$ when considered a map onto its image (note that $\mathrm{im} \mu_0 =B$) replaces the map $\mathrm{pr}^!$. 
    For the second part, following the proof of Lemma \ref{lemma:coregular_homotopy} we only need to show that the homotopy operator still respects the filtration. 
    The flow $\mu_t^\ast$ does because $\mu_t|_{j(B)}=\mathrm{id}_{j(B)}$, i.e.\ $\mu_t$ stabilises $L$. 
    Moreover, since by assumption $a|_Y\in \Gamma(L)$, we have $\mathrm{i}_a \mathcal{F}^p_L\Omega^\bullet(A)\subset \mathcal{F}^p_L\Omega^{\bullet-1}(A)$.
    In conclusion, the homotopy operator $h$ descends to $E_0$, showing that $j^\ast:E_0^{\bullet,\bullet}\to \hat{E}_0^{\bullet,\bullet}$ is indeed a quasi-isomorphism. 
    The Mapping Lemma \cite[Lemma 5.2.4]{Weib94} then implies the second part. 
    The last part follows from \ref{theorem:coregular_cohomology:item:lemma} with $L=B$.
\end{proof}

\section{Lie algebroid extensions}\label{sec:liealgebroid_extensions} 

In this section, we consider the spectral sequence corresponding to the kernel of a Lie algebroid submersion $\Pi:A\to B$. More precisely, we fix a diagram of Lie algebroid maps
\begin{equation}\label{eq:ss_Liealgebroid_ausgangslage}
\begin{tikzpicture}[baseline=(current bounding box.center)]
\usetikzlibrary{arrows}
\node (1) at (0,1) {$ 0 $};
\node (2) at (1.5,1) {$ L $};
\node (3) at (3,1) {$ A $};
\node (4) at (4.5,1) {$ B $};
\node (5) at (6,1) {$ 0 $};

\node (w) at (1.5,0) {$ M $};
\node (e) at (3,0) {$ M $};
\node (r) at (4.5,0) {$ Q $};
\draw[-Implies,double equal sign distance]
(2) -- (w);
\draw[-Implies,double equal sign distance]
(3) -- (e);
\draw[-Implies,double equal sign distance]
(4) -- (r);
\path[->]
(1) edge node[]{$  $} (2)
(2) edge node[above]{$ i $} (3)
(3) edge node[above]{$ \Pi $} (4)
(4) edge node[]{$  $} (5)
(w) edge node[above]{$\mathrm{id}_M $} (e)
(e) edge node[above]{$ \pi $} (r);
\end{tikzpicture}
\end{equation}
which is exact, i.e.\ for each  $x\in M$, we have a short exact sequence of vector spaces
\[0\to L_x\to A_x\to B_{\pi(x)}\to 0,\]
and the base map $\pi: M\to Q $ is a surjective submersion. Fix also a representation $V$ of $A$.

In the case when $A$ and $B$ are over the same base, i.e., $M=Q$ and $\pi=\mathrm{id}_{M}$, the Serre spectral sequence associated to the $L\subset A$ has been studied extensively in \cite[Section 7]{Mack05}. For certain proofs, we will use this reference. Also over the same base, this spectral sequence was studied in great detail in the holomorphic setting in \cite{BMRT15} and in more algebraic setting in \cite{Bruz17}.

The Serre spectral sequence in the full generality of \eqref{eq:ss_Liealgebroid_ausgangslage} was first considered in \cite[Section 3]{Brah10}. Using the notion of a generalised representation defined in Section \ref{sec:basics}, in the next theorem we make rigorous the interpretation from \cite[Section 3]{Brah10} of the first two pages.
    

\begin{theorem}\label{theorem:first_page_with_differential}
The Serre spectral sequence associated to $i:L\to A$ satisfies
\[E_2^{p,q}\simeq \mathrm{H}^p(B,\mathrm{H}^q(L,V)).\]
More precisely, the following hold.
\begin{enumerate}[(a)]
\item
We have a canonical isomorphism  (see also \eqref{eq:forms:in:module})
\[E_0^{p,q}\simeq \Omega^p(B,\Omega^{q}(L,V)),\]
where the $C^{\infty}(Q)$-module structure on $\Omega^{q}(L,V)$ is induced by the inclusion \[\pi^*:C^{\infty}(Q)\hookrightarrow C^{\infty}(M).\]
\item The differential $d_0$ is $C^{\infty}(Q)$-linear and, under the isomorphism from (a), it becomes
\[d_0(\omega)(\beta_1,\ldots,\beta_p)=(-1)^p d_L(\omega(\beta_1,\ldots,\beta_p)), \]
for all $\beta_1,\ldots ,\beta_p\in \Gamma(B)$. 
\item For the induced $C^{\infty}(Q)$-module structure on $\mathrm{H}^q(L,V)$, we have isomorphisms
\[E_1^{p,q}\simeq \Omega^p(B,\mathrm{H}^q(L,V)).\]
\item There is a generalised representation of $B$ on the $C^{\infty}(Q)$-module  $\mathrm{H}^q(L,V)$
\[\nabla:\Gamma(B)\times \mathrm{H}^q(L,V)\to \mathrm{H}^q(L,V).\]
\item Under the isomorphism from (c), $d_1:E_1^{p,q}\to E_1^{p+1,q}$ corresponds to the differential calculating cohomology of $B$ with values in the generalised representation from (d).
\end{enumerate}
\end{theorem}
\begin{proof}
The first identification follows from the short exact sequence
\[0\to \mathcal{F}_L^{p+1}\Omega^{p+q}(A,V)\to \mathcal{F}_L^{p}\Omega^{p+q}(A,V)\stackrel{\mathrm{pr}}{\to} \Omega^{p}(B,\Omega^q(L,V))\to 0,\]
where the map $\mathrm{pr}$ acts as
\[\mathrm{pr}(\omega)(\beta_1,\ldots,\beta_p)(\lambda_1,\ldots,\lambda_q):=
\omega(\tilde{\beta}_1,\ldots,\tilde{\beta}_p,\lambda_1,\ldots,\lambda_q),
\]
for all $\beta_1,\ldots, \beta_p\in \Gamma(B)$ and $\lambda_1,\ldots,\lambda_q\in \Gamma(L)$.  Here, for a section $\beta\in\Gamma(B)$, we have denoted by $\tilde{\beta}\in \Gamma(A)$ any lift of $\beta$, i.e.\ $\Pi\circ \tilde{\beta}=\beta\circ \pi$. That $\mathrm{pr}$ is well-defined and that the sequence is exact in the middle follow directly from the description of the filtration given in Lemma \ref{theorem:classical_form}. To show surjectivity, and also for later use, we choose a complement of $L$ in $A$ as in the proof of Lemma \ref{theorem:classical_form}, $A=L\oplus C$. Note that $\Pi$ induces a vector bundle isomorphism $C\simeq \pi^*B$. Using these maps, and Lemma \ref{lemma:tensor_vb} \eqref{App_1} and \eqref{App_3} we obtain isomorphisms
\begin{equation}
\Omega^{\bullet}(A,V)\simeq \bigoplus_{p+q=\bullet}\Omega^p(\pi^*B)\otimes_{C^{\infty}(M)}\Omega^{q}(L,V)\simeq \bigoplus_{p+q=\bullet} \Omega^p(B)\otimes_{C^{\infty}(Q)}\Omega^{q}(L,V).\label{eq:huge_iso}
\end{equation}
Under these isomorphisms, we have that 
 \[\mathcal{F}_L^{p}\Omega^{p+q}(A,V)\simeq \bigoplus_{0\leq i\leq q} \Omega^{p+i}(B)\otimes_{C^{\infty}(Q)}\Omega^{q-i}(L,V),\]
 and $\mathrm{pr}$ becomes the projection onto the first component. This ensures its surjectivity. So, we obtain isomorphisms
 \begin{equation}\label{eq:small_iso}
 E_0^{p,q}\simeq \Omega^p(B)\otimes_{C^{\infty}(Q)}\Omega^q(L,V).
 \end{equation}
 
Next, we need to identify the differential $d_{0}$. First note that for $p=0$ the map $\mathrm{pr}$ is just the pullback map along the Lie algebroid map $i:L\to A$, therefore a cochain map, and so we have a short exact sequence of cochain complexes 
\[0\to (\mathcal{F}_L^{1}\Omega^{\bullet}(A,V), d_A)\to (\Omega^{\bullet}(A,V),d_A)\stackrel{\mathrm{pr}}{\to} (\Omega^{\bullet}(L,V),d_L)\to 0.\]
This implies that the operator $d_0$ on $E_0^{0,\bullet}$ corresponds to $d_L$. 

From the diagram \eqref{eq:ss_Liealgebroid_ausgangslage} it follows that the anchor map of $L$ maps to $\ker d\pi$. This implies that the Lie bracket on $\Gamma(L)$ and the representation on $\Gamma(V)$ are $C^{\infty}(Q)$-linear. Therefore, the differential $d_L$ on $\Omega^{\bullet}(L,V)$ is indeed $C^{\infty}(Q)$-linear. 

Next, let us note that the isomorphism from \eqref{eq:huge_iso}
is $\Omega(B)$-linear, where the multiplication on the right is the obvious one, and the one on the left uses the map $\Pi:A\to B$ 
\[\omega\cdot \eta:=\Pi^*(\omega)\wedge \eta,\]
for $\omega\in \Omega(B)$ and $\eta\in \Omega(A,V)$. Moreover, since $\Pi$ is a Lie algebroid map, it follows that the $\Omega(B)$-module structure is compatible with the differentials 
\begin{equation}\label{eq:derivation:d_A_d_B}
d_A(\omega\cdot  \eta)=d_B(\omega)\cdot  \eta+(-1)^p\omega\cdot d_A(\eta),
\end{equation}
where $\omega\in \Omega^{p}(B)$. This fact and the description of the filtration imply that the differential $d_0$ on $E_0$ is $\Omega^{p}(B)$-linear in the following sense:
\[d_0(\omega\cdot \eta)=(-1)^p\omega\cdot d_0(\eta).\]
This implies that, under the isomorphism \eqref{eq:small_iso}, $d_0$ becomes:
\[(-1)^p\mathrm{id}\otimes d_L :\Omega^p(B)\otimes_{C^{\infty}(Q)}\Omega^q(L,V)\to\Omega^p(B)\otimes_{C^{\infty}(Q)}\Omega^{q+1}(L,V),\]
which is equivalent to the formula given in (b).

Item (c) follows from Lemma \ref{lemma:tensor_vb} \eqref{App_4}
\[E_1^{p,q}\simeq \mathrm{H}^{q}(\Omega^p(B)\otimes_{C^{\infty}(Q)}\Omega^{\bullet}(L,V), \mathrm{id}\otimes d_L)\simeq \Omega^p(B)\otimes_{C^{\infty}(Q)}\mathrm{H}^{q}(L,V).\]

We start by calculating $d_1:E_1^{0,q}\to E_1^{1,q}$. Let $c\in \mathrm{H}^{q}(L,V)$ with representative a closed $q$-from $\eta\in \Omega^{q}(L,V)$. Let $\tilde{\eta}\in \Omega^{q}(A,V)$ be an extension of $\eta$. Then we have that 
\[d_A\tilde{\eta}\in \mathcal{F}^{1}\Omega^{q+1}(A,V),\]
and $d_1c\in \Omega^1(B,\mathrm{H}^q(L,V))$ can be calculated as
\begin{equation}\label{eq:def_generalised_connection}
    d_1c(\beta)=[i^*(\mathrm{i}_{\tilde{\beta}}d_A\tilde{\eta})]\in \mathrm{H}^q(L,V),
\end{equation}
where $\tilde{\beta}\in \Gamma(A)$ is a lift of $\beta\in \Gamma(B)$, $i^*:\Omega^{\bullet}(A,V)\to \Omega^{\bullet}(L,V)$ is the pullback, and where we note that $i^*(\mathrm{i}_{\tilde{\beta}}d_A\tilde{\eta})$ is $d_L$-closed, and so it defines a cohomology class in $\mathrm{H}^{q}(L,V)$. We will exploit the fact that this operation is indeed well-defined, i.e.\ independent of the choices of extension $\tilde{\eta}$ and lift $\tilde{\beta}$. Define the operator from item (d) via the formula
\[\nabla:\Gamma(B)\times \mathrm{H}^{q}(L,V)\to \mathrm{H}^q(L,V),\qquad \nabla_\beta c:=d_1c(\beta).\]
It is easy to see that the operator is $C^{\infty}(Q)$-linear in $\beta$. For the other component, to simplify the computation, choose $\tilde{\eta}$ and $\tilde{\beta}$ such that $\mathrm{i}_{\tilde{\beta}}\tilde{\eta}=0$. Then, for $f\in C^{\infty}(Q)$, we have that
\[\mathrm{i}_{\tilde{\beta}}d_A(\widetilde{f\eta})=\mathrm{i}_{\tilde{\beta}}d_A(\pi^*(f)\tilde{\eta})=
\mathrm{i}_{\tilde{\beta}}(\pi^*(d_B f)\wedge \tilde{\eta})+\mathrm{i}_{\tilde{\beta}}(\pi^*(f)d_A \tilde{\eta})=
\pi^*(\mathscr{L}_{\sharp\beta}f) \tilde{\eta}+\pi^*(f)\mathrm{i}_{\tilde{\beta}}d_A \tilde{\eta},
\]
which yields the second condition in \eqref{eq:rep:1}. To show \eqref{eq:rep:2}, note that we can also write 
\[\nabla_\beta c=[i^*(\mathscr{L}_{\tilde{\beta}}\tilde{\eta})],\qquad \textrm{where}\qquad \mathscr{L}_{\tilde{\beta}}=\mathrm{i}_{\tilde{\beta}}\circ d_{A}+ d_{A}\circ \mathrm{i}_{\tilde{\beta}}.\]
The commutator formula 
\[\mathscr{L}_{\tilde{\beta_1}}\circ \mathscr{L}_{\tilde{\beta_2}}-\mathscr{L}_{\tilde{\beta_2}}\circ \mathscr{L}_{\tilde{\beta_1}}=\mathscr{L}_{[\tilde{\beta}_1,\tilde{\beta_2}]},\]
and the fact that $[\tilde{\beta}_1,\tilde{\beta_2}]$ is a lift of $[\beta_1,\beta_2]$ yield now \eqref{eq:rep:2}. Thus, $\nabla$ defines a generalised representation of $B$ on $\mathrm{H}^{q}(L,V)$.

To obtain (e), we need to show that the map corresponding to $d_1:E_1^{p,q}\to E_1^{p+1,q}$ under the isomorphism from (c) coincides with the differential calculating Lie algebroid cohomology
\[d_B:\Omega^{p}(B, \mathrm{H}^q(L,V))\to \Omega^{p+1}(B, \mathrm{H}^q(L,V)).\]
By the definition of $\nabla$, this holds in degree $p=0$. Next, one can easily show that both operators satisfy the derivation rule with respect to $\Omega(B)$ \eqref{eq:derivation:d_A_d_B}. By also using the isomorphism \eqref{eq:iso:forms:tensor}, these two properties imply that the differentials must coincide in all degrees $p\geq 0$. 
\end{proof}

\begin{remark}\rm\label{remark:nice:setting}
In the setting of Theorem \ref{theorem:first_page_with_differential}, a natural question is whether the generalised representation $\mathrm{H}^q(L,V)$ comes from a classical representation. A candidate for the vector bundle is obtained as follows. First note that, because the anchor of $L$ is tangent to the fibres of $\pi$, for any $x\in Q$, $L_x:=L|_{\pi^{-1}(x)}$ is a Lie subalgebroid of $L$ and $V_x:=V|_{\pi^{-1}(x)}$ is endowed with the pullback representation. Consider the collection of vector spaces  $\mathcal{H}^{\bullet}(L,V)\to Q$, 
\[
\mathcal{H}^{\bullet}(L,V):=\sqcup_{x\in Q}\mathrm{H}^{\bullet}(L_x,V_x).\]
In some special, but interesting cases, which will be discussed in the sequel, $\mathcal{H}^{\bullet}(L,V)\to Q$ is a smooth (finite dimensional) vector bundle, and we have an isomorphism \[\mathrm{H}^{\bullet}(L,V)\simeq \Gamma(\mathcal{H}^{\bullet}(L,V)),\]
given as follows: to $c\in \mathrm{H}^{\bullet}(L,V)$, we assign the section  
\[Q\ni x\mapsto \iota_x^*c\in \mathrm{H}^{\bullet}(L_x,V_x),\]
where $\iota_x:L_x\to L$ denotes the inclusion. 

These properties might fail to hold for various reasons: the fibres $\mathrm{H}^{\bullet}(L_x,V_x)$ are infinite dimensional, or their dimension varies with the point, or not every element in $\mathrm{H}^{\bullet}(L_x,V_x)$ can be extended to an element in $\mathrm{H}^{\bullet}(L,V)$, etc. However, when the properties do hold, then Theorem \eqref{theorem:first_page_with_differential} implies that $\mathcal{H}^{\bullet}(L,V)$ is a classical representation of $B$, and that
\[E^{p,q}_2\simeq \mathrm{H}^{p}(B,\mathcal{H}^{q}(L,V)).\]
\end{remark}

\begin{example}
Let $\gg$ be a Lie algebra, $V$ be a representation of $\gg$, and $\mathfrak{h}\subset \gg$ an ideal. As in \cite{HoSe53}, the second page of the Serre spectral sequence of the inclusion $\hh\subset \gg$ is 
\[E_2^{p,q}\simeq \mathrm{H}^p(\mathfrak{g}/\mathfrak{h},\mathrm{H}^{q}(\mathfrak{h},V)).\]  
\end{example}

\begin{example}
Let $A\Ato M$ be a transitive Lie algebroid, which means that the anchor $\sharp:A\to TM$ is a surjective. Then we have a short exact sequence of Lie algebroids over $M$
\[0\to L\to A\to TM\to 0,\]
where $L:=\ker \sharp$ is called the isotropy bundle of $A$. The corresponding Serre spectral sequence is discussed in detail in \cite[Section 7]{Mack05}. We have that $L\Ato M$ is a locally trivial bundle of Lie algebras. Moreover, for any representation $V$ of $A$, we can locally trivialise $L$ and make its action on $V$ constant at the same time. Using this, one obtains the setting of Remark \ref{remark:nice:setting}, i.e.\ 
$\mathcal{H}^{\bullet}(L,V)\to M$ is a vector bundle with a flat connection, and there is a canonical isomorphism $\Gamma(\mathcal{H}^{\bullet}(L,V))\simeq \mathrm{H}^{\bullet}(L,V)$. Therefore, the second page of the spectral sequence contains the cohomology of $M$ with coefficients in this flat bundle \cite[Theorem 7.4.5]{Mack05}
\[E_2^{p,q}\simeq \mathrm{H}^{p}(M,\mathcal{H}^q(L,V)).\]
\end{example}

\subsection{The Leray-Serre spectral sequence}\label{subsection:Leray-Serre}

A surjective submersion $\pi:M\to Q$ yields a short exact sequence of Lie algebroids
\begin{equation}\label{eq:ss_Liealgebroid:A=TM,submersion}
\begin{tikzpicture}[baseline=(current bounding box.center)]
\usetikzlibrary{arrows}
\node (1) at (0,1) {$ 0 $};
\node (2) at (1.5,1) {$ \ker d \pi $};
\node (3) at (3,1) {$ TM $};
\node (4) at (4.5,1) {$ TQ $};
\node (5) at (6,1) {$ 0 $};

\node (w) at (1.5,0) {$ M $};
\node (e) at (3,0) {$ M $};
\node (r) at (4.5,0) {$ Q $};
\draw[-Implies,double equal sign distance]
(2) -- (w);
\draw[-Implies,double equal sign distance]
(3) -- (e);
\draw[-Implies,double equal sign distance]
(4) -- (r);
\path[->]
(1) edge node[]{$  $} (2)
(2) edge node[above]{$ i $} (3)
(3) edge node[above]{$ d \pi $} (4)
(4) edge node[]{$  $} (5)
(w) edge node[above]{$\mathrm{id}_M $} (e)
(e) edge node[above]{$ \pi $} (r);
\end{tikzpicture}
\end{equation}
As remarked in Example \ref{example:Leray-Serre1}, the spectral sequence associated to the subalgebroid $\ker d \pi$ of $TM$ is the classical Leray–Serre spectral sequence in de Rham cohomology, which was worked out for example in \cite{Hatt60}. We will discuss this construction here in detail, and in the following two subsections, we will discuss two different extensions in the setting of Lie algebroids.

The cohomology $\mathrm{H}^{\bullet}(\ker d \pi)$ is the foliated cohomology of the foliation on $M$ induced by $\pi$. The bundle $\mathcal{H}^{q}(\ker d\pi)\to Q$ has as fibres the de Rham cohomology of the fibres of $\pi$
\[\mathcal{H}^{q}(\ker d\pi)_x=\mathrm{H}^{q}(\pi^{-1}(x)).\] 

Under appropriate extra conditions, the properties from Remark \ref{remark:nice:setting} hold in this setting.

\begin{theorem}\label{theorem:serre-leray-ss-TQ}
	Assume that $\pi:M\to Q$ is a locally trivial fibre bundle with typical fibre a manifold $F$. If $ \mathrm{H}^q(F) $ is finite dimensional then $\mathcal{H}^q(\ker d\pi)$ is a smooth vector bundle with
    \begin{equation}\label{eq:identify_Hq_with_sections}
        \mathrm{H}^q(\ker d\pi)\simeq \Gamma(\mathcal{H}^q(\ker d\pi)).
    \end{equation}
Therefore, the second page of the Serre spectral sequence of $\ker d\pi$ is isomorphic to the cohomology of $Q$ with twisted coefficients in the flat bundle $\mathcal{H}^q(\ker d \pi)\to Q$
\[E_2^{p,q}\simeq \mathrm{H}^p(Q,\mathcal{H}^q(\ker d \pi)).\]
\end{theorem}

To prove Theorem \ref{theorem:serre-leray-ss-TQ}, we first show that $\mathcal{H}^q(\ker d\pi)$ has a smooth vector bundle structure which carries the so-called Gauss-Manin flat connection. Then we prove that \eqref{eq:identify_Hq_with_sections} holds and that the generalised representation is induced by the Gauss-Manin connection.

\begin{lemma}\label{lemma:H(kerdpi)_locally_trivial}
    Let $\pi:M\to Q$ be a locally trivial fibre bundle with typical fibre $F$. 
    \begin{enumerate}
        \item $\mathcal{H}^q(\ker d\pi)\to Q$ is a locally trivial bundle of vector spaces. There exist canonical local trivialisations with locally constant transition functions.
        \item If $ \mathrm{H}^q(F) $ is finite dimensional then $\mathcal{H}^q(\ker d\pi)$ is a smooth vector bundle endowed with a flat connection.
    \end{enumerate}
\end{lemma}
\begin{proof}One obtains local trivialisations for $\mathcal{H}^q(\ker d\pi)$ as follows. Any local trivialisation $\lambda_U: F\times U\diffto \pi^{-1}(U)$ induces a local trivialisation 
    \begin{equation}\label{eq:local:trivia}
    \lambda_U^*:\mathcal{H}^{q}(\ker d\pi )|_{U}\diffto \mathrm{H}^{q}(F)\times U.
    \end{equation}
    Two trivialisations $\lambda_U: F\times U\diffto \pi^{-1}(U)$ and $\lambda_{U'}: F\times U'\diffto \pi^{-1}(U')$ are related by a smooth family of diffeomorphisms on the overlap $\delta_{U,U'}:U\cap U'\to \mathrm{Diff}(F)$, which yields transition maps in cohomology
    \[\delta_{U,U'}^*:\mathrm{H}^{q}(F)\times U\cap U'\diffto  \mathrm{H}^{q}(F)\times U\cap U',\qquad (c,x)\mapsto (\delta_{U,U'}(x)^*(c),x).\]
    Since isotopic diffeomorphisms induce the same map in cohomology, it follows that the transition map $\delta_{U,U'}^*$ is locally constant, giving rise to a flat connection on $\mathcal{H}^q(\ker d\pi)$. If $ \mathrm{H}^q(F) $ is finite dimensional, one obtains a smooth flat vector bundle.
\end{proof}

\begin{definition}
    The flat connection from Lemma \ref{lemma:H(kerdpi)_locally_trivial} is called the \textbf{Gauss-Manin connection}. More precisely, its flat sections are locally constant in the trivialisations \eqref{eq:local:trivia}.
\end{definition}

\begin{lemma}\label{lemma:iso_of_sections}
   Let $\pi:M\to Q$ be a locally trivial fibre bundle with typical fibre $F$. If $ \mathrm{H}^q(F) $ is finite dimensional then the assignment that sends $[\omega]\in \mathrm{H}^q(\ker d\pi) $ to the section of $\mathcal{H}^q(\ker d \pi)$, 
    \begin{equation}\label{eq:assignment_class_to_section}
        Q\ni x\mapsto [\iota_x^\ast \omega]\in \mathrm{H}^q(\pi^{-1}(x)),
    \end{equation}
    where $\iota_x: \pi^{-1}(x)\to M$ is the inclusion, is an isomorphism of $C^\infty(Q)$-modules 
    \[ \mathrm{H}^q(\ker d \pi)\simeq \Gamma(\mathcal{H}^q(\ker d \pi)). \]
    Under this identification, the generalised connection on $\mathrm{H}^q(\ker d \pi)$ from Theorem \ref{theorem:first_page_with_differential} (d) corresponds to the Gauss-Manin connection.
\end{lemma}
\begin{proof}
    Note that \eqref{eq:assignment_class_to_section} gives a set-theoretic section of the vector bundle $\mathcal{H}^q(\ker d \pi)$, and that the assignment is compatible with the $C^\infty(Q)$-module structure. 
    
    To show that \eqref{eq:assignment_class_to_section} is an isomorphism we follow the arguments of \cite[Lemma 3]{CrMa15}. Fix a basis $\{e_i\}$ of $\mathrm{H}_q(F)$, with dual basis $ \{e^i\}$ of $\mathrm{H}^q(F)\simeq \mathrm{H}_q(F)^*$. Using a trivialisation $\pi^{-1}(U)\simeq F\times U\to U$ we obtain the flat local frame $\{  \underline{e}^i \}$ for $\mathcal{H}^q(\ker d\pi)|_U\simeq \mathrm{H}^q(F)\times U\to U$. Then, for $[\omega]\in \mathrm{H}^q(\ker d\pi)$, the coefficients of the section \eqref{eq:assignment_class_to_section} in this frame are given by integrating $\omega$ over representatives $\sigma_i$ of $e_i$, i.e.\ $x\mapsto \int_{\sigma_i}\iota_{x}^*\omega$. Smoothness of $\omega$ implies that these coefficients are smooth. Thus, \eqref{eq:assignment_class_to_section} is well-defined. 
 
    Local injectivity follows from \cite[Corollary 2]{CrMa15}, which shows that if $[\iota_x^\ast\omega]=0$ for all $x\in U$ then there exists a smooth family of primitives for $\iota^\ast_x\omega$, i.e.\ $[\omega]=0$ is exact. To go global, note that we can glue local primitives using a partition of unity on $Q$ subordinate to local trivialisations.
    
    Similarly, it suffices to show local surjectivity. Note that constant sections are in the image of the map \eqref{eq:assignment_class_to_section}, because they are obtained via pulling back by the Lie algebroid map $d\mathrm{pr}:\ker d \pi|_U\simeq TF\times U \to TF$. These sections generate everything. 
Also note that the classes obtained by the pullback along $d\mathrm{pr}$ are in fact restrictions of de Rham classes on $\pi^{-1}(U)$, and therefore they are flat elements of $\mathrm{H}^{q}(\ker d \pi|_{\pi^{-1}}(U))$. Therefore, the two connections coincide. 
\end{proof}

\begin{remark}\rm\label{remark:kerdpi_in_infty_dim} In the case when the fibres of the locally trivial fibration $\pi:M\to Q$ do not have finite dimensional cohomology, note that Lemma \ref{lemma:iso_of_sections} still provides local trivializations for $\mathcal{H}^q(\ker d\pi)\to Q$ with locally constant transition maps. A direct extension of Theorem \ref{theorem:serre-leray-ss-TQ} to this setting would require a notion of smooth sections of this infinite dimensional bundle. Instead, we explain here a different approach, based on the homology bundle
\[\mathcal{H}_q(\ker d \pi)\to Q, \qquad 
\mathcal{H}_q(\ker d \pi)_x:=\mathrm{H}_q(\pi^{-1}(x),\mathbb{Z}).
\]
As in Lemma \ref{lemma:H(kerdpi)_locally_trivial}, we can build local trivialisations 
\[\mathcal{H}_q(\ker d \pi)|_U\simeq \mathrm{H}_q(F,\mathbb{Z})\times U,
\]
for which the transition maps come from pushforwards along diffeomorphisms, hence are locally constant group automorphisms. Therefore,
$\mathcal{H}_q(\ker d \pi)$ is a smooth, locally trivial bundle of discrete groups over $Q$. Consider the space of smooth 1-cocycles on $\mathcal{H}_q(\ker d \pi)$
\[Z^1(\mathcal{H}_q(\ker d \pi))=\big\{\varphi\in C^{\infty}(\mathcal{H}_q(\ker d \pi))\, |\, \varphi(c_1+c_2)=\varphi(c_1)+\varphi(c_2), \, c_i\in \mathrm{H}_q(\pi^{-1}(x),\mathbb{Z})\big\}. \]
We have the following version of Theorem \ref{theorem:serre-leray-ss-TQ}. 
\begin{theorem}\label{theorem:kerdpi_in_infty_dim}
Pullback to fibres followed by integration yields an isomorphism
\[\mathrm{H}^q(\ker d \pi)\simeq Z^1(\mathcal{H}_q(\ker d \pi)), \qquad [\omega]\mapsto \Big( \mathrm{H}_q(\pi^{-1}(x),\mathbb{Z})\ni c \mapsto \int_c\iota_x^*\omega\Big).\]
\end{theorem}
The proof follows the same lines as that of Lemma \ref{lemma:iso_of_sections} and is also based on 
\cite[Lemma 3]{CrMa15}, which gives the result for a local trivialisation $F\times U\to U$. Namely, the cited result shows that pullback followed by integration gives an isomorphism
\[\mathrm{H}^{q}(TF\times U)\simeq \mathrm{Hom}_{\mathbb{Z}}(\mathrm{H}_q(F,\mathbb{Z}),C^{\infty}(U)),\]
and clearly, the second set can be regarded as $Z^1(\mathrm{H}_q(F,\mathbb{Z})\times U)$.
\end{remark}

\subsection{Pullback Lie algebroids}\label{sec:pullback_liealgebroids}
The Leray-Serre spectral sequence can be generalised to the following setting. Consider the pullback Lie algebroid $\pi^!B\Ato M$ of a Lie algebroid $B\Ato Q$ along a surjective submersion $\pi: M\to Q$ (recall \eqref{eq:pullback_def}). This fits into a short exact sequence
\begin{equation}\label{eq:ss_Liealgebroid:pullback,submersion}
\begin{tikzpicture}[baseline=(current bounding box.center)]
\usetikzlibrary{arrows}
\node (1) at (0,1) {$ 0 $};
\node (2) at (1.5,1) {$ \ker d\pi$};
\node (3) at (3,1) {$ \pi^! B $};
\node (4) at (4.5,1) {$ B $};
\node (5) at (6,1) {$ 0 $};

\node (w) at (1.5,0) {$ M $};
\node (e) at (3,0) {$ M $};
\node (r) at (4.5,0) {$ Q $};
\draw[-Implies,double equal sign distance]
(2) -- (w);
\draw[-Implies,double equal sign distance]
(3) -- (e);
\draw[-Implies,double equal sign distance]
(4) -- (r);
\path[->]
(1) edge node[]{$  $} (2)
(2) edge node[above]{$ i $} (3)
(3) edge node[above]{$ \pi^! $} (4)
(4) edge node[]{$  $} (5)
(w) edge node[above]{$\mathrm{id}_M $} (e)
(e) edge node[above]{$ \pi $} (r);
\end{tikzpicture}
\end{equation}

Using Theorems \ref{theorem:first_page_with_differential} and \ref{theorem:serre-leray-ss-TQ} we obtain the following.

\begin{theorem}\label{theorem:E_2_pullback_LA}
    Let $B\Ato Q$ be a Lie algebroid over the base of a fibre bundle $\pi: M\to Q$ with typical fibre $F$. Assume that $\mathrm{H}^\bullet(F)$ is finite dimensional. The Serre spectral sequence 
   associated to $\ker d\pi\subset\pi^! B$ converges to $\mathrm{H}^\bullet(\pi^! B)$ and satisfies
    \[ E_2^{p,q}\simeq\mathrm{H}^p(B,\mathcal{H}^q(\ker d\pi)), \]
    where the representation of $B$ on $\mathcal{H}^q(\ker d\pi)$ is the Gauss-Manin connection factored through the anchor, i.e.\ $\nabla_b=\nabla_{\sharp b}^{GM}$.
\end{theorem}
\begin{proof}
    Using Theorem \ref{theorem:serre-leray-ss-TQ} all that is left to show is the statement about the representation of $B$. We argue locally; $\pi^{-1}(U)\simeq F\times U\stackrel{\mathrm{pr}}{\to}U$. If $\theta\in \Omega^\bullet(F)$ is closed then $[\mathrm{pr}^\ast \theta]$ is a flat section of $\mathcal{H}^\bullet(\ker d \pi)|_{U}$ for the Gauss-Manin connection, and so also flat for the $B$-connection $\nabla^{GM}_{\sharp}$. On the other hand, note that 
      \[\tilde{\eta}:=(\sharp_{\pi^!B})^\ast \mathrm{pr}^\ast \theta\in \Omega^\bullet(\pi^! B|_{\pi^{-1}(U)})\]
     satisfies $d_{\pi^! B}\tilde{\eta}=0$. Therefore, the definition of the representation of $B$ from \eqref{eq:def_generalised_connection} implies that $[\mathrm{pr}^*\theta]$ is flat also with respect to $\nabla$. Using that sections of the form $[\mathrm{pr}^*\theta]$ span all sections of $\mathcal{H}^{\bullet}(\ker d \pi)|_U$ and the Leibniz rule \eqref{eq:rep:1}, we find that the two $B$-connections coincide.
\end{proof}

When the cohomology of the typical fibre is fairly simple, one can get more precise results. The following was obtained in \cite[Theorem 2]{Crai03}, using the same spectral sequence.

\begin{corollary}
    Let $B\Ato Q$ be a Lie algebroid and $\pi: M\to Q$ fibre bundle such that the typical fibre $F$ is $k$-connected, i.e.\ has cohomology
    \[ \mathrm{H}^q(F)=\begin{cases}
        \mathbb{R} &\text{ if }q=0\\
        0 &\text{ if }1\leq q\leq k.
    \end{cases} \]
Then the map 
    \[ (\pi^!)^\ast: \mathrm{H}^{q}(B)\to \mathrm{H}^{q}(\pi^! B) \] 
is an isomorphism for $q=0,\dots,k$ and is injective for $q=k+1$.
\end{corollary}

\begin{example}
    Another class of fibres with relatively simple cohomology are spheres $F=S^n$. In this case, we obtain a Lie algebroid version of the classical Gysin sequence. On the second page of the associated Serre spectral sequence only the zeroth and $n$-th row are non-trivial, and the only non-trivial differential is on page $E_{n}$. One obtains a long exact sequence (see for example the analogous discussion in \cite{BoTu82}, before Proposition 14.33)
\[\ldots {\longrightarrow} \mathrm{H}^k(\pi^! B)
\stackrel{a}{\longrightarrow} \mathrm{H}^{k-n}(B,\mathcal{H}^n(\ker d\pi))
\stackrel{d_n}{\longrightarrow}\mathrm{H}^{k+1}(B)
\stackrel{b}{\longrightarrow}\mathrm{H}^{k+1}(\pi^! B) 
{\longrightarrow}\ldots
\]
The maps can be described as follows. First, note that we have an isomorphism of flat bundles
\[ \mathcal{H}^n(\ker d\pi) \simeq o(M)\times_{\mathbb{Z}_2}\mathbb{R},\]
where $o(M)\to Q$ is the double cover corresponding to the two possible orientations on the fibres of $M\to Q$. Under this identification, $a$ is the map integrating along the fibres, $d_n=\sharp^\ast e\wedge$, where $e\in \mathrm{H}^{n+1}(Q)$ is the Euler class of the sphere bundle, and $b=(\pi^!)^\ast$ is the pullback. See \cite{Sch24} for details.
\end{example}

\subsection{Submersions by Lie algebroids}\label{sec:submersions_by_LAs}

A different class of Lie algebroid extensions for which the Leray-Serre spectral sequence admits an interesting generalisation are the \textbf{submersions by Lie algebroids}, introduced and studied recently in \cite{Frej19}. These are pairs $(A,\pi)$ composed of a Lie algebroid $A\Ato M$ and a surjective submersion $\pi: M\to Q$ such that $d\pi\circ\sharp$ is surjective. Hence, if $L:=\ker (d\pi\circ \sharp)$, 
there is a short exact sequence of Lie algebroids
\begin{equation}\label{eq:ss_Liealgebroid:submersion_by_liealgebroids}
\begin{tikzpicture}[baseline=(current bounding box.center)]
\usetikzlibrary{arrows}
\node (1) at (0,1) {$ 0 $};
\node (2) at (1.8,1) {$ L $};
\node (3) at (3.6,1) {$ A $};
\node (4) at (5.4,1) {$ TQ $};
\node (5) at (7.2,1) {$ 0 $};

\node (w) at (1.8,0) {$ M $};
\node (e) at (3.6,0) {$ M $};
\node (r) at (5.4,0) {$ Q $};
\draw[-Implies,double equal sign distance]
(2) -- (w);
\draw[-Implies,double equal sign distance]
(3) -- (e);
\draw[-Implies,double equal sign distance]
(4) -- (r);
\path[->]
(1) edge node[]{$  $} (2)
(2) edge node[above]{$ i $} (3)
(3) edge node[above]{$ d\pi\circ\sharp $} (4)
(4) edge node[]{$  $} (5)
(w) edge node[above]{$\mathrm{id}_M $} (e)
(e) edge node[above]{$ \pi $} (r);
\end{tikzpicture}
\end{equation}

For locally trivial submersion by Lie algebroids (recalled below), a Leray-type spectral sequence for $\mathrm{H}^{\bullet}(A,V)$ was constructed in \cite[Section 5]{Frej19}. In particular,  the second page of this spectral sequence contains the \v{C}ech cohomology of $Q$ with values in a certain presheaf.
We show that the sheafification of this presheaf appears naturally in our setting of the Serre spectral sequence associated to the extension, and we explain in Theorem \ref{theorem:comparing_sp_sq} how the two constructions are related. 

Note that, for a submersion by Lie algebroids $(A,\pi)$, the fibres of $\pi$ are transverse submanifolds for $A$. By the normal form theorem \cite[Theorem 4.1]{BLM19} each fibre $\pi^{-1}(x)$ has a tubular neighbourhood $\mathrm{pr}: \mathcal{U}\to \pi^{-1}(x)$, such that there is an isomorphism of Lie algebroids $A|_{\mathcal{U}}\simeq \mathrm{pr}^! L_x$ covering $\mathrm{id}_{\mathcal{U}}$, where $L_x:=L|_{\pi^{-1}(x)}$. 
If we could take $\mathcal{U}$ to be part of a trivialisation of $\pi$, i.e.\ $\mathcal{U}=\pi^{-1}(U)$ and $(\pi,\mathrm{pr}):\pi^{-1}(U)\diffto U\times \pi^{-1}(x)$ is a diffeomorphism (e.g.\ if $\pi$ is proper) then we would obtain a local trivialisation of $A$
\begin{equation}\label{eq:local:triv}
A|_{\pi^{-1}(U)}\simeq TU\times L_x.
\end{equation}
If $Q$ can be covered by such local trivialisations, the submersion by Lie algebroids \eqref{eq:ss_Liealgebroid:submersion_by_liealgebroids} is called \textbf{locally trivial}. A condition equivalent to local triviality is the existence of \textbf{complete Ehresmann connections} \cite[Theorem 3]{Frej19}. This is a splitting of vector bundles $A=L\oplus C$ such that the induced horizontal lift 
\[ hor: \mathfrak{X}(Q)\to \Gamma(C)\subset \Gamma(A) \]
maps complete vector fields to complete sections of $A$ (i.e.\ whose anchor is complete). The resulting parallel transport can be used to locally trivialise $A$. 

Local trivialisations \eqref{eq:local:triv} yield local trivialisations of the ``vector bundle'' $\mathcal{H}^q(L,V)\to Q$ with locally constant transition functions \cite[Lemma 4]{Frej19}. However, since the fibres of $\mathcal{H}^q(L,V)\to Q$ will in general be infinite dimensional (e.g.\ if $L$ is the zero Lie algebroid), the interpretation from Remark \ref{remark:nice:setting} of $\mathrm{H}^\bullet(L,V)$ as sections of a vector bundle is not suitable in this general setting. Instead, we show that elements of $\mathrm{H}^\bullet(L,V)$ that are flat under the representation of $TQ$ constitute a sheaf, which is the sheafification of the presheaf appearing in \cite[Section 5]{Frej19}. Using this, we describe the second page of the Serre spectral sequence in terms of cohomology of $Q$ with values in this sheaf. 

\begin{lemma}\label{lemma:subm_by_LA:sheaf_resolution}
    The assignment, sending an open set $U\subset Q$ to
    \[ \mathcal{S}^q_{L,V}(U)=\big\{ c\in \mathrm{H}^q(L|_{\pi^{-1}(U)},V|_{\pi^{-1}(U)})\, |\,  \nabla c=0  \big\}, \]
    constitutes a sheaf on $Q$, where $\nabla$ is the generalised representation from Theorem \ref{theorem:first_page_with_differential} (d). Moreover, if $A$ is locally trivial, then
    \begin{equation}\label{eq:resolution_by_fine_sheaves}
        0{\longrightarrow} \mathcal{S}^q_{L,V}{\longrightarrow}\Omega^0_Q(\cdot,\mathrm{H}^q(L,V))\stackrel{d^{\nabla}}{\longrightarrow}\Omega^1_Q(\cdot,\mathrm{H}^q(L,V))\stackrel{d^{\nabla}}{\longrightarrow} \ldots
    \end{equation}
    is a resolution of $\mathcal{S}^q_{L,V}$ by fine sheaves.
\end{lemma}
\begin{proof}
 We first check that the presheaf $\mathcal{S}^q_{L,V}$ is complete. 

Let $\{U_i\}_i$ be a family of open subsets of $Q$ and let $U:=\cup_i U_i$. Consider $c\in \mathcal{S}^q_{L,V}(U)$ such that $c|_{U_i}=0$, for all $i$. We want to show that $c=0$. Let $\omega\in \Omega^q(L|_{\pi^{-1}(U)},V|_{\pi^{-1}(U)})$ be a representative of $c$. Then $\omega|_{\pi^{-1}(U_i)}=d_L \eta_i$, for some with $\eta_i\in \Omega^{q-1}(L|_{\pi^{-1}(U_i)},V|_{\pi^{-1}(U_i)})$. Choose a partition of unity $\{ \chi_i\}_i$ on $U$ subordinate to the cover $\{U_i\}_i$. Define the $q-1$-form 
\[\eta:=\sum_i \pi^\ast\chi_i\, \eta_i\in \Omega^{q-1}(L|_{\pi^{-1}(U)},V|_{\pi^{-1}(U)}).\]
Since $d_L$ is $C^{\infty}(Q)$-linear, we have that 
\[d_L\eta=\sum_i \pi^\ast\chi_i\, d_L\eta_i=\sum_i \pi^\ast\chi_i\, \omega=\omega,\]
hence indeed $c=0$. 

Next, consider classes $c_i\in \mathcal{S}^q_{L,V}(U_i)$ such that
    \[ c_i|_{U_i\cap U_j}=c_j|_{U_i\cap U_j}.\]
    To glue these elements to a global section over $U$, let $\omega_i\in \Omega^q(L|_{\pi^{-1}(U_i)},V|_{\pi^{-1}(U_i)})$ be representatives of $c_i$. 
    Then we can write \[\omega_i|_{U_i\cap U_j}-\omega_j|_{U_i\cap U_j}=d_L\theta_{i,j},\quad \textrm{with}\quad \theta_{i,j}\in \Omega^{q-1}(L|_{\pi^{-1}(U_{i}\cap U_j)},V|_{\pi^{-1}(U_{i}\cap U_j)}).\] 
    Define
    \[ \omega=\sum_i \pi^\ast\chi_i\, \omega_i\in \Omega^q(L|_{\pi^{-1}(U)},V|_{\pi^{-1}(U)}).\]
    Using that $d_L$ is $C^{\infty}(Q)$-linear, it follows that $\omega$ is closed and that 
    \[\omega|_{U_j}=\sum_i \pi^\ast\chi_i (\omega_j|_{U_{i}\cap U_j}+d_L\theta_{i,j})=\omega_j+d_L(\sum_i \pi^\ast\chi_i\, \theta_{i,j}).\]
Hence $c:=[\omega]\in H^{q}(L|_{\pi^{-1}(U)},V|_{\pi^{-1}(U)})$ satisfies $c|_{U_j}=c_j$. To show that $\nabla c=0$, we use that flatness is a local property. Namely, for any vector field $X\in \Gamma(TU)$ we have that $\nabla_Xc|_{U_i}=\nabla_{X|_{U_i}}c_i=0$, which implies, as in the first part, that $\nabla_Xc=0$. 

We conclude that $\mathcal{S}^q_{L,V}$ is a sheaf.

Exactness of \eqref{eq:resolution_by_fine_sheaves} at $\Omega^0_Q(\cdot,\mathrm{H}^q(L,V))$ is clear. We show exactness in degree $p\geq 1$ and at some point $x\in Q$. First, by the proof of \cite[Lemma 4]{Frej19}, the fact that $A$ admits a complete connection implies that we can simultaneously trivialise $A$ and the representation
\begin{equation}\label{eq:local:triv_with_representation}
(A|_{\pi^{-1}(U)}, V|_{\pi^{-1}(U)})\simeq (TU\times L_x, \mathrm{pr}_2^!V_x),
\end{equation}
where $V_x:=V|_{\pi^{-1}(x)}$, $U$ is a small neighbourhood of $x$, and $\mathrm{pr}_2:U\times \pi^{-1}(x)\to \pi^{-1}(x)$ is the second projection. We shrink $U$ to admit coordinates $\{y_i\}$ in which $U$ corresponds to a ball. Using this isomorphism, we can decompose 
\[\Omega^{\bullet}(A|_{\pi^{-1}(U)},V|_{\pi^{-1}(U)})\simeq \bigoplus_{p+q=\bullet} \Omega^p(U,\Omega^q(\mathrm{pr}_2^!L_x,\mathrm{pr}_2^!V_{x})).\]
 A form $\eta \in\Omega^p(U,\Omega^q(\mathrm{pr}_2^!L_x,\mathrm{pr}_2^!V_{x}))$ can be written uniquely as 
\[ \eta=\frac{1}{p!}\sum_{i_1\ldots i_p}{\eta}_{i_1\ldots i_p} dy_{i_1}\wedge  \ldots \wedge dy_{i_p}, \]
    where the fully skew-symmetric coefficients ${\eta}_{i_1 \ldots i_p} \in \Omega^q(\mathrm{pr}_2^!L_x,\mathrm{pr}_2^!V_{x})$ can be thought of as smooth families of forms on $L_x$ with values in $V_x$
    \[U\ni y\mapsto \eta_{i_1 \ldots i_p}(y)\in \Omega^q(L_x,V_{x}).\]
Under these identifications, the $A$-differential decomposes as $d_A=d_{L_x}+(-1)^qd$, where $d_{L_x}$ has bidegree $(0,1)$ and is the  differential of $L_x$ and $d$ has bidegree $(1,0)$ and is the `de Rham' differential on $U$, i.e.\
\begin{align*}
d_{L_x} \eta&=\frac{1}{p!}\sum_{i_1 \ldots i_p}(d_{L_x}\eta_{i_1 \ldots i_p})dy_{i_1}\wedge \ldots \wedge d y_{i_p}\qquad \textrm{and}\\
d \eta&=\frac{1}{p!}\sum_{i_0 \ldots i_p} \partial_{y_{i_{0}}}(\eta_{i_1 i_2 \ldots i_p})dy_{i_0}\wedge dy_{i_1}\wedge \ldots \wedge d y_{i_p}.
\end{align*}

To show exactness of \eqref{eq:resolution_by_fine_sheaves}, let $c=[\eta]\in \Omega^p(U,\mathrm{H}^q(L,V))$ be such that $d_{\nabla}c=0$, where $U$ is some neighbourhood of $x$. We need to show that, after possibly shrinking $U$, there is some $e\in \Omega^{p-1}(U,\mathrm{H}^q(L,V))$ such that $c=d_{\nabla}e$. So 
 we may assume that there is a trivialisation \eqref{eq:local:triv_with_representation} above $U$ and there are coordinates $\{y_i\}$ centred at $x$ in which $U$ corresponds to a ball. Then, under the above identifications, $d_{\nabla}$ becomes the de Rham differential, so 
\[0=d_{\nabla}c=d_{\nabla}[\eta]=[d_A\eta]=(-1)^q[d\eta]\in \Omega^{p+1}(U,\mathrm{H}^q(L,V)).\]
Hence, any coefficient $(d\eta)_{i_0i_1\ldots i_p}$ of $d\eta$ is $d_{L_x}$ exact, so we can write 
\[d\eta=d_{L_x}\theta,\qquad \textrm{with}\quad \theta\in \Omega^{p+1}(U,\Omega^{q-1}(\mathrm{pr}_2^!L_x,\mathrm{pr}_2^!V_{x})).\] Next, consider the standard homotopy operators from the Poincar\'e Lemma on $U$, corresponding to the contraction $\mu_t(y)=ty$ in the coordinates $\{y_i\}$, i.e.\ if $\xi=\sum_i y_i\partial_{y_i}$, let
\[h \alpha=\int_0^1\frac{1}{t}\mathrm{i}_{\xi} \mu_t^*\alpha \, d t= 
\int_0^1\frac{t^{l-1}}{(l-1)!}\sum_{i_1 \dots i_l}y_{i_1}\alpha_{i_1\ldots i_p}(ty)d y_{i_2}\wedge\ldots \wedge d y_{i_{l}}\, dt.\]
Note that these operators make sense for $\alpha\in \Omega^{l}(U,\Omega^{\bullet}(\mathrm{pr}_2^!L_x,\mathrm{pr}_2^!V_{x}))$, $l\geq 1$, and still satisfy the homotopy relation and, moreover, commute with $d_{L_x}$, i.e.\
\[\alpha=dh \alpha+ hd  \alpha\qquad \textrm{and} \qquad d_{L_x}h\alpha =hd_{L_x}\alpha.\]
This implies that 
\[\eta=dh\eta+hd\eta=dh\eta+h d_{L_x}\theta=dh\eta+d_{L_x}h\theta,\]
and also that
\[d_{L_x}h\eta=hd_{L_x}\eta=0.\]
Thus the element $h\eta\in \Omega^{p-1}(U,\Omega^q(\mathrm{pr}_2^!L_x,\mathrm{pr}_2^!V_{x}))$ defines a class $e:=[h\eta]\in \Omega^{p-1}(U,\mathrm{H}^q(L,V))$ which satisfies
\[d_{\nabla}e=(-1)^q[d h\eta]=(-1)^q[\eta- d_{L_x}h\theta]=(-1)^q[\eta]=(-1)^qc.\]
Therefore $c$ is $d_{\nabla}$ exact. 

Finally, the resolution is by fine sheaves because $\Omega^\bullet(U,\mathrm{H}^q(L,V))$ is a $C^{\infty}(U)$-module.
\end{proof}


\begin{theorem}\label{theorem:page:2:submersion}
    Let $(A\Ato M,\pi:M\to Q)$ be a locally trivial submersion by Lie algebroids, and $V\to M$ a representation of $A$. The spaces on the second page of the associated Serre spectral sequence are isomorphic to the cohomology of $Q$ with coefficients in the sheaf $\mathcal{S}^q_{L,V}$
    \[ E_2^{p,q}\simeq\mathrm{H}^p(Q,\mathcal{S}^q_{L,V}). \]
\end{theorem}

\begin{proof}
    By Lemma \ref{lemma:subm_by_LA:sheaf_resolution} and \cite[Theorem 5.25]{War83} we obtain a canonical isomorphism 
    \[ \mathrm{H}^\bullet(Q,\mathcal{S}^q_{L,V})\simeq \mathrm{H}^\bullet(Q,\mathrm{H}^q(L,V)), \]
    since \eqref{eq:resolution_by_fine_sheaves} is a resolution by fine sheaves. This together with Theorem \ref{theorem:first_page_with_differential} proves the claim.
\end{proof}

\begin{remark}\label{remark:Pedro's_spectral_Sequence}\rm
In \cite{Frej19}, a different spectral sequence was constructed for a submersion by Lie algebroids $(A,\pi)$. Namely, as explained in \cite[Proposition 3]{Frej19}, the \v{C}ech-Lie algebroid double complex of an open cover $\mathcal{U}=\{U_i\}$ of $Q$,
\begin{equation}\label{eq:Cech-Lie_alg}
    C^{p}(\mathcal{U},\Omega^{q}(A,V)):=\prod_{i_0 \ldots i_p}\Omega^q(A|_{\pi^{-1}(U_{i_0<\ldots< i_p})},V|_{\pi^{-1}(U_{i_0\ldots i_p})}),\quad  \mathrm{d}=\delta+(-1)^p d_A,
\end{equation}
    is quasi-isomorphic to $(\Omega^{\bullet}(A,V),d_A)$, and so the associated spectral sequence $\check{E}_{r}^{p,q}(\mathcal{U})$ converges to $\mathrm{H}^{\bullet}(A,V)$. Moreover, the second page is the \v{C}ech cohomology of $\mathcal{U}$ with values in the pushforward presheaf $\mathcal{P}^q=\pi_*(\mathrm{H}^q(A,V))$
    \begin{equation}\label{eq:second_page_cech}
       \check{E}_2^{p,q}(\mathcal{U})\simeq \check{\mathrm{H}}^p(\mathcal{U}, \mathcal{P}^q). 
    \end{equation}
    
    
    Assume now that the submersion by Lie algebroids is \emph{locally trivial}. Then, as shown in \cite[Lemma 4]{Frej19}, $\mathcal{P}^q$ is a locally constant presheaf: if $x\in Q$ and $U$ is a  contractible neighbourhood of $x$, which admits a trivialisation as in \eqref{eq:local:triv_with_representation}, then the inclusion $L_x\hookrightarrow A$ induces an isomorphism 
    \begin{equation}\label{eq:locally:constant:sheaf}
    \mathcal{P}^q(U)=\mathrm{H}^{q}(A|_{\pi^{-1}(U)},V|_{\pi^{-1}(U)})\simeq \mathrm{H}^q(TU\times L_x, \mathrm{pr}^*_2(V_x))\simeq \mathrm{H}^q(L_x, V_x).
    \end{equation}

To relate \cite[Proposition 3]{Frej19} to our Theorem \ref{theorem:page:2:submersion}, we first note the following. 
    
    \begin{lemma}\label{lemma:sheafification}
        If the submersion by Lie algebroids is locally trivial, then the sheafification of the presheaf $\mathcal{P}^q$ is canonically isomorphic to the sheaf $\mathcal{S}^q_{L,V}$.
    \end{lemma}
    \begin{proof}
    For every $U\subset Q$, the pullback along the inclusion $L\hookrightarrow A$ induces a map 
    \begin{equation}\label{eq:map_of_presh}
    \mathrm{H}^{q}(A|_{\pi^{-1}(U)},V|_{\pi^{-1}(U)})
    \to \mathrm{H}^{q}(L|_{\pi^{-1}(U)},V|_{\pi^{-1}(U)}).
    \end{equation}
    From the definition of the connection on $\mathrm{H}^{q}(L,V)$, we see that this map takes values in flat section. Thus, we have a canonical map of presheaves $\mathcal{P}^q\to \mathcal{S}^q_{L,V}$. We show that this is an isomorphism locally, which will imply the claim. Let $x\in Q$. Fix an open subset $U\subset Q$ as in the second half of the proof of Lemma \ref{lemma:subm_by_LA:sheaf_resolution}, i.e.\ $U$ corresponds to a ball centred at $0$ in the coordinates $\{y_i\}$ and we fix an trivialisation over $U$ as in \eqref{eq:local:triv_with_representation}. By \eqref{eq:locally:constant:sheaf}, the composition
     \[\mathrm{H}^{q}(A|_{\pi^{-1}(U)},V|_{\pi^{-1}(U))})\to\mathrm{H}^{q}(L|_{\pi^{-1}(U)},V|_{\pi^{-1}(U)})\to \mathrm{H}^q(L_x, V_x)\]
     is an isomorphism, hence the map \eqref{eq:map_of_presh} is injective. To show surjectivity, fix $c=[\eta]\in \mathrm{H}^{q}(L|_{\pi^{-1}(U)},V|_{\pi^{-1}(U)})$. As in the proof of Lemma \ref{lemma:subm_by_LA:sheaf_resolution}, we regard $\eta\in \Omega^q(\mathrm{pr}_2^!L_x,\mathrm{pr}_2^!V_x)$ as a family of forms $\eta(y)\in \Omega^q(L_x,V_x)$ depending smoothly on $y\in U$ and satisfying $d_{L_x}\eta(y)=0$. On the other hand, we have that 
        \[ 0=d_\nabla [\eta]=(-1)^q[d\eta]=\sum_i [\partial_{y_i}\eta]dy_i. \]
        Thus we can write $\partial_{y_i}\eta=d_{L_x} \theta^i(y)$, for some $\theta^i\in \Omega^{q-1}(\mathrm{pr}_2^!L_x,\mathrm{pr}_2^!V_x)$. Therefore,
        \begin{align*} 
        \eta(y)-\eta(0)&=\int_0^1\frac{d}{dt}\big|_{t=0}\eta(ty) dt=\sum_i \int_0^1y_i\partial_{y_i}\eta(ty) dt=d_{L_x}\sum_i \int_0^1y_i\theta^i(ty) dt.
        \end{align*}
        Hence $c=[\eta(0)]$, which is clearly in the image of the map \eqref{eq:map_of_presh}. Namely, the trivialisation \eqref{eq:local:triv_with_representation} yields a Lie algebroid map $\mathrm{pr}_2:A|_{\pi^{-1}(U)}\to L_x$ which respects representations, and the class $[\mathrm{pr}_2^*\eta(0)]$ is mapped under \eqref{eq:map_of_presh} to $c$. 
    \end{proof}

For \textbf{good covers} $\mathcal{U}$, we obtain that the second page of the spectral sequence $\check{E}_r^{p,q}(\mathcal{U})$ constructed in \cite{Frej19} is isomorphic to the second page of the Serre spectral sequence $E_r^{p,q}$ 
corresponding to the extension \eqref{eq:ss_Liealgebroid:submersion_by_liealgebroids}--both converging to $\mathrm{H}^{\bullet}(A,V)$.

\begin{theorem}\label{theorem:comparing_sp_sq}
Let $\mathcal{U}$ be a good cover of $Q$ (i.e.\ all intersections $U_{i_0\ldots i_p}$ are contractible), so that over each $U_i\in \mathcal{U}$ a trivialisation as in \eqref{eq:local:triv_with_representation} exists. There are isomorphisms
\[\check{E}_2^{p,q}(\mathcal{U})\simeq \check{\mathrm{H}}^p(\mathcal{U},\mathcal{P}^q)
\simeq 
\check{\mathrm{H}}^p(\mathcal{U},\mathcal{S}^q_{L,V})
\simeq \mathrm{H}^p(Q,\mathcal{S}^q_{L,V})\simeq E_{2}^{p,q}.
\]
\end{theorem}
\begin{proof}
The first isomorphism was observed already in \eqref{eq:second_page_cech}. The second isomorphism follows from the proof of Lemma \ref{lemma:sheafification}, which shows that, for any intersection $U_{i_0\ldots i_p}$ of elements in $\mathcal{U}$,
\[\mathcal{P}^q(U_{i_0\ldots i_p})=\mathcal{S}_{L,V}^q(U_{i_0\ldots i_p});\]
hence, the \v{C}ech complexes computing the two cohomology groups coincide. For the third isomorphism, we use Leray's result  \cite{Ler46}, which says that a good cover of a manifold can be used to calculate cohomology with values in a locally constant sheaf (see, e.g.\ \cite[Theorem 15.30]{GaQu22}). The last isomorphism was the content of Theorem \ref{theorem:page:2:submersion}.
\end{proof}
\end{remark}

\subsubsection{Coupling Poisson and Dirac structures}\label{sec:coupling_poisson}

We apply the tools developed in this section to a class of submersions by Lie algebroids coming from Poisson geometry which were introduced by Vorobjev in the study of normal forms for Poisson structures around symplectic leaves. In particular, we recover the description from \cite[Theorem 4.8]{VBVo18} of the first Poisson cohomology group of coupling Poisson structures.

A \textbf{Poisson structure} (see e.g.\ \cite{Wein83}) on a manifold $M$ is a bivector field $w\in \Gamma(\wedge^2TM)$ satisfying $[w,w]=0$ for the Schouten bracket. A Poisson structure yields a Lie algebroid structure on the cotangent bundle $T^*M$ with anchor and bracket given by
\[\sharp\alpha:=w(\alpha, \cdot), \qquad 
[\alpha,\beta]_{w}:=\mathscr{L}_{\sharp \alpha}\beta-\mathrm{i}_{\sharp \beta}d \alpha,\qquad \alpha,\beta\in \Gamma(T^*M).\]
The \textbf{Poisson cohomology} of $(M,w)$ is defined as the Lie algebroid cohomology of $T^*M\Ato M$, and can be computed using the complex of multivector fields and differential $d_{w}:=[w,\cdot]$, i.e.\
\begin{equation}\label{eq:iso:Poisson:cohom}
(\Omega^\bullet(T^*M),d_{T^*M})\simeq (\Gamma(\wedge^\bullet TM),d_{w}).    
\end{equation}

A \textbf{Dirac structure} is a subbundle $D\subset TM\oplus T^*M$ which is Lagrangian for the canonical,  split signature pairing on $TM\oplus T^*M$ and whose space of sections is closed under the Dorfman bracket (for basics on Dirac geometry, see \cite{Cou90,Bur13,CFM21}). Any Dirac structure is a Lie algebroid with the Dorfman bracket. Dirac structures generalise simultaneously Poisson structures and closed forms (via their graphs), as well as foliations ($D=T\mathcal{F}\oplus (T\mathcal{F})^{\circ}$).

A Poisson structure $w$ on the total space of a surjective submersion $\pi:M\to Q$ is called 
\textbf{horizontally nondegenerate} \cite{Vor01,Vor05} if it satisfies 
\[\ker d\pi\oplus \sharp(\ker d \pi)^{\circ}=TM,\]
where $(\ker d \pi)^{\circ}\subset T^*M$ is the annihilator of $\ker d \pi$, i.e.\ the space of `horizontal' 1-forms. 

More generally, a Dirac structure on $D\subset TM\oplus T^*M$ on the total space of a surjective submersion $\pi:M\to Q$ is called horizontally nondegenerate \cite{Wade08}, if it satisfies 
\[D\cap \big(\ker d\pi\oplus (\ker d \pi)^{\circ}\big) =0.\]
As explained in \cite{Vor01,Vor05} for Poisson and in \cite{Wade08} in general, a horizontally nondegenerate Dirac structure $D$ gives rise to the following \textbf{coupling data}:
\begin{enumerate}[(i)]
\item A vertical bivector field $\mathcal{W}\in \Gamma(\wedge^2 \ker d \pi)$;
\item An Ehresmann connection, i.e.\ a subbundle $H\subset TM$ such that $TM=H\oplus \ker d\pi$;
\item A horizontal 2-form $\mathbb{F}\in \Gamma(\wedge^2 (\ker d \pi)^{\circ})$.
\end{enumerate}
The relation between the Dirac structure and the coupling data is 
\begin{equation}\label{eq:D_triple}
D=\{(X,\mathrm{i}_X\mathbb{F})\, |\, X\in H\}\oplus\{(\mathcal{W}^{\sharp}\alpha,\alpha)\, |\, \alpha\in H^{\circ}\}.
\end{equation}
In fact, \cite[Theorem 2.9]{Wade08} shows that this equality yields a one-to-one correspondence between horizontally nondegenerate Dirac structures $D$ and triples $(\mathcal{W},H,\mathbb{F})$ satisfying the following four conditions (see \cite[Proposition 2.4]{Vor05} for the Poisson case, and also \cite{Mar13})
\begin{enumerate}[(1)]
\item $\mathcal{W}$ is a Poisson structure, i.e.\ 
\[[\mathcal{W},\mathcal{W}]=0;\]
\item the parallel transport of $H$ preserves $\mathcal{W}$, i.e.\ 
\begin{equation}\label{eq:W is inv}
\mathscr{L}_{\mathrm{hor}(X)}\mathcal{W}=0,\qquad \textrm{for all } X\in \mathfrak{X}(Q),
\end{equation}
where $\mathrm{hor}(X)\in \Gamma(H)$ denotes the $H$-horizontal lift of $X$;
\item $\mathbb{F}$ is horizontally closed, i.e.\
\[d\mathbb{F}(\mathrm{hor}(X_1),\mathrm{hor}(X_2),\mathrm{hor}(X_3))=0,\qquad \textrm{for all}\ X_1,X_2,X_3\in \mathfrak{X}(Q);\]
\item the curvature of $H$, defined as
\[R_H(\mathrm{hor}(X_1),\mathrm{hor}(X_2)):=
\mathrm{hor}([X_1,X_2])-[\mathrm{hor}(X_1),\mathrm{hor}(X_2)]\in \Gamma(\ker d\pi),\] 
satisfies
\[R_H(\mathrm{hor}(X_1),\mathrm{hor}(X_2))=\mathcal{W}^{\sharp}d(\mathbb{F}(\mathrm{hor}(X_1),\mathrm{hor}(X_2))),\qquad \textrm{for all}\ X_1,X_2\in \mathfrak{X}(Q).\]
\end{enumerate}

The horizontally nondegenerate Dirac structure $D$ is the graph of a Poisson structure $w$ if and only if $\mathbb{F}$ is nondegenerate when regarded as a 2-form on the vector bundle $TM/(\ker d\pi)$.

A horizontally non-degenerate Dirac structure $D$ yields a submersion by Lie algebroids

\begin{equation}\label{eq:Vorobjev:submersion_by_liealgebroids}
\begin{tikzpicture}[baseline=(current bounding box.center)]
\usetikzlibrary{arrows}
\node (1) at (0,1) {$ 0 $};
\node (2) at (1.8,1) {$ (\ker d \pi)^* $};
\node (3) at (3.6,1) {$ D $};
\node (4) at (5.4,1) {$ TQ $};
\node (5) at (7.2,1) {$ 0 $};

\node (w) at (1.8,0) {$ M $};
\node (e) at (3.6,0) {$ M $};
\node (r) at (5.4,0) {$ Q $};
\draw[-Implies,double equal sign distance]
(2) -- (w);
\draw[-Implies,double equal sign distance]
(3) -- (e);
\draw[-Implies,double equal sign distance]
(4) -- (r);
\path[->]
(1) edge node[]{$  $} (2)
(2) edge node[above]{$ i $} (3)
(3) edge node[above]{$ d\pi\circ\sharp $} (4)
(4) edge node[]{$  $} (5)
(w) edge node[above]{$\mathrm{id}_M $} (e)
(e) edge node[above]{$ \pi $} (r);
\end{tikzpicture}
\end{equation}
Here we have used the decomposition \eqref{eq:D_triple} and the following isomorphism 
\[\ker d \pi\circ \sharp=\{(\mathcal{W}^{\sharp}(\alpha),\alpha)\ | \ \alpha\in H^{\circ}\}\simeq H^{\circ}\simeq (\ker d \pi)^*.\] 
Under this identification, the Lie algebroid structure on $(\ker d \pi)^*$ becomes the family of cotangent Lie algebroids $T^*M_x$, $x\in Q$, corresponding to the Poisson manifolds $(M_x,\mathcal{W}_x)$, where $M_x=\pi^{-1}(x)$ and $\mathcal{W}_x=\mathcal{W}|_{M_x}$. 

Since $D$ is transverse to the Dirac structure $\ker d\pi\oplus (\ker d\pi)^{\circ}$, the pairing of the Courant algebroid $TM\oplus T^*M$ gives a canonical isomorphism $D^*\simeq \ker d\pi\oplus (\ker d\pi)^{\circ}$. 
This yields a decomposition of the space of forms on $D$
\[\Omega^{\bullet}(D)=\bigoplus_{p+q=\bullet}\Omega^p(Q,\Gamma(\wedge^q \ker d\pi)).\]
Under this identification, the differential decomposes as $d_D=d^{(0,1)}+d^{(1,0)}+d^{(2,-1)}$, where
\begin{align*}
(d^{(0,1)}\eta)(X_1\ldots X_p)=&(-1)^p[\mathcal{W},\eta(X_1 \ldots X_p)]\\
(d^{(1,0)}\eta)(X_1\ldots X_{p+1})=&\sum_{i}(-1)^{i+1}[\mathrm{hor}(X_i),\eta(X_1\ldots \hat{X}_i\ldots X_{p+1})]\\
&+\sum_{i<j}(-1)^{i+j}\eta([X_i,X_j],X_1\ldots \hat{X}_i\ldots \hat{X}_j\ldots X_{p+1})\\
(d^{(2,-1)}\eta)(X_1\ldots X_{p+2})=&\sum_{\sigma\in \mathrm{Sh}(2,p)}(-1)^{|\sigma|+p}[\mathbb{F}(\mathrm{hor}(X_{\sigma(1)}),\mathrm{hor}(X_{\sigma(2)})),
\eta(X_{\sigma(3)}\ldots X_{\sigma(p+2)})]
\end{align*}
for all $\eta\in \Omega^{p}(Q, \Gamma(\wedge^q \ker d\pi))$, where we used the Schouten bracket, and $\mathrm{Sh}(2,p)$ denotes the set of $(2,p)$-shuffles on $\{1,\ldots p+2\}$. For a proof in the Poisson case see \cite[Proposition 5.3]{CrFe10}, and in the Dirac case \cite[Proposition 4.2.8]{Mar13}. 

By Theorem \ref{theorem:first_page_with_differential} (a) and (b) we obtain that the zeroth-page of the Serre spectral sequence associated to the Lie subalgebroid $(\ker d \pi)^*$ is given by
\[(E_0^{p,q}, d_0)\simeq (\Omega^p(Q,\Gamma(\wedge^q \ker d\pi)),d^{(0,1)}=(-1)^p[\mathcal{W},\cdot]).\]

Inspired by \cite{Vor05}, we will denote the cohomology of the Lie algebroid $(\ker d \pi)^*$ by
\[\mathrm{H}_{V}^{\bullet}(M,\mathcal{W}):=\mathrm{H}^{\bullet}((\ker d\pi)^*).\]
 The subscript is suggestive for the fact that this is the cohomology of the subcomplex of the Lichnerowicz complex of $\mathcal{W}$ consisting of `vertical' multivector fields
\[(\Gamma(\wedge^{\bullet}\ker d\pi), [\mathcal{W},\cdot])\, \subset\, (\Gamma(\wedge^{\bullet}TM), [\mathcal{W},\cdot]).\]
By the heuristic interpretation from Remark \ref{remark:nice:setting}, one can think about elements $c\in \mathrm{H}_{V}^{\bullet}(M,\mathcal{W})$ as smooth families $\{c_x\}_{x\in Q}$ of Poisson cohomology classes $c_x\in \mathrm{H}^{\bullet}(M_x,W_x)$. In degree 0, we obtain the space of Casimir functions of the Poisson structure $\mathcal{W}$
\[\mathrm{H}_{V}^{0}(M,\mathcal{W})=\mathrm{H}^{0}(M,\mathcal{W})=\mathrm{Cas}(M,\mathcal{W})\subset C^{\infty}(M).\]

By Theorem \ref{theorem:first_page_with_differential} (c)-(e), the $C^{\infty}(Q)$-modules $\mathrm{H}_{V}^{\bullet}(M,\mathcal{W})$ have a flat $TQ$-connection $\nabla^H$, 
\[\nabla^H_{X}[V]:=[\mathscr{L}_{\mathrm{hor}(X)}(V)], \quad [V]\in \mathrm{H}_{V}^{\bullet}(M,\mathcal{W}),\]
and we have that the first page of the Serre spectral sequence is given by
\begin{equation}\label{eq:E_1}
(E_1^{p,q}, d_1)\simeq (\Omega^p(Q,\mathrm{H}_{V}^{q}(M,\mathcal{W})),d^{\nabla^H}).
\end{equation}
Of course, $d^{\nabla^H}$ is induced by $d^{(1,0)}$. The differential for $q=0$,
\[d^{\nabla^H}:\Omega^{p}(Q,\mathrm{Cas}(M,\mathcal{W}))\to \Omega^{p+1}(Q,\mathrm{Cas}(M,\mathcal{W})),\]
was used intensively in \cite{Vor01,Vor05}.

For horizontally nondegenerate Poisson structures, a description of Poisson cohomology in degree one was obtained in \cite[Theorem 4.8]{VBVo18}. We recover and extend this result to the Dirac setting in the following theorem, which is a direct consequence of Theorem \ref{theorem:first_page_with_differential}.

\begin{theorem}\label{theorem:Horizontally_nondegenerate_Dirac}
Let $D$ be a horizontally nondegenerate Dirac structure on the total space of a surjective submersion $\pi:M\to Q$, with coupling data $(\mathcal{W},H,\mathbb{F})$. The second page of the Serre spectral sequence associated to the extension \eqref{eq:Vorobjev:submersion_by_liealgebroids} is given by 
\[E_2^{(p,q)}\simeq \mathrm{H}^{p}(Q,\mathrm{H}^{q}_V(M,\mathcal{W})),\]
where the right-hand side is the cohomology of the complex \eqref{eq:E_1}. 

Hence, Dirac cohomology in degree zero is given by $\nabla^H$-flat $\mathcal{W}$-Casimir functions on $M$,
\[\mathrm{H}^0(D)\simeq  E_2^{(0,0)}=\{f\in \mathrm{Cas}(M,\mathcal{W})\, |\, {\nabla}^Hf=0\},\]
and in degree one, we have that 
\[\mathrm{H}^1(D)\simeq E_3^{(0,1)}\oplus E_2^{(1,0)}=\mathrm{ker}\big(d_2:E_2^{(0,1)}\to E_2^{(2,0)}\big) \oplus E_2^{(1,0)},\]
where 
\[
E_2^{(1,0)}\simeq \mathrm{H}^{1}(Q,  \mathrm{Cas}(M,\mathcal{W})),\ \
E_2^{(0,1)}\simeq \mathrm{H}^0(Q,\mathrm{H}^{1}_V(M,\mathcal{W})),\ \ 
E_2^{(2,0)}\simeq \mathrm{H}^2(Q,\mathrm{Cas}(M,\mathcal{W})).\]
\end{theorem}

The differential $d_2$ which appears in the theorem is recovered by the general construction of the spectral sequence. Namely, consider a class $c\in \mathrm{H}^0(Q,\mathrm{H}^{1}_V(M,\mathcal{W}))$. Choose a representative $V\in \Gamma(\ker d\pi)$. Then $[\mathcal{W},V]=0$ and 
\[d^{\nabla^H}[V]=[d^{(1,0)}V]=0\in \Omega^1(Q,\mathrm{H}^{1}_V(M,\mathcal{W})).\]
Hence, by Lemma \ref{lemma:tensor_vb} (d), there is $\theta\in \Omega^1(Q,C^{\infty}(M))$ such that
\[d^{(1,0)}V=-d^{(0,1)}\theta =[\mathcal{W},\theta].\]
Then $d_2$ is given by 
\[d_2c=[d^{(2,-1)}V+d^{(1,0)}\theta]=[-\mathscr{L}_V\mathbb{F}+d^{(1,0)}\theta]\in \mathrm{H}^2(Q,\mathrm{Cas}(M,\mathcal{W})).\]

Finally, we also mention the following direct consequence of Theorem \ref{theorem:page:2:submersion}.

\begin{corollary}
In the setting of Theorem \ref{theorem:Horizontally_nondegenerate_Dirac}, if the 
Ehresmann connection $H$ is complete, then the assignment
\[U\mapsto \mathcal{S}^{q}_{\mathcal{W}}(U):=\big\{ [V]\in \mathrm{H}^{q}_V(\pi^{-1}(U),\mathcal{W})\, |\, \nabla^H [V]=0\big\}\]
is a locally constant sheaf on $Q$, and 
\[E_2^{(p,q)}\simeq \mathrm{H}^{p}(Q,\mathcal{S}^{q}_{\mathcal{W}}).\]
\end{corollary}

\subsubsection{Normal forms around presymplectic leaves}\label{sec:normal_forms_presympl_leaves}

In this subsection, we briefly discuss the linearisation problem around leaves for Poisson and Dirac structures, and how this leads to horizontally nondegenerate structures. We then apply the Serre spectral sequence to calculate the cohomology in low degrees of a class of structures that are partially linearisable, and obtain infinitesimal versions of results in \cite{Bra04,Vor05}.

Horizontally nondegenerate Poisson and Dirac structures arise naturally in the study of normal forms for such structures around leaves \cite{Vor01,Vor05,Wade08}. Namely, let $(M,D)$ be a Dirac manifold and $(Q,\omega_Q)$ be an embedded presymplectic leaf. Denote by $\pi: E:=TM|_Q/TQ\to Q$ the normal bundle of $Q$ and consider a tubular neighbourhood $\iota: E\hookrightarrow M$ of $Q$. Then $\iota^*D$ is a Dirac structure on $E$ which is horizontally nondegenerate around $Q$---because it is so along $Q$---so, by shrinking $\iota$, we may assume that we have a horizontally nondegenerate Dirac structure on $E$, which, for simplicity, we denote again by $D$. Since the zero section is a presymplectic leaf, the associated coupling data $(\mathcal{W},H,\mathbb{F})$ satisfies the following properties.
\begin{enumerate}[(i)]
    \item For each $x\in Q$, the Poisson structure $\mathcal{W}_x$ on $E_x=\pi^{-1}(x)$ vanishes at $x$, and coincides with the transverse Poisson structure to the leaf, constructed in \cite{Wein83} in the Poisson setting and in \cite{DuWa08} in the Dirac setting;
    \item $H$ is tangent to $Q$, i.e.\ $H|_Q=TQ$;
    \item $\mathbb{F}|_{Q}=-\omega_Q$.
\end{enumerate}

The vector bundle structure on $\pi:E\to Q$ can be used to build the \textbf{linearisation} of $D$. This was described in the Poisson case in \cite{Vor01,Vor05}; here we follow \cite[Chapter 4]{Mar13} and \cite{CrMa12}. We construct a path of horizontally non-degenerate Dirac structures on $E$
\[D_t:= t\cdot \mu_t^*(e^{(t-1)\pi^*\omega_Q}D),\qquad t\in (0,1],\]
where $e^{(t-1)\pi^*\omega_Q}D$ denotes the gauge transform of $D$ via the closed 2-form $(t-1)\pi^*\omega_Q\in \Omega^2(E)$, $\mu_t^*$ is the pullback of Dirac structures via the fiberwise multiplication $\mu_t:E\to E$, and finally we use the ``rescaling'' of Dirac structure, defined as $t\cdot (X+\xi)=t\cdot X+\xi$. Clearly $D_1=D$. Each Dirac structure $D_t$ is horizontally nondegenerate, with corresponding coupling data
\[\mathcal{W}_t:=t\cdot \mu_t^*(\mathcal{W}),\qquad H_t:=\mu_t^*(H),\qquad \mathbb{F}_t:=\frac{1}{t}\big(\mu_t^*(\mathbb{F})+(t-1)\pi^*\omega_Q\big).\]
Using these formulas, one easily shows that $D_0:=\lim_{t\to 0} D_t$ exists and is a horizontally nondegenerate Dirac structure. The corresponding coupling data, denoted by
\[\big(\mathcal{W}_{\mathrm{lin}}, H_{\mathrm{lin}},\mathbb{F}_{\mathrm{lin}}-\pi^*\omega_Q\big),\]
satisfies the following:
\begin{enumerate}[(i)]
    \item $\mathcal{W}_{\mathrm{lin}}$ is a family of linear Poisson structures on the fibres of $E$,  i.e.\ for each $x\in Q$,  $(\mathcal{W}_{\mathrm{lin}})_x$ is a linear Poisson structure on the vector space $E_x$, or equivalently, a Lie algebra structure on $E_x^*$, which is precisely the isotropy Lie algebra at $x$ of the original Dirac structure $D$;
    \item $H_{\mathrm{lin}}$ is a linear connection on $E$, hence, in particular, it is complete;
    \item $\mathbb{F}_{\mathrm{lin}}$ is a linear horizontal 2-form, i.e.\ it can be viewed as an element in $\Omega^2(Q,E^*)$.
\end{enumerate}

The Dirac structure $D_0$ plays the role of the \textbf{first order approximation} of $D$ around $Q$. It can be described more intrinsically using the Lie algebroid $D|_Q\Ato Q$ (see \cite[Section 5.2]{Vor01} and \cite[Proposition 4.2.25]{Mar13}). 
Let us note that, even if we start with a Poisson structure, i.e.\ $D=\mathrm{graph}(w)$, the first order approximation $D_0$ will correspond to a Poisson structure only on some neighbourhood $U$ of $Q$, i.e.\ $D_0|_U=\mathrm{graph}(w_{\mathrm{lin}})$.

The \textbf{linearisation problem} asks whether $D$ and $D_0$ are isomorphic around $Q$. For Poisson structures, the isomorphisms are diffeomorphism which fix $Q$ pointwise; in general, one should also allow for exact gauge transformations. For points, i.e.\ $Q=\{x\}$, this is the linearisation problem for Poisson structures around zeroes initiated in \cite{Wein83}. In this setting, Conn's famous theorem \cite{Conn85} says that a Poisson structure $w$ is linearisable around a zero $x\in M$, provided the isotropy Lie algebra $T^*_xM$ is semisimple and compact. For symplectic leaves of Poisson structures, the linearisation problem was studied \cite{Vor01,Bra04,Vor05,CrMa12}. The first approaches were based on Conn's theorem, and produced the following partial linearisation result (see \cite[Corollary 2.5]{Bra04} or \cite[Theorem 4.12]{Vor05}).

\begin{theorem}\label{th:partial:linearisation}
Let $(M,w)$ be a Poisson manifold, and $(Q,\omega_Q)$  be an embedded symplectic leaf such that the isotropy Lie algebra at points $x\in Q$ is semisimple and compact. Then there is a tubular neighbourhood $\iota: E\hookrightarrow M$ of $Q$ in $M$ such that $\iota^*w$ is horizontally nondegenerate around $Q$ with coupling data
\[\big(\mathcal{W}_{\mathrm{lin}}, H_{\mathrm{lin}},\mathbb{F}\big),\]
where $\mathcal{W}_{\mathrm{lin}}$ is fibrewise linear and $H_{\mathrm{lin}}$ is a linear connection on $E$.
\end{theorem}

We apply the techniques of this section to deduce an infinitesimal version of this partial linearisation result. Namely, Theorem \ref{th:partial:linearisation} shows that Poisson structures as in the statement can be deformed only in the direction of the $\mathbb{F}$-component. On the other hand, deformations of Poisson structures are infinitesimally encoded by the second Poisson cohomology group. A consequence of Theorem \ref{theorem:Cohomology:Dirac} below is that the second Poisson cohomology of a Poisson structure $(E,w)$ as in Theorem \ref{th:partial:linearisation} satisfies 
\[\mathrm{H}^2(E,w)\simeq \mathrm{H}^2(Q,\mathrm{Cas}(E,\mathcal{W}_{\mathrm{lin}})).\]

For this, we first recall some classical results. Let $\gg$ be a compact semisimple Lie algebra and denote the linear Poisson structure on its dual by $(\mathfrak{g}^*,w_{\mathrm{lin}})$. The Poisson cohomology of these linear Poisson structures has been computed in \cite[Theorem 3.2]{GiWei92}, and is given by 
\begin{equation}\label{eq:GinzWein}
\mathrm{H}^{\bullet}(\mathfrak{g}^*,w_{\mathrm{lin}})\simeq \mathrm{Cas}(\mathfrak{g}^*,w_{\mathrm{lin}})\otimes \mathrm{H}^{\bullet}(\mathfrak{g}).
\end{equation}

Consider a vector bundle $\pi:E\to Q$ endowed with a vertical, fiberwise linear Poisson structure $\mathcal{W}_{\mathrm{lin}}$. Then the dual bundle $E^*$ is a smooth bundle of Lie algebras. We assume that the bundle is locally trivial with typical fibre a Lie algebra denoted $\gg$. Then Lie algebra cohomology of the fibres yields a vector bundle over $Q$ with typical fibre $\mathrm{H}^{\bullet}(\gg)$, denoted by
\[\mathcal{H}^{\bullet}(E^*)\to Q, \quad 
\mathcal{H}^{\bullet}(E^*)_x
:=\mathrm{H}^{\bullet}(E^*_x)\simeq \mathrm{H}^{\bullet}(\gg).\]

We have the following extension of \cite[Theorem 3.2]{GiWei92}.
\begin{lemma}
If $\gg$ is compact and semisimple, then we have an isomorphism
\begin{equation}\label{eq:decomp:vertical:cohom}
\mathrm{H}_{\mathrm{V}}^{\bullet}(E,\mathcal{W}_{\mathrm{lin}})\simeq \mathrm{Cas}(E,\mathcal{W}_{\mathrm{lin}})\otimes_{C^{\infty}(Q)} \Gamma(\mathcal{H}^{\bullet}(E^*)).
\end{equation}
\end{lemma}
\begin{proof}
For the proof, one uses the homotopy operators constructed in \cite{GiWei92} for the isomorphism \eqref{eq:GinzWein} in a local trivialisation. These, when applied to smooth families of forms, yield smooth families of forms---see the Remark following the proof of \cite[Lemma 3.6]{GiWei92}.    
\end{proof}

Next, we show that these bundles have canonical connections.

\begin{lemma}
If $\gg$ is compact and semisimple, then the $C^{\infty}(Q)$-module $\mathrm{H}^{\bullet}_{\mathrm{V}}(E,\mathcal{W}_{\mathrm{lin}})$ has a canonical flat $TQ$-connection
\[\nabla^{\mathrm{can}}:\Gamma(TQ)\times \mathrm{H}^{\bullet}_{\mathrm{V}}(E,\mathcal{W}_{\mathrm{lin}})\to \mathrm{H}^{\bullet}_{\mathrm{V}}(E,\mathcal{W}_{\mathrm{lin}}).\]
Moreover, this connection preserves the submodules $\mathrm{Cas}(E,\mathcal{W}_{\mathrm{lin}})$ and $\Gamma(\mathcal{H}^{\bullet}(E^*))$, and acts as a derivation with respect to the decomposition \eqref{eq:decomp:vertical:cohom}, i.e.\ it satisfies
\[\nabla^{\mathrm{can}}_X(f\otimes c)=
\nabla^{\mathrm{can}}_X(f)\otimes c+f\otimes \nabla^{\mathrm{can}}_X(c).
\]
\end{lemma}
\begin{proof}
Consider any linear Ehresmann connection $\nabla$ which preserves $\mathcal{W}_{\mathrm{lin}}$ in the sense that \eqref{eq:W is inv} holds for the corresponding linear Ehresmann connection $H_{\mathrm{lin}}$. Then we obtain a map 
\[\nabla:\Gamma(TQ)\times \Gamma
(\wedge^{\bullet}\ker d\pi)\to \Gamma
(\wedge^{\bullet}\ker d\pi), \qquad \nabla_XV:=[\mathrm{hor}(X),V].\]
Since $\mathrm{hor}(X)$ is a Poisson vector field for $\mathcal{W}_{\mathrm{lin}}$, we have an induced map $\nabla^{\mathrm{can}}_X$ in cohomology, which clearly preserve the submodules from the statement and satisfy the derivation rule.

We show that the connection  $\nabla^{\mathrm{can}}_X$ on cohomology is independent of choices. Let $\nabla'$ be a second connection preserving $\mathcal{W}_{\mathrm{lin}}$. Then, for each $X\in \Gamma(TQ)$, we have a vertical Poisson vector field 
\[\theta(X):=[\mathrm{hor}(X)-\mathrm{hor}'(X)]\in \mathrm{H}^1_{\mathrm{V}}(E,\mathcal{W}_{\mathrm{lin}})\simeq \mathrm{Cas}(E,\mathcal{W}_{\mathrm{lin}})\otimes_{C^{\infty}(Q)}\Gamma(\mathcal{H}^1(E^*)).\]
However, since the typical fibre $\gg$ is semisimple, the Whitehead Lemma gives that $\mathcal{H}^1(E^*)=0$. Hence $\theta(X)=0$, and so $\nabla_X$ and $\nabla'_X$ have the same action on $\mathrm{H}^{\bullet}_{\mathrm{V}}(E,\mathcal{W}_{\mathrm{lin}})$. 

Flatness can be checked locally, where one can use the trivial connection $\nabla$ corresponding to a trivialisation of the bundle of Lie algebras.  
\end{proof}

By using Theorems \ref{theorem:first_page_with_differential} and \ref{theorem:page:2:submersion}, and Whitehead's Lemma, we obtain the following. 

\begin{theorem}\label{theorem:Cohomology:Dirac}
Let $D$ be a horizontally nondegenerate Dirac structure on a vector bundle $\pi:E\to Q$ whose coupling data has the first two components linear
\[\big(\mathcal{W}_{\mathrm{lin}}, H_{\mathrm{lin}},\mathbb{F}\big).\]
Assume that the typical fibre of the bundle of Lie algebras $E^*$ corresponding to $\mathcal{W}_{\mathrm{lin}}$ is semisimple and compact. Then the second page of the Serre spectral sequence associated to the extension \eqref{eq:Vorobjev:submersion_by_liealgebroids} is given by 
\[E_2^{(p,q)}\simeq \mathrm{H}^{p}(Q,\mathcal{S}^{q}_{\mathcal{W}_{\mathrm{lin}}}),\]
where $U\mapsto \mathcal{S}^{q}_{\mathcal{W}_{\mathrm{lin}}}(U)$ is the locally constant sheaf
\[\mathcal{S}^{q}_{\mathcal{W}_\mathrm{lin}}(U):=\big\{ f\otimes c\in \mathrm{Cas}(\pi^{-1}(U),\mathcal{W}_{\mathrm{lin}})\otimes_{C^{\infty}(U)}\Gamma(\mathcal{H}^{q}(E^*|_U))\, |\, \nabla^{\mathrm{can}} (f\otimes c)=0\big\}.\]

In particular, for $i=0,1,2$, we have that 
\[\mathrm{H}^{i}(D)\simeq E^{(i,0)}_2\simeq 
\mathrm{H}^{i}(Q,\mathcal{S}^0_{\mathcal{W}_\mathrm{lin}})\simeq \mathrm{H}^{i}(Q,\mathrm{Cas}(E,\mathcal{W}_{\mathrm{lin}})).\]
\end{theorem}

For Poisson structures, versions of this result have appeared in the literature. For instance, in degree $i=1$, the theorem implies \cite[Claim 1.2]{VBVo18}. 
As mentioned above, for $i=2$, Poisson cohomology encodes infinitesimal deformations and the result is an infinitesimal version of \cite[Corollary 2.5]{Bra04} or \cite[Theorem 4.12]{Vor05} (stated above as Theorem \ref{th:partial:linearisation}). Moreover, \cite[Section 4]{Vor05} shows that a relative version of the group $\mathrm{H}^2(Q,\mathrm{Cas}(E,\mathcal{W}_{\mathrm{lin}}))$ encodes the obstructions for the linearisation problem.

\section{Abelian extensions}\label{sec:abelian_extensions}

A special case of \eqref{eq:ss_Liealgebroid_ausgangslage} where the differential on page $E_2$ can be described more explicitly is that of an \textbf{abelian extension}, i.e.\ a short exact sequence 
\begin{equation}\label{eq:ss_abelian_Liealgebroid}
\begin{tikzpicture}[baseline=(current bounding box.center)]
\usetikzlibrary{arrows}
\node (1) at (0,1) {$ 0 $};
\node (2) at (1.5,1) {$ L $};
\node (3) at (3,1) {$ A $};
\node (4) at (4.5,1) {$ B $};
\node (5) at (6,1) {$ 0 $};

\node (w) at (1.5,0) {$ M $};
\node (e) at (3,0) {$ M $};
\node (r) at (4.5,0) {$ Q $};
\draw[-Implies,double equal sign distance]
(2) -- (w);
\draw[-Implies,double equal sign distance]
(3) -- (e);
\draw[-Implies,double equal sign distance]
(4) -- (r);
\path[->]
(1) edge node[]{$  $} (2)
(2) edge node[above]{$ i $} (3)
(3) edge node[above]{$ \Pi $} (4)
(4) edge node[]{$  $} (5)
(w) edge node[above]{$\mathrm{id}_M $} (e)
(e) edge node[above]{$ \pi $} (r);
\end{tikzpicture}
\end{equation}
of Lie algebroids in which $L\Ato M$ has zero Lie bracket. These have been studied in \cite{Mack05} in the case when $M=Q$ and $\pi=\mathrm{id}_M$. By using generalised representations, we extend the results obtained there to the case when $A$ and $B$ are over different bases.

Consider a representation $V\to M$ of $A$ on which $L$ acts trivially.
Then, since $L$ is abelian, we have that $\mathrm{H}^q(L,V)=\Omega^q(L,V)$, and so Theorem \ref{theorem:first_page_with_differential} yields
\[E_2^{p,q}\simeq \mathrm{H}^p(B,\Omega^q(L,V)),\]
where, as explained in Theorem \ref{theorem:first_page_with_differential}, $\Omega^{q}(L,V)$ is a generalised $B$-representation. 

Next, note that there is a generalised representation of $B$ also on the $C^{\infty}(Q)$-module $\Gamma(L)$ given by $ \nabla^L_b =[ a, \cdot  ] $, where $a\in \Gamma(A)$ is any lift of $b\in \Gamma(B)$, i.e.\ satisfies $\Pi\circ a=b\circ \pi$. In the case when $\pi=\mathrm{id}_M$ this is a classical representation; see e.g.\ \cite[Proposition 3.3.20]{Mack05}. 

To identify the differential on the second page, fix a \textbf{Lie algebroid Ehresmann connection} for the extension \eqref{eq:ss_abelian_Liealgebroid}, i.e.\ a splitting $\sigma:M\times_{Q}B\to A$ of $\Pi$. We obtain a $C^{\infty}(Q)$-linear map at the level of sections, also denoted $\sigma:\Gamma(B)\to \Gamma(A)$, such that $\sigma(b)$ is a lift of $b\in \Gamma(B)$. The \textbf{curvature} of $\sigma$ is defined, for $b_1,b_2\in \Gamma(B)$, as
 \begin{equation}\label{eq:extension_class_definition}
 \gamma(b_1,b_2)=[ \sigma (b_1),\sigma (b_2) ]_A-\sigma([b_1,b_2]_B) \in \Gamma(L).
 \end{equation}
The curvature is $C^{\infty}(Q)$-bilinear, so $\gamma\in \Omega^2(B,\Gamma(L))$. The Jacobi identity implies that $\gamma$ is closed for the generalised representation of $B$. The \textbf{extension class} of \eqref{eq:ss_abelian_Liealgebroid} is defined by
 \[[\gamma]\in \mathrm{H}^{2}(B,\Gamma(L)).\] 
Any other splitting is of the form $\sigma'=\sigma+\lambda$, with $\lambda\in \Omega^1(B,\Gamma(L))$, and has corresponding curvature $\gamma'=\gamma+d_B\lambda$. Hence the extension class is independent of the Lie algebroid Ehresmann connection, and $[\gamma]=0$ if and only if the extension admits a flat Lie algebroid Ehresmann connection, in which case $A\simeq \tilde{B}\ltimes L$, where $\tilde{B}:=\mathrm{Im}\,\sigma$.

The extension class determines the differential on the second page of the spectral sequence. More precisely, contraction with the closed 2-form $\gamma\in \Omega^2(B,\Gamma(L))$ induces a cochain map
\[\mathrm{i}_{\gamma}:\big(\Omega^{\bullet}(B,\Omega^q(L,V)), d_B\big)\to \big(\Omega^{\bullet+2}(B,\Omega^{q-1}(L,V)), d_B\big),\]
and therefore a map in cohomology, which depends only on $[\gamma]$. 
We have the following result, which generalises \cite[Theorem 7.4.11]{Mack05} (for $B=TQ$), \cite[Theorem 8]{HoSe53} (for Lie algebras), and 
\cite[Corollary 4.3]{MaZe21} (for the anchor of a regular Lie algebroid).

 \begin{theorem}\label{prop:ses_abelian_extension_page2}
 	Let an abelian extension $L\to A\to B$ as in \eqref{eq:ss_abelian_Liealgebroid} be given, and let $V$ be a representation of $A$ on which $L$ acts trivially. Under the isomorphism of Theorem  \ref{theorem:first_page_with_differential}, the differential on the second page of the Serre spectral sequence becomes
 	\begin{equation}
 	\big(d_2: E_2^{p,q}\to E_2^{p+2,q-1}\big) \simeq \big( (-1)^p\mathrm{i}_{[\gamma]}: \mathrm{H}^{p}(B,\Omega^q(L,V))\to \mathrm{H}^{p+2}(B,\Omega^{q-1}(L,V))\big).
 	\end{equation}
 \end{theorem}
The proof is standard (see e.g.\ the proof of \cite[Proposition 4.2]{MaZe21}) and so we omit it.

\subsection{Vaisman's spectral sequence for regular Poisson manifolds}\label{sec:vaisman}

As an application, we discuss the spectral sequence associated to the bundle of isotropy Lie algebras on a regular Poisson manifold, introduced in \cite{Vai90}.

Recall that a Poisson manifold $(M,w)$ carries a \textbf{singular symplectic foliation}, i.e.\ a decomposition into immersed, connected submanifolds endowed with symplectic structures, 
\[M=\sqcup_S (S,\omega_S),\]
whose members are called \textbf{symplectic leaves}. 
The singular symplectic foliation is determined by the following conditions (see e.g.\ \cite[Proposition 1.8 \& Theorem 4.1]{CFM21})
\[T_{x}S=\mathrm{Im}(\sharp_x:T^*_xM\to T_xM),\qquad \omega_S(\sharp \alpha ,\sharp \beta)=-w(\alpha,\beta),\qquad \textrm{for all }x\in S.\]

We consider \textbf{regular} Poisson manifolds, i.e.\ for which the bivector field $w$ has constant rank. In this case, the symplectic leaves form a smooth (regular) foliation $\mathcal{F}$, endowed with a leafwise symplectic form $\omega\in \Omega^2(\mathcal{F})$. The pair $(\mathcal{F},\omega)$ is called a \textbf{symplectic foliation}. This gives a one-to-one correspondence between regular Poisson manifolds and symplectic foliations. 

For a regular Poisson structure, we have that $\ker \sharp=\nu_{\mathcal{F}}^*$, and this Lie subalgebroid has trivial bracket. So we have an abelian extension
\begin{equation}\label{eq:extension:Poisson}
    0\longrightarrow \nu_{\mathcal{F}}^*
\longrightarrow T^*M
\stackrel{\sharp}{\longrightarrow} T\mathcal{F}
\longrightarrow
0,
\end{equation}
where we have used the notation from Example \ref{example:foliations}, as we will also do in the rest of the section. The Serre spectral sequence for Poisson cohomology of this extension was considered in this generality first in \cite{Vai90}. The corresponding filtration is given by
\[\mathcal{F}^p_{\nu_{\mathcal{F}}^*}\Gamma(\wedge^{p+q}TM)=\Gamma(\wedge^{p}T\mathcal{F}) \wedge \Gamma(\wedge^{q} TM).\]

The following determines the differential on the second page. 
\begin{lemma}
Let $(M,w)$ be a regular Poisson manifold with underlying foliation $\mathcal{F}$ and leafwise symplectic structure $\omega\in \Omega^2(\mathcal{F})$. The extension class corresponding to \eqref{eq:extension:Poisson} is given by
\[\mathrm{var}(\omega):=d_1[\omega]\in \mathrm{H}^2(\mathcal{F},\nu_{\mathcal{F}}^*),\]
where $d_1:\mathrm{H}^2(\mathcal{F})\to \mathrm{H}^2(\mathcal{F},\nu_{\mathcal{F}}^*)$ is the differential on the first page of the spectral sequence corresponding to $T\mathcal{F}\subset TM$ from Example \ref{example:foliations}. Explicitly, if $\tilde{\omega}\in \Omega^2(M)$ is a 2-form extending $\omega$ then the curvature form corresponding to the splitting $\sigma:=\tilde{\omega}^{\flat}:T\mathcal{F}\to T^*M$ is given by
\[\gamma\in \Omega^2(\mathcal{F},\nu_{\mathcal{F}}^*),\qquad \gamma(X_1,X_2):=(d\tilde{\omega})(X_1,X_2,\cdot)\in \nu_{\mathcal{F}}^*.\]
\end{lemma}
\begin{proof}
The definition of $\sharp$ and $\omega$ imply that $\sharp \mathrm{i}_X\tilde{\omega}=X$ for $X\in \Gamma(T\mathcal{F})$, i.e.\ $\sigma=\tilde{\omega}^{\flat}$ is a splitting of \eqref{eq:extension:Poisson}. Using this and standard formulas of Cartan calculus, we obtain the claimed result
\begin{align*}
\gamma(X_1,X_2)&=[\mathrm{i}_{X_1}\tilde{\omega},\mathrm{i}_{X_2}\tilde{\omega}]-\mathrm{i}_{[X_1,X_2]}\tilde{\omega}\\
&=\mathscr{L}_{X_1}\mathrm{i}_{X_2}\tilde{\omega}-\mathrm{i}_{X_2}d \mathrm{i}_{X_1}\tilde{\omega}-\mathrm{i}_{[X_1,X_2]}\tilde{\omega}\\
&=\mathrm{i}_{X_2}\mathrm{i}_{X_1}(d\tilde{\omega}). \qedhere
\end{align*}
\end{proof}

Theorem \ref{prop:ses_abelian_extension_page2} yields the following result.

\begin{theorem}\label{theorem:PC_regular}
Let $(M,w)$ be a regular Poisson manifold with underlying foliation $\mathcal{F}$ and leafwise symplectic structure $\omega\in \Omega^2(\mathcal{F})$. The second page of the Serre spectral sequence corresponding to the extension \eqref{eq:extension:Poisson} is given by
\begin{equation*}
 	\big(E_2^{p,q}, d_2\big) \simeq \big( \mathrm{H}^{p}(\mathcal{F},\wedge^q\nu_{\mathcal{F}}), (-1)^p\mathrm{i}_{\mathrm{var}(\omega)}).
 	\end{equation*}
  In particular, for the first Poisson cohomology group, we have an isomorphism 
\[\mathrm{H}^1(M,w)\simeq \mathrm{H}^1(\mathcal{F})\oplus \ker\big(\mathrm{i}_{\mathrm{var}(\omega)}: \mathrm{H}^0(\nu_{\mathcal{F}})\to \mathrm{H}^2(\mathcal{F}) \big).\]
\end{theorem}

\begin{remark}\rm 
The spectral sequence for Poisson cohomology discussed here was used implicitly in \cite{VoKa88,VoKa89} to calculate Poisson cohomology in low degrees around leaves which admit a product neighbourhood. In the generality of this section, the spectral sequence was introduced in \cite{Vai90}, where a version of Theorem \ref{theorem:PC_regular} was obtained, but without the interpretation of the differential as the cup product with the extension class. 
\end{remark}

\begin{remark}\rm
The extension class $\mathrm{var}(\omega)\in \mathrm{H}^2(\mathcal{F},\nu_{\mathcal{F}}^*)$ plays an important role in the Poisson geometry of the symplectic foliation $(\mathcal{F},\omega)$---see also \cite[Section 3.4]{CFMT19} for an overview.

First, by pulling back the extension class to a symplectic leaf $(S,\omega_S)$, we obtain the class
$i_S^*(\mathrm{var}(\omega))\in H^2(S,\nu_{S}^*)$, which can be interpreted as the \textbf{cohomological variation} of the symplectic forms on the leaves at $S$. For example, if the foliation was trivial around $S$, so that the symplectic foliation is isomorphic to the family $\{(S\times \{x\},\omega_x)\}_{x\in U}$, where $U\subset \mathbb{R}^k$ is a neighbourhood of $0$, then 
$i_S^*(\mathrm{var}(\omega))\in H^2(S)\otimes \mathbb{R}^k$ has components
\[\big[\partial_{x_i}\omega_{x}|_{x=0}\big]\in H^2(S),\quad 1\leq i\leq k.\]
In general, fix $x_0\in S$ and let $\widetilde{S}_{x_0}$ be the universal cover of $S$, viewed as a principal $\pi_1(S,x_0)$-bundle. By pulling back $i_S^*(\mathrm{var}(\omega))$ to $\widetilde{S}$ and using the Bott connection to trivialise the normal bundle over $\widetilde{S}_{x_0}$, we obtain a $\pi_1(S,x_0)$-equivariant linear map
\begin{equation}\label{eq:cohomo_variation}
[\delta_S\omega]:\nu_{S,x_0}\longrightarrow  H^2(\widetilde{S}_{x_0}).
\end{equation}
In \cite{CrMa15} this map is called the \textbf{cohomological variation} of the symplectic form at $S$ and is shown to play an important role in the local structure of the symplectic foliation around $S$. 

Next, by composing the isomorphim $\pi_2(S,x_0)\simeq \pi_2(\widetilde{S}_{x_0},x_0)$ (coming from the long exact sequence in homotopy corresponding to the fibration $\widetilde{S}_{x_0}\to S$) with the Hurewicz map $\pi_2(\widetilde{S}_{x_0},x_0)\to H_2(\widetilde{S}_{x_0},\mathbb{Z})$, we see that the dual of 
\eqref{eq:cohomo_variation} gives a group homomorphism 
\[\partial_{x_0}: \pi_2(S,x_0)\longrightarrow \nu_{S,x_0}^*.\]
This map is the so-called \textbf{monodromy map} at $x_0$, which plays a crucial role in the integrability problem of the Poisson manifold by a symplectic groupoid; namely, integrability is equivalent to the image of all these maps being uniformly discrete in $\nu_{\mathcal{F}}^*$ \cite{CrFe04}. The monodromy map has the geometric interpretation of being the \textbf{variation of symplectic area of spheres} at $x_0$, i.e.
\[\langle\partial_{x_0}[\sigma],v\rangle=\frac{d}{dt}\int_{\mathbb{S}^2}\sigma_t^*\omega\big|_{t=0},\qquad [\sigma]\in \pi_2(S,x_0),\quad  v\in \nu_{S,x_0}\]
where $\sigma_t:\mathbb{S}^2\to M$ is a family of smooth leafwise spheres, with $[\sigma_0]=[\sigma]$ and sending the north pole $N\in \mathbb{S}^2$ to a curve $x_t=\sigma_t(N)$ with direction $v$, i.e.\ $[\frac{d}{dt} x_t|_{t=0}]=v$.
\end{remark}

\begin{remark}\rm 
For a non-regular Poisson manifold $(M,w)$, there are several different ways in which one can generalise the above filtration; at least the following. 

First, consider the filtration by the powers of the ideal of Hamiltonian vector fields
\[\mathcal{F}_{\mathrm{ham}}^{k}:=\wedge^k\mathcal{I}_{\mathrm{ham}},
\qquad \mathcal{I}_{\mathrm{ham}}:=\big\{\sharp d f\wedge \vartheta \,  |\, f\in C^{\infty}(M), \vartheta\in \Gamma(\wedge T M)\big\}.\]
It is easy to see that $\mathcal{I}_{\mathrm{ham}}$ can also be generated by vector fields of the form $\sharp \alpha$, with $\alpha\in \Gamma(T^*M)$. The corresponding spectral sequence is mentioned in \cite{Vai90}. 

Secondly, by considering tangential vector fields, one obtains a potentially larger filtration 
\[\mathcal{F}_{\mathrm{tang}}^{k}:=\wedge^k \mathcal{I}_{\mathrm{tang}}, \qquad \mathcal{I}_{\mathrm{tang}}:=\big\{ X \wedge \vartheta \,  |\, X\in \Gamma(T\mathcal{F}), \vartheta\in \Gamma(\wedge T M)\big\},\]
where $\Gamma(T\mathcal{F})$ denotes the set of vector fields tangent to the symplectic leaves of $w$. The corresponding spectral sequence was used in the proof of \cite[Theorem 4.15]{Gi96}. In general, the inclusion $\mathcal{F}_{\mathrm{ham}}^{k}\subset \mathcal{F}_{\mathrm{tang}}^{k}$ is strict. Indeed, for the Poisson structure $w:=(x^2+y^2) \partial_x\wedge \partial_y$ on $\mathbb{R}^2$, we have that $x\partial_x\in \mathcal{I}_{\mathrm{tang}}\backslash  \mathcal{I}_{\mathrm{ham}}$. 

Next, the conormal bundle of any symplectic leaf $S$ is a Lie subalgebroid $\iota_{\nu^*_S}:\nu^*_S\to T^*M$, and so it gives a differential graded ideal $\mathcal{I}_{\nu_S^*}:=\ker (\iota_{\nu^*_S}^*)\subset \Gamma(\wedge^{\bullet}TM)$. We obtain a filtration
\[ \mathcal{F}^k_{\nu^*}:= \cap_{S} \wedge^k \mathcal{I}_{\nu_S^*},\]
where the intersection is over all leaves. In general, the inclusion $\mathcal{F}_{\mathrm{tang}}^k\subset \mathcal{F}_{\nu^*}^k$ is strict. For example, for the Poisson manifold $(\mathbb{R}^2,w=x\partial_x\wedge \partial_y)$, we have that $w\in \mathcal{F}^2_{\nu^*}\backslash \mathcal{F}^2_{\mathrm{tang}}$.

Next, consider multivector fields that satisfy pointwise the tangential condition
\[\mathcal{F}^k_{\mathrm{pt}}:=\big\{\vartheta\,|\, \vartheta_x\in \wedge^k T_x\mathcal{F}\wedge\ \ldots\ ,\  \textrm{for all}\ x\in M\big\}.\]
This filtration has appeared in \cite[Proposition 5.5]{Vai94} in the following equivalent form
\[\mathcal{F}^k_{\mathrm{pt}}:=\big\{\vartheta \,|\, \vartheta(\alpha_1,\ldots,\alpha_{n})=0,\ \textrm{for all} \ \alpha_1,\ldots \alpha_{n+1-k}\in \ker \sharp_{x},\ x\in M\big\}.\]
We show that $\mathcal{F}^k_{\mathrm{pt}}$ is indeed a filtration by differential graded ideals. For this, we use the Lie subalgebroid $\iota_{T^*M|_S}:T^*M|_S\to T^*M$ corresponding to a symplectic leaf $S$ and the differential graded ideal corresponding to the inclusion $j_{\nu_S^*}:\nu_S^*\hookrightarrow T^*M|_S$, 
\[
\mathcal{J}_{\nu_S^*}:=\ker\big(j_{\nu_S^*}^*:\Gamma(\wedge TM|_S)\to \Gamma(\wedge \nu_S)\big).
\]
Then we have the following description of the filtration
\[\mathcal{F}^k_{\mathrm{pt}}=\cap_S (\iota^*_{T^*M|_S})^{-1}(\wedge^k\mathcal{J}_{\nu_S^*}).\]
Clearly, 
$(\iota^*_{T^*M|_S})^{-1}(\mathcal{J}_{\nu_S^*})=\mathcal{I}_{\nu_S^*}$, but for powers of the ideals, the analogue equality might fail if $S$ is not an embedded leaf. Therefore we expect also the inclusion $\mathcal{F}_{\nu^*}\subset \mathcal{F}_{\mathrm{pt}}$ to be strict. 

We can build another filtration by using the regular part of $M$, i.e.\ the open set $M_{\mathrm{reg}}$ consisting of points in $M$ where the rank of $w$ is locally constant. Namely, define
\[\mathcal{F}^k_{\mathrm{reg}}:=\{\vartheta \,|\, \vartheta|_{M_{\mathrm{reg}}}\in \mathcal{F}^k(M_{\mathrm{reg}})\},\]
where $\mathcal{F}^k(M_{\mathrm{reg}})$ denotes any of the above filtrations for the regular Poisson manifold $M_{\mathrm{reg}}$. It is easy to see that the inclusion $\mathcal{F}^k_{\mathrm{pt}}\subset\mathcal{F}^k_{\mathrm{reg}}$ is in general strict.  
\end{remark}

\appendix

\section{Appendix}

\begin{lemma}\label{lemma:tensor_vb}
Let $V\to Q$ be a vector bundle. We have canonical isomorphisms:
\begin{enumerate}[(a)]
\item For any vector bundle $W\to Q$, \label{App_1}
\[\Gamma(V)\otimes_{C^{\infty}(Q)}\Gamma(W)\simeq \Gamma(V\otimes W);\]
\item For any $C^{\infty}(Q)$-module $\mathfrak{M}$, \label{App_2}
\[\Gamma(V)\otimes_{C^{\infty}(Q)}\mathfrak{M}\simeq \mathrm{Hom}_{C^{\infty}(Q)}\big(\Gamma(V^*),\mathfrak{M});\]
\item For any smooth map $\phi:M\to Q$,\label{App_3}
\[\Gamma(V)\otimes_{C^{\infty}(Q)}C^{\infty}(M)\simeq \Gamma(\phi^*V);\]
\item For any cochain complex $(\mathcal{C}^{\bullet},d)$ of 
$C^{\infty}(Q)$-modules,
\label{App_4}
\[\Gamma(V)\otimes_{C^{\infty}(Q)} \mathrm{H}^{\bullet}(\mathcal{C})\simeq \mathrm{H}^{\bullet}(\Gamma(V)\otimes_{C^{\infty}(Q)} \mathcal{C}).\]
\end{enumerate}
\end{lemma}
\begin{proof}
The statements hold because $\Gamma(V)$ is a finitely generated, projective $C^{\infty}(Q)$-module. More precisely, the statements are obviously satisfied when $V$ is trivial, and in general, they follows because we can write $V$ as a subbundle of a trivial bundle $\mathbb{R}^n_Q\to Q$, with inclusion map $i:V\to \mathbb{R}^n_{Q}$ and projection map $p: \mathbb{R}^n_Q\to V$. 

For the first claim, we need to show that the obvious map
\[\varphi:\Gamma(V)\otimes_{C^{\infty}(Q)}\Gamma(W)\to \Gamma(V\otimes W)\]
is an isomorphism. Clearly, for $V=\mathbb{R}^n_Q$ the map $\varphi$ is an isomorphism. Denote the inverse by
\[\psi:\Gamma(\mathbb{R}^n_Q\otimes W)\to
\Gamma(\mathbb{R}^n_Q)\otimes_{C^{\infty}(Q)}\Gamma(W).\]
Then the inverse of $\varphi$ is given by 
\[\varphi^{-1}=(p_*\otimes \mathrm{id}_{\Gamma(W)})\circ \psi\circ (i\otimes\mathrm{id}_{W})_*.\]
The same scheme can be used to prove the other isomorphisms.  
\end{proof}

\subsubsection*{Acknowledgements}

We would like to thank Lennart Obster and Jo\~ao Mestre for sharing their insights into the conormal VB-algebroid of a subalgebroid and into the construction of exterior powers of representations up to homotopy. I.M.\ would like to thank Pedro Frejlich for useful discussions over many years about the structure of submersions by Lie algebroids and their cohomology. We would also like to thank Marco Zambon for his feedback on the paper and Silvia Sch\"{u}\ss ler for pointing out typos. A.S.\ was financially supported by Methusalem grant METH/21/03 – long term structural funding of the Flemish Government. 



{\small
\def\cprime{$'$} \def\polhk#1{\setbox0=\hbox{#1}{\ooalign{\hidewidth
			\lower1.5ex\hbox{`}\hidewidth\crcr\unhbox0}}} \def\cprime{$'$}
\def\cprime{$'$} \def\cprime{$'$} \def\cprime{$'$} \def\cprime{$'$}
\def\polhk#1{\setbox0=\hbox{#1}{\ooalign{\hidewidth
			\lower1.5ex\hbox{`}\hidewidth\crcr\unhbox0}}} \def\cprime{$'$}
\def\cprime{$'$} \def\cprime{$'$} \def\cprime{$'$} \def\cprime{$'$}
\providecommand{\bysame}{\leavevmode\hbox to3em{\hrulefill}\thinspace}
\providecommand{\MR}{\relax\ifhmode\unskip\space\fi MR }
\providecommand{\MRhref}[2]{%
	\href{http://www.ams.org/mathscinet-getitem?mr=#1}{#2}
}

\bibliographystyle{alpha}

}





\Addresses
\end{document}